\documentclass[a4paper,12pt]{article}
\usepackage{geometry}
\usepackage[T1]{fontenc}
\usepackage{lmodern}
\usepackage{amsmath}
\usepackage{amssymb}
\usepackage{amsthm}
\usepackage{mathrsfs}
\usepackage{graphicx}

\usepackage[abbrev, shortalphabetic]{amsrefs}

\newcommand{\R}{\mathbb{R}}
\newcommand{\Z}{\mathbb{Z}}
\newcommand{\N}{\mathbb{N}}
\newcommand{\K}{\mathcal{K}}

\newcommand{\supp}{\operatorname{supp}}

\newcommand{\sgn}{\operatorname{sgn}}

\newcommand{\interior}{\operatorname{int}}
\newcommand{\conv}{\operatorname{conv}}
\newcommand{\polygon}{\operatorname{polygon}}

\usepackage{bm}
\newcommand{\vr}{\bm{r}}
\newcommand{\vn}{\bm{n}}

\newcommand{\hatK}{\hat{\mathcal{K}}}

\theoremstyle{plain}
\newtheorem{theorem}{Theorem}[section]
\newtheorem{lemma}[theorem]{Lemma}
\newtheorem{remark}[theorem]{Remark}
\newtheorem{proposition}[theorem]{Proposition}
\newtheorem{corollary}[theorem]{Corollary}

\theoremstyle{definition}

\newtheorem*{claim}{Claim}
\newtheorem*{conjecture}{Conjecture}
\newtheorem*{example}{Example}

\newcommand{\abs}[1]{\lvert#1\rvert}

\newcommand{\linearop}{\mathcal{A}}

\usepackage{fancyhdr}
\begin{document}

\title{Symmetric Mahler's conjecture 
for the volume product
in the three dimensional case}
\author{
Hiroshi Iriyeh%
\thanks{College of Science,
Ibaraki University, 2-1-1, Bunkyo, Mito, 310-8512, JAPAN \hfill
\break
e-mail: hiroshi.irie.math@vc.ibaraki.ac.jp
}\,
and
Masataka Shibata%
\thanks{Department of Mathematics, Tokyo Institute of Technology,
2-12-1 Oh-okayama, Meguro-ku, Tokyo 152-8551, JAPAN \hfill\break
e-mail: shibata@math.titech.ac.jp
}
\setcounter{footnote}{-1}
\thanks{The first author was partly supported by the Grant-in-Aid for Science Research (C) (No.~16K05120), JSPS.
The second author was partly supported by the Grant-in-Aid for Science Research (C) (No.~18K03356), JSPS.}
}

%\subjclass[2010]{52A40, 55M25, 52A38.}
%\keywords{Mahler's conjecture}

\maketitle

\begin{abstract}
In this paper, we prove Mahler's conjecture concerning the volume product of
centrally symmetric convex bodies in $\R^n$ in the case where $n=3$.
More precisely, we show that for every three dimensional centrally symmetric
convex body $K \subset \R^3$, the volume product $\abs{K} \, \abs{K^\circ}$
is greater than or equal to $32/3$
with equality if and only if $K$ or $K^\circ$ is a parallelepiped (r1-1).
\end{abstract}

\section{Introduction}

\subsection{Mahler's conjecture for the volume product}

A {\it convex body} in $\R^n$ is a compact convex set in $\R^n$ with nonempty interior. 
Denote by $\K^n$ the set of all convex bodies in $\R^n$.
A convex body $K \in \K^n$ is said to be {\it centrally symmetric} if it satisfies $K=-K$.
We denote by $\K^n_0$ the set of all $K \in \K^n$ which are centrally symmetric.

Let $K \in \K^n$ be a convex body.
The interior of $K$ is denoted by $\interior K$.
For $z \in \interior K$, the {\it polar body} of $K$ with respect to $z$ is defined by
\begin{equation*}
K^z =
\left\{
y \in \R^n;
(y-z) \cdot (x-z) \leq 1 \text{ for any } x \in K
\right\},
\end{equation*}
where $\cdot$ denotes the standard inner product on $\R^n$.
Denote by $|K|$ the $n$-dimensional volume of $K$ in $\R^n$.
Then the {\it volume product} of $K$ is defined by
\begin{equation}
\label{eq:a}
 \mathcal{P}(K) := \min_{z \in \interior K}|K| \, |K^z|.
\end{equation}
Note that this quantity is an affine invariant, i.e., $\mathcal{P}(\linearop K)=\mathcal{P}(K)$ for any invertible affine map $\linearop :\R^n \to \R^n$.
It is well-known that for each $K \in \K^n$
the minimum of \eqref{eq:a} is attained at the unique point $z$ on $K$,
which is called {\it Santal\'o point} of $K$ (see \cite{Sc}).
For a centrally symmetric convex body $K \in \K^n_0$,
the Santal\'o point of $K$ is nothing but the origin $O$.
In the following, the polar of $K$ with respect to $O$ is denoted by $K^\circ$.

The upper bound for $\mathcal{P}(K)$ is well-known as Blaschke-Santal\'o inequality
(see \cite{Sa})
and the bound is attained only for ellipsoids (see \cite{MP}).
In contrast, the sharp lower bound estimate is fairly difficult and it remains as
a longstanding open problem since 1939 as follows.

\begin{conjecture}[Symmetric Mahler's conjecture; \cite{Ma}]
Any $K \in \K_0^n$ satisfies that
\begin{equation}
\label{eq:1}
 \mathcal{P}(K)=\abs{K} \, \abs{K^\circ} \geq \frac{4^n}{n!}.
\end{equation}
\end{conjecture}

This is trivial for $n=1$ and was proved by Mahler himself for $n=2$ \cite{Ma0}.
There are many alternative proofs and the characterization of the equality case
for $n=2$ (see, e.g., \cite{Me}, \cite{Sc}*{Section 10.7}, and references therein).
However, the case where $n \geq 3$ is still widely open.
We solve affirmatively symmetric Mahler's conjecture in the three dimensional case.

\begin{theorem}
\label{thm:main}
For any three dimensional centrally symmetric convex body $K \in \K_0^3$,
we have
\begin{equation*}
\abs{K} \, \abs{K^\circ} \geq \frac{32}{3},
\end{equation*}
with equality if and only if either $K$ or $K^\circ$ is a parallelepiped.
\end{theorem}

Further, we exhibit known partial results about the conjecture.
An asymptotic lower bound of the volume product
was first given by Bourgain and Milman \cite{BM}
and the best known constant is due to Kuperberg \cite{K}.
Meanwhile, the conjecture itself has been proved for very restricted cases;
for zonoids by Reisner \cites{R,R2} (see also \cite{GMR})
and for unconditional bodies by Saint-Raymond \cite{SR} (see \cite{Me2} for a short proof).
Note that the equality of \eqref{eq:1} is attained by an $n$-cube, for instance.
In \cite{NPRZ}, the estimate \eqref{eq:1} was confirmed for $K \in \K_0^n$
sufficiently close to the unit $n$-cube in the Banach-Mazur distance.
For other partial results, see e.g., \cite{BF}, \cite{FMZ}, \cite{LR}, \cite{RSW}.

There is another question when $K \in \K^n$ is not necessarily assumed to be
centrally symmetric.

\begin{conjecture}[Non-symmetric Mahler's conjecture]
Any $K \in \K^n$
satisfies that
\begin{equation*}
\mathcal{P}(K)
=\min_{z \in \interior K} \abs{K} \, \abs{K^z} \geq \frac{(n+1)^{n+1}}{(n!)^2}.
\end{equation*}
\end{conjecture}

This was proved by Mahler for $n=2$ (see \cite{Ma0}).
The alternative proof of Mahler's result by Meyer \cite{Me} is notable.
Indeed, we were able to extract an important idea for attacking to
the three dimensional symmetric case from \cite{Me}.
The idea, which is sketched in Section 1.3, improves the method of \cite{Me2} for unconditional bodies.
Note that non-symmetric Mahler's conjecture remains open for $n \geq 3$, see e.g., \cite{BF}, \cite{FMZ}, \cite{KR}, \cite{RSW}.

\subsection{An application to Viterbo's conjecture}

Recently, a surprising connection between symmetric Mahler's conjecture and
a conjecture in the field of symplectic geometry was discovered.
In 2000, Viterbo posed in \cite{V} an isoperimetric-type conjecture for symplectic capacities of convex bodies in $\R^{2n}$ with the standard symplectic structure $\omega_0$.
A symplectic capacity $c$ is a symplectic invariant which
assigns a non-negative real number to each of
symplectic manifolds of dimension $2n$.
The Hofer-Zehnder capacity $c_{\rm HZ}$ is one of the important symplectic
capacities, which is related to Hamiltonian dynamics on symplectic manifolds.
For details, see \cite{AKO} or a foundational book \cite{HZ}.

\begin{conjecture}[Viterbo \cite{V}]
For any symplectic capacity $c$ and any convex body
$\Sigma \in \K^{2n}$,
\begin{equation}
\label{eq:c}
 \frac{c(\Sigma)}{c(B^{2n})} \leq \left(\frac{\mathrm{vol}(\Sigma)}{\mathrm{vol}(B^{2n})}\right)^{1/n}
\end{equation}
holds, where $B^{2n}$ denotes the $2n$-dimensional unit ball and $\mathrm{vol}(\Sigma)$
denotes the symplectic volume of $\Sigma$.
\end{conjecture}

This conjecture is unsolved even in the case of $n=2$.
Note that $c(B^{2n})=\pi$.
In \cite{AKO}, Artstein-Avidan, Karasev, and Ostrover calculated the Hofer-Zehnder capacity of a convex domain $K \times K^\circ \subset \R^{2n}$ $(K \in \K_0^n)$ as follows.

\begin{theorem}[\cite{AKO}, Theorem 1.7] \label{thm:a}
For any centrally symmetric convex body $K \in \K_0^n$,
we have
\begin{equation*}
c_{\rm HZ}(K \times K^\circ) = 4.
\end{equation*}
\end{theorem}

Using it, they gave the following remarkable observation that
{\it Viterbo's conjecture implies symmetric Mahler's conjecture}
(see \cite{AKO}*{Section 1}).
\begin{equation*}
\frac{4^n}{\pi^n}
= \frac{c_{\rm HZ}(K \times K^\circ)^n}{\pi^n}
\leq \frac{\mathrm{vol}(K \times K^\circ)}{\pi^n/n!}
= \frac{\abs{K} \abs{K^\circ}}{\pi^n/n!}.
\end{equation*}

Conversely, if we assume that symmetric Mahler's conjecture holds,
then Theorem \ref{thm:a} implies the inequality \eqref{eq:c} for the case that
$c=c_{\rm HZ}$ and $\Sigma=K \times K^\circ$.
Therefore, Theorem \ref{thm:main} immediately yields

\begin{corollary}
Let $K \subset \R^3$ be a centrally symmetric convex body.
Then the inequality \eqref{eq:c} holds for $\Sigma=K \times K^\circ \subset (\R^6,\omega_0)$
with respect to the Hofer-Zehnder capacity $c_{\rm HZ}$.
\end{corollary}

Of course, due to the result of \cite{Ma0},
the inequality \eqref{eq:c} also holds for $\Sigma=K \times K^\circ \subset (\R^4,\omega_0)$, where $K \in \K_0^2$,
with respect to $c_{\rm HZ}$.

\subsection{Organization of the paper and our method}

In Section \ref{sec:2}, we review the two dimensional case based on the argument in \cite{Me2} with a slight modification.
Let $K \in \K_0^2$.
Let $B$ (resp.\,$C$) be the intersection of $\partial K$ and the positive part of $x$ (resp.\,$y$) axis.
First, by using coordinate axises, we divide $K$ into four parts $\pm K_{++}, \pm K_{-+}$.
In \cite{Me2}, $K$ is supposed to be an unconditional body, which means that the four parts are all congruent and hence it suffices to consider $K_{++}$ in the first quadrant, and $K^\circ$ is also unconditional.
Then, for any point $P \in K_{++}$, the area $|K_{++}|$ is estimated from below by the area of a quadrilateral $OBPC$, which gives a test point in $K^\circ_{++}$.
The same observation on $K^\circ$ yields another test point in $K_{++}$.
By paring these two points, we easily get $\mathcal{P}(K) \geq 8$.
In the general case, we may assume that $|K_{++}|=|K_{-+}|$.
Although $K^\circ$ is not any longer unconditional, we can find an appropriate division of  $K^\circ$ by using the points $B^\circ, C^\circ \in \partial K^\circ$ with $B^\circ \cdot B=1, C^\circ \cdot C=1$ (see e.g., \cite[Section 4]{BMMR}).
Including the equality case we sketch a proof in Section \ref{sec:2}.

From Section \ref{sec:3}, we begin the proof of the three dimensional case.
We first divide $K$ into eight parts $K_{\pm\pm\pm}$ as in the two dimensional case.
However, compare to the two dimensional case, to make the corresponding decomposition of $K^\circ$ is not so simple.
The first ingredient of our proof of Theorem \ref{thm:main} is to concentrate only on the class of convex bodies $K \in \K_0^3$ which are {\it strongly convex and $C^\infty$ boundary}.
A smoothing approximation procedure (Proposition \ref{prop:13}) guarantees that the problem is reduced to considering only this class.
For a convex body $K$ in the class, we can define a diffeomorphism $\partial K \to \partial K^\circ$, which yields the corresponding eight pieces decomposition $K^\circ_{\pm\pm\pm}$ of $K^\circ$ from that of $K$.
The precise setting is explained in Section \ref{sec:3.2}.

The second ingredient is that only a few test points give a sharp estimate of $\mathcal{P}(K)$, {\it provided $K$ has an additional symmetry}.
When $n=2$, the symmetry we need is $|K_{++}|=|K_{-+}|$, although this is not an essential assumption.
When $n=3$, we have to control not only the volume of $K_{\pm\pm\pm}$ but also the areas of boundary faces of $K_{\pm\pm\pm}$ in the three coordinate planes.
In Section \ref{sec:3.5}, we give a sharp estimate of $\mathcal{P}(K)$ from below under the condition that $K$ has ``good'' symmetries.
The condition is \eqref{eq:3}.
Under this condition, the process of estimating $\mathcal{P}(K)$ is regarded as a direct generalization of the two dimensional case.
A (signed) volume comparison inequality (Lemma \ref{lem:1}) which yields test points on $K$ and $K^\circ$ is given in Section \ref{sec:3.3}.

How can we release $K$ from the very strong condition \eqref{eq:3}?
Note that $\mathcal{P}(K)$ is invariant under linear transformations of $\R^3$.
The third ingredient is to make the most of the freedom of the action of linear transformations.
As a test case, in Appendix \ref{sec:4}, we prove Theorem \ref{thm:main} for $K \in \K^3_0$ equipped with an extra symmetry with respect to a hyperplane through $O$ (Proposition \ref{prop:15}).
In this case, by a linear transformation $\mathcal A$, we can deform $K$ into $\mathcal{A}K$ which satisfies the condition \eqref{eq:3} by means of the intermediate value theorem.
However, for a general $K \in \K^3_0$ to find a linear transformation $\mathcal A$ such that $\mathcal{A}K$ satisfies the condition \eqref{eq:3} is highly nontrivial.

Sections \ref{sec:5} and \ref{sec:6}, which are the technical part of the present paper, are devoted to find an appropriate linear transformation $\mathcal A$ for each $K \in \K^3_0$ which is strongly convex with smooth boundary $\partial K$.
We proceed as follows.
First, we define $K(\theta, \phi, \psi) \in \K^3_0$ as the image of the initial $K$ under the action of $SO(3)$.
(Here, $\theta, \phi, \psi$ mean rotation angles with $x, y, z$-axises, respectively.)
Next, in Section \ref{sec:5.2} we define a linear transform $\mathcal A$ such that $\mathcal{A}K (\theta, \phi, \psi)$ satisfies the condition \eqref{eq:22} (Proposition \ref{prop:3}), which is a partial condition of \eqref{eq:3}.
Finally, to get the full condition \eqref{eq:3} for $\mathcal{A}K (\theta, \phi, \psi)$, we introduce in Section \ref{sec:5.3} three smooth functions $F(K), G(K), H(K)$ on a contractible region $D \subset \mathbb R^3$ with the coordinates $\theta, \phi, \psi$.
These functions are defined by the volume of pieces of the eight part decomposition of $\mathcal{A}K(\theta, \phi, \psi)$ and the resulting two dimensional quantities.
Then the problem to find $\mathcal{A}K$ which satisfies the desired condition \eqref{eq:3} is reduced to the existence of a zero $(\theta, \phi, \psi)$ of the map
$(F,G,H): D \to \mathbb R^3$.
Here, if $(F,G,H)$ has no zero on $\partial D$, then we can define a map
\begin{equation*}
 \mathscr{F}=\frac{(F,G,H)}{\sqrt{F^2+G^2+H^2}}:\partial D \to S^2.
\end{equation*}

From Section \ref{sec:5.4} to \ref{sec:5.8}, we discuss properties of the map $(F,G,H)$ in order to calculate the degree of the map $\mathscr{F}$.
The centrally symmetric convex body $\mathcal{A}K(\theta, \phi, \psi)$ equips with additional symmetries under rotations by specific angles around $x,y,z$-axises.
These induces certain symmetries of the functions $F,G$, and $H$ (Lemmas \ref{lem:5}, \ref{lem:6}, \ref{lem:8}, and \ref{lem:10}).
The results are summarized in Proposition \ref{prop:7} and essential for the degree calculation in the next section.

In Section \ref{sec:6}, we actually prove that $\deg \mathscr{F} \neq 0$ after a suitable perturbation of this map $\mathscr{F}$.
This immediately implies the existence of a zero of the map $(F,G,H)$.
Note that the smoothness of $\partial K$ also has enough merit to the calculation of $\deg \mathscr{F}$.
The construction of the above perturbation of $\mathscr{F}$ is technical, so it is discussed in Appendix \ref{sec:A}.
Anyway the perturbed $\mathscr{F}$ equips with $(\pm1,0,0)$ as regular values.
This fact is crucial for the calculation of $\deg \mathscr{F}$.
Indeed, it enables us to reduce the calculation of $\deg \mathscr{F}$ to the calculation of the winding number of a map $\mathscr{G}: S^1 \to S^1$ (Proposition \ref{prop:5}), which is fairly accessible.
We see that the winding number of $\mathscr{G}$ is odd.
This reduction process is explained in Section \ref{sec:6.2} and the calculation of the winding number of the map $\mathscr{G}$ is carried out in Section \ref{sec:6.3}.
Consequently, the existence of a zero of the map $(F,G,H)$ ensures that $\mathcal{P}(K) \geq 32/3$ for any convex body $K \in \K_0^3$ which is strongly convex with $C^\infty$ boundary.
Finally, a smoothing approximation theorem implies that the inequality holds for all $K \in \K_0^3$. 

In Section \ref{sec:7}, we determine the equality condition  (Propositions \ref{prop:10} and \ref{prop:11}). 
First, for a general $K \in \K^3_0$ with $\mathcal{P}(K)=32/3$, we find a linear transformation $\mathcal{A}$ such that $\mathcal{A}K$ satisfies the condition \eqref{eq:3}.
Indeed, there exists a sequence $\{K_n\}_{n \in \N}$ such that $\lim_{n \rightarrow \infty} K_n=K$, each $K_n \in \K^3_0$ is strongly convex with smooth boundary.
By the results in Sections \ref{sec:5} and \ref{sec:6}, there exists $\{\mathcal{A}_n\}_{n \in \N}$ such that each $\mathcal{A}_n K_n$ satisfies the condition $\eqref{eq:3}$.
As the limit of a subsequence, we get the required $\mathcal{A}$.
Moreover, from $\mathcal{P}(K)=32/3$, we can show that $P_i(K^\circ)=(Q_i(K))^\circ$ and $\abs{Q_i(K)} \abs{P_i(K^\circ)}=8$ where $Q_i$ is the cross section with a coordinate plane and $P_i$ is the projection to the coordinate plane $(i=1,2,3)$.
Thus, the result of the two dimensional case (\cite{Me2}, \cite{R2}) implies that $Q_i(K)$ and $P_i(K^\circ)$ are parallelograms.
Next, we extend the sharp estimate in Section \ref{sec:3.5}.
Then we see $K$ is polyhedron with at most 20 vertices.
Finally, by case analysis and calculating the dual faces of the vertices, we prove that either $K$ or $K^\circ$ is a parallelepiped.

\section{The two dimensional case}
\label{sec:2}

In this section, we sketch the proof of Mahler's theorem for the two dimensional case based on the argument in \cite{Me2} (see also \cite[Section 2]{BF}, \cite[Section 4]{BMMR}), because the equality case of Proposition \ref{prop:9} below will be used in Section 6.
Although \cite{Me2} treats only unconditional bodies, its slight modification yields the following

\begin{theorem}[\cite{Ma0}, \cite{R2}]
\label{prop:12}
Let $K \in \K_0^2$. Then $\mathcal{P}(K) \geq 8$
with equality if and only if $K$ is a parallelogram.
\end{theorem}

To prove it, for $K \in \K_0^2$, we divide $K$ into the following four pieces:
\begin{equation*}\begin{aligned}
K_{++}
&:=
K \cap
\left\{
(x,y) \in \R^2; x \geq 0, y \geq 0
\right\},\quad K_{--}:=-K_{++},
\\
K_{-+}
&:=
K \cap
\left\{
(x,y) \in \R^2; x \leq 0, y \geq 0
\right\},\quad K_{+-}:=-K_{-+}.
\end{aligned}\end{equation*}
Theorem \ref{prop:12} is an easy consequence of the following

\begin{proposition}
\label{prop:9}
Let $K \in \K^2_0$.
Suppose that $\abs{K_{++}}=\abs{K_{-+}}=\abs{K}/4$, $(1,0), (0,1) \in \partial K$.
Then $\mathcal{P}(K) \geq 8$ holds. 
In addition, when $\mathcal{P}(K)=8$, 
there exists a constant $a \in (-1,1]$ such that 
\begin{equation}
\label{eq:38}
\begin{aligned}
K
&=
\conv \left\{
\frac{\pm 1}{1+a^2}(1-a, 1+a), 
\frac{\pm 1}{1+a^2}(-1-a, 1-a)
\right\}, \\
K^\circ
&=
\conv \left\{
\pm(1,a), \pm(-a,1)
\right\}.
\end{aligned}
\end{equation}
Especially, $K$ and $K^\circ$ are squares, $\abs{K}=4/(1+a^2)$, and $\abs{K^\circ}=2(1+a^2)$.
\end{proposition}

\begin{proof}[Proof of Theorem \ref{prop:12}]
Since the volume product $\mathcal{P}(K)$ is invariant under rotations of $K$ around the origin $O=(0,0)$, we may assume that $|K_{++}| = |K_{-+}|$ owing to the intermediate value theorem.
By the assumption that $K$ is centrally symmetric, this means that
\begin{equation}
\label{eq:43}
\abs{K_{++}}=\abs{K_{-+}}=\frac{\abs{K}}{4}.
\end{equation}
The condition \eqref{eq:43} is preserved under scaling by a diagonal matrix.
Therefore, for any $K \in \K^2_0$, there exists a linear transformation $\mathcal{A}$ such that $\mathcal{A}K \in \K^2_0$ satisfies the assumptions of Proposition \ref{prop:9}.
Hence, $\mathcal{P}(K) = \mathcal{P}(\mathcal{A}K) \geq 8$.
If $\mathcal{P}(K)=8$, then $\mathcal{A}K$ is a square.
That is, $K$ is a parallelogram.
\end{proof}

\begin{proof}[Proof of Proposition \ref{prop:9}]
We put $B:=(1,0), C:=(0,1)$.
By the definition of the polar $K^\circ$, we have
$\pm u = \pm B \cdot (u,v) \leq 1, \pm v = \pm C \cdot (u,v) \leq 1 \text{ for any } (u,v) \in K^\circ$.
Hence, $K^\circ \subset [-1,1]\times[-1,1]$.
If $(1,1) \in K^\circ$, then we have $x+y \leq 1$ for any $(x,y) \in K$.
The convexity of $K_{++}$, $K_{-+}$ and their area estimates easily yield
$K=\conv \left\{\pm (1,0), \pm (0,1)\right\}$.
This means \eqref{eq:38} holds with $a=1$.
Similarly, if $(-1,1) \in K^\circ$, then we get the same conclusion.
Thus, since $K$ is centrally symmetric, it is sufficient to consider the case where
\begin{equation}
\label{eq:39}
K^\circ \subset [-1,1]\times[-1,1] \setminus \{\pm (1, 1), \pm (-1,1)\}.
\end{equation}
Since $B, C \in \partial K$, by \eqref{eq:39} and the definition of $K^\circ$, 
there exist points $B^\circ=(1,b), \,C^\circ=(c,1) \in K^\circ$ with $b,c \in (-1,1)$ satisfying that
$B \cdot B^\circ=1$ and $C \cdot C^\circ =1$.
(Note that these points are not necessarily unique in general.)
The points $B^\circ$ and $C^\circ$ enable us to divide $K^\circ$ into the following:
\begin{equation*}\begin{aligned}
K^\circ_{++}
&:=
K^\circ \cap
\left\{
(u,v) \in \R^2; v \geq b u,\ c v \leq u
\right\},
&
K^\circ_{--}
&:=-K^\circ_{++},
\\
K^\circ_{-+}
&:=
K^\circ \cap
\left\{
(u,v) \in \R^2; v \geq b u, c v \geq u
\right\}, 
&
K^\circ_{+-}
&:=-K^\circ_{-+}.
\end{aligned}\end{equation*}

Now we shall find test points to estimate $\mathcal{P}(K)$.
For any $P=(u,v) \in K^\circ$, the sum of the signed area of the triangle $OB^\circ P$ and that of the triangle $OPC^\circ$ is less than or equal to $|K^\circ_{++}|$, that is,
\begin{equation*}
\frac{1}{2}
\begin{vmatrix}
1 & b \\
 u & v
\end{vmatrix}
+
\frac{1}{2}
\begin{vmatrix}
 u & v \\
c & 1
\end{vmatrix}
\leq |K^\circ_{++}|.
\end{equation*}
This inequality with $b, c \in (-1,1)$ implies $S_1:=(1-b,1-c)/(2|K^\circ_{++}|) \in K_{++}$.
A similar argument for the piece $K^\circ_{-+}$ yields
$S_2:=(-1-b,1+c)/(2|K^\circ_{-+}|) \in K_{-+}$.
Repeating the same arguments for $K_{++}$ and $K_{-+}$, we have two test points in $K^\circ$:
$R_1:=2(1,1)/|K|, R_2:=2(-1,1)/|K|$.
Since $R_1 \cdot S_1 \leq 1,\ R_2 \cdot S_2 \leq 1$, that is,
\begin{equation*}
(1,1) \cdot (1-b,1-c) \leq |K| |K^\circ_{++}|, \quad
(-1,1) \cdot (-1-b,1+c) \leq |K| |K^\circ_{-+}|,
\end{equation*}
we obtain
$\mathcal{P}(K) = |K| |K^\circ| = 2|K|(|K^\circ_{++}| + |K^\circ_{-+}|) \geq 8$.

Hereafter, let us determine the equality condition.
Now suppose that $\mathcal{P}(K)=8$.
Since $K \supset \conv \{\pm (1,0), \pm (0,1)\}$, we have $\abs{K} \geq 2$.
If $\abs{K}=2$, then we get $K^\circ=[-1,1]\times[-1,1]$, which contradicts to \eqref{eq:39}.
Thus $\abs{K}>2$ holds.
We consider
\begin{equation*}
\polygon \left\{ O, B, S_1, C \right\} \subset K_{++},\ 
\polygon \left\{ O, C, S_2, -B \right\} \subset K_{-+},
\end{equation*}
where $\polygon \left\{P_1, \dots, P_m\right\}$ denotes the star-shaped polygon in $\R^2$ consists of successive vertices $P_1, \dots, P_m \, (m \geq 3)$ in the counterclockwise order and the edges $P_1 P_2, \dots, P_{m-1} P_m, P_m P_1$.
Taking the volume of these polygons into account, we see that they coincide with $K_{++}, K_{-+}$, respectively, and
\begin{equation}
\label{eq:27}
\abs{K^\circ_{++}}=\frac{2-b-c}{\abs{K}}, \quad
\abs{K^\circ_{-+}}=\frac{2+b+c}{\abs{K}}.
\end{equation}
Consequently, we have
$K= \polygon \left\{B, S_1, C, S_2, -B, -S_1, -C, -S_2 \right\}$.
%A similar argument for $K^\circ$ yields
Similarly,
$K^\circ= \polygon \left\{B^\circ, R_1, C^\circ, R_2, -B^\circ, -R_1, -C^\circ, -R_2 \right\}$.
%\begin{equation}
%\label{eq:42}
%K^\circ=%\tilde{K}^\circ=
%\polygon 
%\left\{
%B^\circ, R_1, C^\circ, R_2, 
%-B^\circ, -R_1, -C^\circ, -R_2
% \right\}.
%\end{equation}

Next, we consider the dual faces of the edges of $K$.
Since the line segment $B S_1$ is a part of an edge of $K$, its dual face is a vertex of $K^\circ$,
which is calculated as
\begin{equation*}
\left(
1, \frac{2\abs{K^\circ_{++}}-1+b}{1-c}
\right).
\end{equation*}
This point must be one of the elements of
$\left\{ \pm B^\circ, \pm R_1, \pm C^\circ, \pm R_2 \right\}$.
Since $c \in (-1,1)$ and $\abs{K}>2$, we have
%\begin{equation*}
%\left(
%1, \frac{2\abs{K^\circ_{++}}-1+b}{1-c}
%\right) 
%=
%B^\circ
%=
%(1,b). 
%\end{equation*}
$\left(1, (2\abs{K^\circ_{++}}-1+b)/(1-c)\right) =B^\circ=(1,b)$.
Hence $\abs{K^\circ_{++}}=(1-bc)/2$.
From the same argument for $CS_2$, we have $\abs{K^\circ_{-+}}=(1-bc)/2$.
Moreover, $c=-b$ holds by \eqref{eq:27}.
Then, we have $\abs{K^\circ_{++}}=\abs{K^\circ_{-+}}=2/\abs{K}=(1+b^2)/2$ and
%\begin{equation*}
%S_1= \frac{1}{1+b^2} \left(1-b, 1+b\right), \quad
%S_2= \frac{1}{1+b^2} \left(-1-b, 1-b\right).
%\end{equation*}
$S_1= (1-b, 1+b)/(1+b^2), S_2= (-1-b, 1-b)/(1+b^2)$.
%We see that three points $S_1$, $(1,0)$, and $-S_2$ are on the line $x + by =1$. 
%Similarly, $S_1$, $(0,1)$, and $S_2$ are on the line $-b x + y =1$.
Consequently, we obtain
\begin{equation*}
 K = \conv \{ \pm S_1, \pm S_2 \} =
\conv \left\{
\pm  
\frac{1}{1+b^2}
\left(1-b, 1+b\right),
\pm 
\frac{1}{1+b^2}
\left(-1-b, 1-b\right)
\right\}
\end{equation*}
with $\abs{K}=4/(1+b^2)$.
Then $K^\circ = \conv \left\{ \pm (1,b), \pm (-b,1) \right\}$ and $\abs{K^\circ}=2(1+b^2)$.
\end{proof}

\section{A sharp estimate for the three dimensional volume product}
\label{sec:3}

\subsection{Preliminaries}
\label{sec:3.1}

Let $K \in \K^n$ with $O \in \interior K$ and denote by $\mu_K$ its Minkowski gauge.
Assume $\partial K$ is a $C^\infty$-hypersurface in $\R^n$.
$K$ is called {\it strongly convex} if a $C^\infty$-function $\mu_K^2(x)/2$ on $\R^n$ is strongly convex, i.e., the Hessian matrix $D^2(\mu_K^2/2)(x)$ is positive definite for each $x$ with $|x|=1$.
Then $\Lambda := \nabla(\mu_K^2/2) = \nabla\mu_K : \partial K \to \partial K^\circ$ is a $C^\infty$-diffeomorphism.
If $K$ is $C^\infty$ strongly convex, then $K^\circ$ is $C^\infty$ strongly convex (see \cite[Section 1.7.2]{Sc}).

\subsection{Convex bodies in $\R^3$}
\label{sec:3.2}

From now on, we focus on centrally symmetric convex bodies in $\R^3$.
Denote by $\hat{\K}$ the set of all centrally symmetric convex bodies $K \in \K^3_0$
which are strongly convex with smooth boundary $\partial K$.
Throughout this paper, for a convex body $K \in$ $\K^3$
we fix the orientation on $K$ induced from the natural orientation of $\R^3$.
From it the orientation of the boundary $\partial K$, that of any domain $S$ on $\partial K$
and that of $\partial S$ (that is, a closed curve on $\partial K$) are induced, respectively.
For a point $P=(p_1,p_2,p_3)$ in $\R^3$ we write $-P=(-p_1,-p_2,-p_3)$.
If $K \in \hat{\K}$ and $P \in K$, then $-P \in K$.
Note that $\rho_K(P)P \in \partial K$ for any $P \in \R^3 \setminus \{ O \}$,
where $\rho_K:=1/\mu_K$ is the radial function of $K$.
For $K \in \hat{\K}$ and any distinct two points $P, Q \in \partial K$ with $P \not= -Q$,
let us introduce an oriented curve from $P$ to $Q$ on the boundary $\partial K$ defined by
\begin{equation*}
\mathcal{C}_K(P,Q)
:=
\left\{
\rho_K((1-t) P + t Q) ((1-t)P + t Q); \ 0 \leq t \leq 1
\right\}.
\end{equation*}
We call it an {\it oriented segment} on $\partial K$.
The polar body of $K \in \hat{\K}$ is given by
\begin{equation*}
K^\circ =
\left\{
(u,v,w) \in \R^3;
(u,v,w) \cdot (x,y,z) \leq 1 \text{ for all } (x,y,z) \in K
\right\},
\end{equation*}
which is also an element of $\hat{\K}$.
Note that $O \in \interior K^\circ$ and $(K^\circ)^\circ=K$.

Let $K \in \hat{\K}$.
First, we decompose the convex body $K$ into eight pieces.
We choose an orthonormal basis $(e_1,e_2,e_3)$ of $\R^3$ and take six points
\begin{equation*}
A_\pm:=\pm \R_+ e_1 \cap \partial K,\ 
B_\pm:=\pm \R_+ e_2 \cap \partial K,\ 
C_\pm:=\pm \R_+ e_3 \cap \partial K
\end{equation*}
on the boundary $\partial K$, and consider six oriented segments on $\partial K$ as follows:
\begin{figure}
\centering
\includegraphics[width=25em]{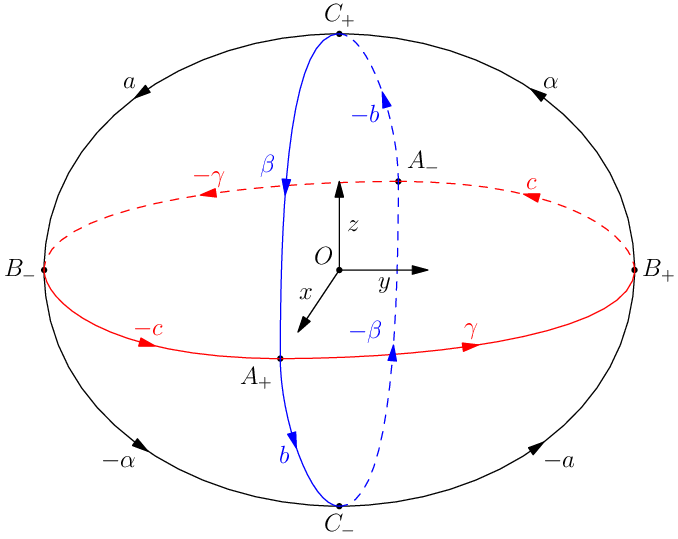}\nobreak
\raisebox{6.5em}{\includegraphics[width=9em]{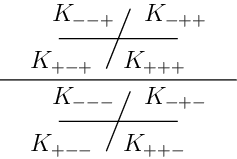}}
\end{figure}
\begin{equation*}\begin{aligned}
\alpha&:=\alpha(K)=\mathcal{C}_{K}(B_+, C_+),
&
a&:=a(K)=\mathcal{C}_{K}(C_+, B_-), \\
\beta&:=\beta(K)=\mathcal{C}_{K}(C_+, A_+),
&
b&:=b(K)=\mathcal{C}_{K}(A_+, C_-), \\
\gamma&:=\gamma(K)=\mathcal{C}_{K}(A_+, B_+),
&
c&:=c(K)=\mathcal{C}_{K}(B_+, A_-).
\end{aligned}\end{equation*}
Then another six segments $-\alpha, -\beta, -\gamma, -a, -b$, and $-c$ on $\partial K$ are automatically defined.
We denote by $Q_i(K)\ (i=1,2,3)$ the sections of $K$ by
the $yz$-plane, $zx$-plane, and $xy$-plane, respectively.
Note that, for instance, in the $yz$-plane $Q_1(K)$ is a strongly convex set whose boundary
consists of the successive arcs $\alpha,a,-\alpha,-a$.
For smooth curves $c(t), d(t) \ (0 \leq t \leq 1)$ in $\R^3$ such that the terminal point of $c$ coincides with the initial point of $d$, denote by $c \cup d$ the piecewise smooth curve which starts at $c(0)$ and ends at $d(1)$.
Then closed curves $\alpha \cup a \cup -\alpha \cup -a$,
$\beta \cup b \cup -\beta \cup -b$, and $\gamma \cup c \cup -\gamma \cup -c$
mutually do not intersect except for $A_\pm,B_\pm,C_\pm$.

For an oriented simple closed curve $c$ on $\partial K$
denote by $\mathcal{S}_K(c) \subset \partial K$
the piece of surface enclosed by $c$
with the orientation compatible with that of $c$.
In this paper, we always equip $\mathcal{S}_K(c)$ with this orientation.
We denote by $O*\mathcal{S}_K(c)$ the truncated cone over $\mathcal{S}_K(c)$:
\begin{equation*}
O*\mathcal{S}_K(c) = \{ \lambda u; u \in \mathcal{S}_K(c) \text{ and } 0 \leq \lambda \leq 1 \}.
\end{equation*}
Using this notation, we divide $K$ into the following eight pieces:
\begin{equation*}\begin{aligned}
K_{+++}&:=O*\mathcal{S}_K(\alpha \cup \beta \cup \gamma), 
&
K_{-++}&:=O*\mathcal{S}_K(\widetilde{\alpha}\, \cup c \cup -b), \\
K_{--+}&:=O*\mathcal{S}_K(\widetilde{a}\, \cup \widetilde{-b}\, \cup -\gamma),
&
K_{+-+}&:=O*\mathcal{S}_K(a \cup -c \cup \widetilde{\beta}), \\
K_{++-}&:=O*\mathcal{S}_K(-a \cup \, \widetilde{\gamma}\, \cup b),
&
K_{-+-}&:=O*\mathcal{S}_K(\widetilde{-a}\, \cup -\beta \cup \widetilde{c}), \\
K_{---}&:=O*\mathcal{S}_K(\widetilde{-\alpha}\, \cup \widetilde{-\gamma}\,\cup \widetilde{-\beta}),
&
K_{+--}&:=O*\mathcal{S}_K(-\alpha \cup \, \widetilde{b}\, \cup \widetilde{-c}),
\end{aligned}\end{equation*}
where for a curve $c \subset \R^3$, denote by $\widetilde{c}$ the curve which has the same image as $c$ with the reverse orientation, that is, $\widetilde{c}(t) = c(1-t) \ (0 \leq t \leq 1)$.

Next, we decompose the polar body $K^\circ$ into eight pieces associated to the above decomposition of $K$.
Using the map $\Lambda: \partial K \to \partial K^\circ$ defined in Section \ref{sec:3.1}, we put
six points and six curves on
$\partial K^\circ$ as follows:
\begin{equation*}\begin{aligned}
&A^\circ_\pm := \Lambda(A_\pm), \quad
B^\circ_\pm := \Lambda(B_\pm), \quad
C^\circ_\pm := \Lambda(C_\pm); \\
&\alpha^\circ:= \Lambda(\alpha), \quad
a^\circ:= \Lambda(a), \quad
\beta^\circ:= \Lambda(\beta), \quad
b^\circ:= \Lambda(b), \quad
\gamma^\circ:= \Lambda(\gamma), \quad
c^\circ:= \Lambda(c).
\end{aligned}\end{equation*}
From the definition of the polar body and the strong convexity of $K^\circ$,
the point $A_+^\circ$ (resp.\ $B_+^\circ, C_+^\circ$) is the unique point with the maximal $u$-coordinate (resp.\ $v$-coordinate, $w$-coordinate) among all the points in $K^\circ$.

Since $\Lambda: \partial K \to \partial K^\circ$ is a $C^{\infty}$-diffeomorphism,
smooth closed curves
$\alpha^\circ \cup a^\circ \cup -\alpha^\circ \cup -a^\circ$,
$\beta^\circ \cup b^\circ \cup -\beta^\circ \cup -b^\circ$, and
$\gamma^\circ \cup c^\circ \cup -\gamma^\circ \cup -c^\circ$
on $\partial K^\circ$ are mutually do not intersect except for
$A^\circ_\pm, B^\circ_\pm, C^\circ_\pm$.
Thus we can define a decomposition of $K^\circ$ into the following eight pieces
\begin{equation*}\begin{aligned}
K^\circ_{+++}&:=O*\mathcal{S}_{K^\circ}(\alpha^\circ \cup \beta^\circ \cup \gamma^\circ), 
&
K^\circ_{-++}&:=O*\mathcal{S}_{K^\circ}(\widetilde{\alpha^\circ} \cup c^\circ \cup -b^\circ), \\
K^\circ_{--+}&:=O*\mathcal{S}_{K^\circ}(\widetilde{a^\circ} \cup \widetilde{-b^\circ} \cup -\gamma^\circ), 
&
K^\circ_{+-+}&:=O*\mathcal{S}_{K^\circ}(a^\circ \cup -c^\circ \cup \widetilde{\beta^\circ}), \\
K^\circ_{++-}&:=O*\mathcal{S}_{K^\circ}(-a^\circ \cup \widetilde{\gamma^\circ} \cup b^\circ), 
&
K^\circ_{-+-}&:=O*\mathcal{S}_{K^\circ}(\widetilde{-a^\circ} \cup -\beta^\circ \cup \widetilde{c^\circ}), \\
K^\circ_{---}&:=O*\mathcal{S}_{K^\circ}(\widetilde{-\alpha^\circ} \cup \widetilde{-\gamma^\circ} \cup \widetilde{-\beta^\circ}),
&
K^\circ_{+--}&:=O*\mathcal{S}_{K^\circ}(-\alpha^\circ \cup \widetilde{b^\circ} \cup \widetilde{-c^\circ}),
\end{aligned}\end{equation*}
which corresponds to the decomposition of $K$ described before.

\begin{figure}
\centering
\includegraphics{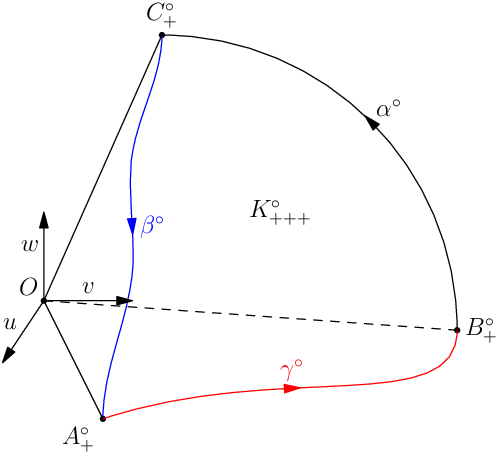}
\end{figure}
Denote by $P_i \ (i=1,2,3)$ the orthogonal projections to the $vw$-plane, $wu$-plane
and $uv$-plane, respectively.
Since $K^\circ \subset \R^3$ is strongly convex,
the projections $P_i(K^\circ) \ (i=1,2,3)$ of $K^\circ$ are compact strongly convex domains in each plane.
\begin{claim}
The boundaries $\partial(P_i(K^\circ)) \ (i=1,2,3)$ coincide with smooth simple closed curves
$P_1(\alpha^\circ \cup a^\circ \cup -\alpha^\circ \cup -a^\circ)$, $P_2(\beta^\circ \cup b^\circ \cup -\beta^\circ \cup -b^\circ)$, and $P_3(\gamma^\circ \cup c^\circ \cup -\gamma^\circ \cup -c^\circ)$, respectively.
\end{claim}

\begin{proof}
Fix a point $Q \in \partial(Q_1(K))$.
Since $Q$ lies on the set
\begin{equation*}
Q_1(\alpha \cup a \cup (-\alpha) \cup (-a)) 
=
\alpha \cup a \cup (-\alpha) \cup (-a),
\end{equation*}
$Q$ is on one of the four segments in the right-hand side.
If $Q \in \alpha$, then there exists $t \in [0,1]$ such that $Q=\alpha(t)$.
By the definitions of $\alpha(t)$ and $\alpha^\circ(t)$,
we have
$\alpha(t)=Q_1(\alpha(t))=(0,\alpha_2(t),\alpha_3(t))$,
$\alpha^\circ(t)=(\alpha^\circ_1(t),\alpha^\circ_2(t),\alpha^\circ_3(t))$, 
$P_1(\alpha^\circ(t))=(0,\alpha^\circ_2(t),\alpha^\circ_3(t))$,
and
\begin{equation*}
1= \alpha(t) \cdot \alpha^\circ(t) = \alpha_2(t) \alpha^\circ_2(t) + \alpha_3(t) \alpha^\circ_3(t)
 = Q_1(\alpha(t)) \cdot P_1(\alpha^\circ(t)).
\end{equation*}
Combining it with the conditions $\alpha(t)=Q_1(\alpha(t)) \in \partial (Q_1(K))$ and $P_1(\alpha^\circ(t)) \in P_1(K^\circ)$,
we find that $P_1(\alpha^\circ(t)) \in \partial P_1(K^\circ)$.
Considering other cases similarly, we have
\begin{equation*}
P_1(\alpha^\circ \cup a^\circ \cup (-\alpha^\circ) \cup (-a^\circ)) \subset \partial(P_1(K^\circ)).
\end{equation*}
On the other hand, for $P \in \partial(P_1(K^\circ))$,
there exists a point $Q$ in
$\partial(Q_1(K))=Q_1(\alpha \cup a \cup (-\alpha) \cup (-a))$
such that $Q \cdot P = 1$.
The point $Q$ is represented, for instance, by $\alpha(t)$ for some $t \in [0,1)$.
By the definition,
$\tilde{P}:=\alpha^\circ(t)=\Lambda(\alpha(t))$ is the unique point on $\partial K$ such that $Q \cdot \tilde{P} =1$.
Moreover, by the strict convexity of $K^\circ$, we have $P_1(\tilde{P})=P$.
Hence,
\begin{equation*}
\partial (P_1(K^\circ)) \subset P_1(\alpha^\circ \cup a^\circ \cup (-\alpha^\circ) \cup (-a^\circ))
\end{equation*}
also holds.
\end{proof}

\subsection{A volume estimate of a piece of $K^\circ$ from below}
\label{sec:3.3}

The next task is to estimate $|K^\circ_{+++}|,|K^\circ_{-++}|,\ldots,|K^\circ_{+--}|$ from below.
Consider
$K^\circ_{+++}=O*\mathcal{S}_{K^\circ}(\alpha^\circ \cup \beta^\circ \cup \gamma^\circ)$,
for instance.
The boundary $\partial K^\circ_{+++}$ is the union of the cone
\begin{equation*}
O*(\alpha^\circ \cup \beta^\circ \cup \gamma^\circ)
\end{equation*}
and $K^\circ_{+++} \cap \partial K^\circ$.
For any point $P=(u,v,w) \in \R^3$, the {\it signed volume} of the cone
\begin{equation*}
P*(O*(\alpha^\circ \cup \beta^\circ \cup \gamma^\circ))
= \{ (1-\lambda)P+\lambda \xi \in \R^3;
 \xi \in O*(\alpha^\circ \cup \beta^\circ \cup \gamma^\circ), 0 \leq \lambda \leq 1 \}
\end{equation*}
is defined by
\begin{equation}
\label{eq:b}
 \frac{1}{6} \int_0^1 
\left(
\begin{vmatrix}
 P \\ \alpha^\circ(t) \\ {\alpha^\circ}'(t)
\end{vmatrix}
+
\begin{vmatrix}
 P \\ \beta^\circ(t) \\ {\beta^\circ}'(t)
\end{vmatrix}
+
\begin{vmatrix}
 P \\ \gamma^\circ(t) \\ {\gamma^\circ}'(t)
\end{vmatrix}
\right)\,dt,
\end{equation}
where $\alpha^\circ(t) = (\alpha^\circ_1(t), \alpha^\circ_2(t), \alpha^\circ_3(t))$ and
$|\cdot|$ denotes the determinant for a square matrix.
Here we put
\begin{equation*}\begin{aligned}
\overline{\alpha^\circ_1}:=
\int_0^1
\begin{vmatrix}
 \alpha^\circ_2(t) & \alpha^\circ_3(t) \\
 {\alpha^\circ_2}'(t) & {\alpha_3^\circ}'(t)
\end{vmatrix}\,dt,
\overline{\alpha^\circ_2}:=
\int_0^1
\begin{vmatrix}
 \alpha^\circ_3(t) & \alpha^\circ_1(t) \\
 {\alpha^\circ_3}'(t) & {\alpha^\circ_1}'(t)
\end{vmatrix}\,dt,
\overline{\alpha^\circ_3}:=
\int_0^1
\begin{vmatrix}
 \alpha^\circ_1(t) & \alpha^\circ_2(t) \\
 {\alpha^\circ_1}'(t) & {\alpha^\circ_2}'(t)
\end{vmatrix}\,dt,
\end{aligned}\end{equation*}
then we have
\begin{equation*}
\int_0^1
\begin{vmatrix}
 P \\ \alpha^\circ(t) \\ {\alpha^\circ}'(t)
\end{vmatrix}\,dt
=
  (\overline{\alpha^\circ_1}, \overline{\alpha^\circ_2}, \overline{\alpha^\circ_3}) \cdot (u,v,w).
\end{equation*}
For simplicity, we put
$\overline{\alpha^\circ} := (\overline{\alpha^\circ_1}, \overline{\alpha^\circ_2}, \overline{\alpha^\circ_3})$.
Similarly, we define
$\overline{a^\circ}, \overline{\beta^\circ}, \overline{b^\circ}, \overline{\gamma^\circ}$,
and $\overline{c^\circ}$.
Then the signed volume (\ref{eq:b}) of the cone
$P*(O*(\alpha^\circ \cup \beta^\circ \cup \gamma^\circ))$
is represented as
$(\overline{\alpha^\circ}+\overline{\beta^\circ}+\overline{\gamma^\circ}) \cdot (u,v,w)/6$.

If the cone $K^\circ_{+++}$ is convex in $\R^3$ and $P \in K^\circ_{+++}$,
then the quantity (\ref{eq:b}) coincides with the volume
$|P*(O*(\alpha^\circ \cup \beta^\circ \cup \gamma^\circ))|$ and it is less than $|K^\circ_{+++}|$,
which is the fact actually used in \cite{Me2} and \cite{BF}.
In case $K^\circ_{+++} \subset \R^3$ is {\it not} convex
or $P \in K^\circ$ is {\it not} in $K^\circ_{+++}$,
the comparison of the signed volume (\ref{eq:b}) with $|K^\circ_{+++}|$
is a nontrivial problem.
However, under the condition that $P$ lies in the convex set $K^\circ$
the following holds, which is a key lemma to estimate $\mathcal{P}(K)$.

\begin{lemma}
\label{lem:1}
Let $K^\circ \in \hatK$.
For any point $P=(u,v,w) \in K^\circ$, we have
\begin{equation*}
(\overline{\alpha^\circ}+\overline{\beta^\circ}+\overline{\gamma^\circ}) \cdot (u,v,w)
\leq
6 |K^\circ_{+++}|.
\end{equation*}
By the definition of polar, we obtain
\begin{equation*}
 \frac{1}{6 |K^\circ_{+++}|}\left(
\overline{\alpha^\circ}+\overline{\beta^\circ}+\overline{\gamma^\circ}
\right) \in K.
\end{equation*}
For the other seven parts similar estimates hold (see the table below).
\end{lemma}

Lemma \ref{lem:1} is a direct consequence of the following

\begin{proposition}
\label{prop:1}
Let $K \in \hatK$.
Let $c$ be a piecewise smooth, oriented, simple closed curve on $\partial K$.
Let $\mathcal{S}_K(c) \subset \partial K$ be a piece of surface enclosed by
the curve $c$.
Then for any point $P \in K$, the inequality
\begin{equation*}
\frac{1}{6} \int_c (P \times \vr) \cdot d \vr \leq |O*\mathcal{S}_K(c)|
\end{equation*} 
holds. 
\end{proposition}

\begin{proof}
Let $S$ be the boundary of the truncated cone $O*\mathcal{S}_K(c)$ and
$S_0:= S \setminus \mathcal{S}_K(c)$.
Then $S$ is a piecewise smooth surface.
Let $\vr$ be the position vector on $S$ and $\vn$ the outward unit normal vector of $S$ at $\vr$.

Since the vector $\vr$ on $S_0$ is perpendicular to the normal vector $\vn$ at the point,
we have $\vr \cdot \vn=0$ on $S_0$.
Hence, by the divergence theorem, $\int_{\mathcal{S}_K(c)}\vr \cdot \vn\,dA
=\int_{O*\mathcal{S}_K(c)}\mathrm{div}\,\vr\,dxdydz=3|O*\mathcal{S}_K(c)|$
holds, where $dA$ denotes the area element.

On the other hand, on $\mathcal{S}_K(c)$ we have $P \cdot \vn \leq \vr \cdot \vn$ for any
$P \in K$, because the support function $h_K(\vn):=\max_{P \in K}P \cdot \vn$ attains the
maximum at $\vr$.
Integrating on $\mathcal{S}_K(c)$, we have
\begin{equation*}
\int_{\mathcal{S}_K(c)}P \cdot \vn\,dA
\leq \int_{\mathcal{S}_K(c)}\vr \cdot \vn\,dA=3|O*\mathcal{S}_K(c)|.
\end{equation*}
Since the rotation of $P \times \vr$ is calculated as $2P$, by Stokes' formula, the left-hand side is equal to $(1/2)\int_c (P \times \vr) \cdot d \vr$, which yields the result.
\end{proof}

To discuss the other seven parts of $K^\circ$ we need further preparations.
For the curve $\widetilde{\alpha^\circ}$ on $\partial K^\circ$ defined by
$\widetilde{\alpha^\circ}(t) = \alpha^\circ(1-t) \ (0 \leq t \leq 1)$,
we have
\begin{equation*}\begin{aligned}
 \overline{(\widetilde{\alpha^\circ})} \cdot (u,v,w)
=&
\int_0^1
\begin{vmatrix}
 P \\ \alpha^\circ(1-t) \\ (\alpha^\circ(1-t))'
\end{vmatrix}\,dt
=
-
\int_0^1
\begin{vmatrix}
 P \\ \alpha^\circ(1-t) \\ {\alpha^\circ}'(1-t)
\end{vmatrix}\,dt 
=
- \overline{\alpha^\circ} \cdot (u,v,w)
\end{aligned}\end{equation*}
for any $P=(u,v,w) \in \R^3$.
Moreover, since
\begin{equation*}
 \overline{(-\alpha^\circ)} \cdot (u,v,w)
=
\int_0^1
\begin{vmatrix}
 P \\ -\alpha^\circ(t) \\ -{\alpha^\circ}'(t)
\end{vmatrix}\,dt
=
\overline{\alpha^\circ} \cdot (u,v,w),
\end{equation*}
we obtain
\begin{equation*}
\overline{(\widetilde{\alpha^\circ})} = -\overline{\alpha^\circ}, \quad
\overline{(-\alpha^\circ)} = \overline{\alpha^\circ}, \quad
\overline{(-\widetilde{\alpha^\circ})} =
\overline{(\widetilde{-\alpha^\circ})} =
-\overline{\alpha^\circ}.
\end{equation*}

We can do the same observation for $K^\circ_{-++}$, $K^\circ_{--+}$, and $K^\circ_{+-+}$.
Since $K^\circ$ is centrally symmetric, the volume of $K^\circ_{++-}$, $K^\circ_{-+-}$,
$K^\circ_{---}$, and $K^\circ_{+--}$ equals that of $K^\circ_{--+}$, $K^\circ_{+-+}$,
$K^\circ_{+++}$, and $K^\circ_{-++}$, respectively.

Summarizing the above arguments, we obtain
\begin{center}
\begin{tabular}{|l|l|} \hline
 & point contained in $K$ \\ \hline
$K^\circ_{+++}=O*\mathcal{S}_{K^\circ}(\alpha^\circ \cup \beta^\circ \cup \gamma^\circ)$ &
$S_1:=\frac{1}{6| K^\circ_{+++}|}\left(\overline{\alpha^\circ}+\overline{\beta^\circ}+\overline{\gamma^\circ}\right)$ \\ \hline
$K^\circ_{-++}=O*\mathcal{S}_{K^\circ}(\widetilde{\alpha^\circ} \cup c^\circ \cup -b^\circ)$ &
$S_2:=\frac{1}{6|K^\circ_{-++}|}\left(-\overline{\alpha^\circ}+\overline{b^\circ}+\overline{c^\circ}\right)$ \\ \hline
$K^\circ_{--+}=O*\mathcal{S}_{K^\circ}(\widetilde{a^\circ} \cup \widetilde{-b^\circ} \cup -\gamma^\circ)$ &
$S_3:=\frac{1}{6|K^\circ_{--+}|}\left(-\overline{a^\circ}-\overline{b^\circ}+\overline{\gamma^\circ}\right)$ \\ \hline
$K^\circ_{+-+}=O*\mathcal{S}_{K^\circ}(a^\circ \cup -c^\circ \cup \widetilde{\beta^\circ})$ &
$S_4:=\frac{1}{6|K^\circ_{+-+}|}\left(\overline{a^\circ}-\overline{\beta^\circ}+\overline{c^\circ}\right)$ \\ \hline
\end{tabular}
\end{center}

Furthermore, similar arguments for the convex body $K$ give
the following results.
Note that the calculation is easier than the case of the polar $K^\circ$,
because the eight parts of $K$ are convex.
\begin{center}
\begin{tabular}{|l|l|} \hline
& point contained in $K^\circ$ \\ \hline
$K_{+++}=O*\mathcal{S}_K(\alpha \cup \beta \cup \gamma)$ &
$R_1:=\frac{1}{6|K_{+++}|}\left(\bar{\alpha}+\bar{\beta}+\bar{\gamma}\right)$ \\ \hline
$K_{-++}=O*\mathcal{S}_K(\widetilde{\alpha} \cup c \cup -b)$ &
$R_2:=\frac{1}{6|K_{-++}|}\left(-\bar{\alpha}+\bar{b}+\bar{c}\right)$ \\ \hline
$K_{--+}=O*\mathcal{S}_K(\widetilde{a} \cup \widetilde{-b}\, \cup -\gamma)$ &
$R_3:=\frac{1}{6|K_{--+}|}\left(-\bar{a}-\bar{b}+\bar{\gamma}\right)$ \\ \hline
$K_{+-+}=O*\mathcal{S}_K(a \cup -c \cup \widetilde{\beta})$ &
$R_4:=\frac{1}{6|K_{+-+}|}\left(\bar{a}-\bar{\beta}+\bar{c}\right)$ \\ \hline
\end{tabular}
\end{center}

\smallskip

In Section \ref{sec:3.5}, these eight points $R_i, S_i \ (i=1,2,3,4)$ will be effectively used
as test points to estimate $\mathcal{P}(K)$ from below.

\subsection{A sharp estimate}
\label{sec:3.5}

In this subsection, we give a sufficient condition of deducing inequality \eqref{eq:1}
(see the condition \eqref{eq:3} below).
In the following sections, for any $K \in \hat{\K}$
we carefully select a linear transformation $\mathcal{A}$
such that $\mathcal{A}K \in \hat{\K}$ satisfies this condition.

From the arguments in Section \ref{sec:3.3}, we know that $R_i \in K^\circ$ and $S_i \in K$ $(i=1,2,3,4)$.
By the definition of polar, we obtain $R_i \cdot S_i \leq 1$.
In other words,
\begin{equation}
\label{eq:2}
\begin{aligned}
\left(\bar{\alpha}+\bar{\beta}+\bar{\gamma}\right)\cdot
\left(\overline{\alpha^\circ}+\overline{\beta^\circ}+\overline{\gamma^\circ}\right)
& \leq 36 |K_{+++}| |K^\circ_{+++}|, \\
\left(-\bar{\alpha}+\bar{b}+\bar{c}\right)\cdot
\left(-\overline{\alpha^\circ}+\overline{b^\circ}+\overline{c^\circ}\right)
& \leq 36 |K_{-++}| |K^\circ_{-++}|, \\
\left(-\bar{a}-\bar{b}+\bar{\gamma}\right)\cdot
\left(-\overline{a^\circ}-\overline{b^\circ}+\overline{\gamma^\circ}\right)
& \leq 36 |K_{--+}| |K^\circ_{--+}|, \\
\left(\bar{a}-\bar{\beta}+\bar{c}\right)\cdot
\left(\overline{a^\circ}-\overline{\beta^\circ}+\overline{c^\circ}\right)
& \leq 36 |K_{+-+}| |K^\circ_{+-+}|.
\end{aligned}
\end{equation}
Here, assume that
\begin{equation}
\label{eq:3}
 |K_{+++}|=|K_{-++}|=|K_{--+}|=|K_{+-+}|, \quad
\overline{\alpha}=
\overline{a}, \quad
\overline{\beta}=
\overline{b}, \quad
\overline{\gamma}=
\overline{c}.
\end{equation}
We now check that this condition implies symmetric Mahler's conjecture
for the three dimensional case.
Indeed, the condition \eqref{eq:3} yields
the volume of each eight part equals $|K|/8$, and hence equations
\eqref{eq:2} and \eqref{eq:3} implies that
\begin{equation*}\begin{aligned}
\left(\bar{\alpha}+\bar{\beta}+\bar{\gamma}\right)\cdot
\left(\overline{\alpha^\circ}+\overline{\beta^\circ}+\overline{\gamma^\circ}\right)
& \leq \frac{9}{2} \abs{K} \abs{K^\circ_{+++}}, \\
\left(-\bar{\alpha}+\bar{\beta}+\bar{\gamma}\right)\cdot
\left(-\overline{\alpha^\circ}+\overline{b^\circ}+\overline{c^\circ}\right)
& \leq \frac{9}{2} \abs{K} \abs{K^\circ_{-++}}, \\
\left(-\bar{\alpha}-\bar{\beta}+\bar{\gamma}\right)\cdot
\left(-\overline{a^\circ}-\overline{b^\circ}+\overline{\gamma^\circ}\right)
& \leq \frac{9}{2} \abs{K} \abs{K^\circ_{--+}}, \\
\left(\bar{\alpha}-\bar{\beta}+\bar{\gamma}\right)\cdot
\left(\overline{a^\circ}-\overline{\beta^\circ}+\overline{c^\circ}\right)
& \leq \frac{9}{2} \abs{K} \abs{K^\circ_{+-+}}.
\end{aligned}\end{equation*}
It then follows from these inequalities that
\begin{equation*}\begin{aligned}
&\frac{9}{4} |K| |K^\circ|
=
\frac{9}{2} \abs{K} (|K^\circ_{+++}|+|K^\circ_{-++}|+|K^\circ_{--+}|+|K^\circ_{+-+}|) \\
& \geq 
\overline{\alpha} \cdot
\left(
\left(\overline{\alpha^\circ}+\overline{\beta^\circ}+\overline{\gamma^\circ}\right)
-\left(-\overline{\alpha^\circ}+\overline{b^\circ}+\overline{c^\circ}\right)
-\left(-\overline{a^\circ}-\overline{b^\circ}+\overline{\gamma^\circ}\right)
+\left(\overline{a^\circ}-\overline{\beta^\circ}+\overline{c^\circ}\right)
\right) \\
&\quad +
\overline{\beta} \cdot
\left(
\left(\overline{\alpha^\circ}+\overline{\beta^\circ}+\overline{\gamma^\circ}\right)
+\left(-\overline{\alpha^\circ}+\overline{b^\circ}+\overline{c^\circ}\right)
-\left(-\overline{a^\circ}-\overline{b^\circ}+\overline{\gamma^\circ}\right)
-\left(\overline{a^\circ}-\overline{\beta^\circ}+\overline{c^\circ}\right)
\right) \\
&\quad +
\overline{\gamma} \cdot
\left(
\left(\overline{\alpha^\circ}+\overline{\beta^\circ}+\overline{\gamma^\circ}\right)
+\left(-\overline{\alpha^\circ}+\overline{b^\circ}+\overline{c^\circ}\right)
+\left(-\overline{a^\circ}-\overline{b^\circ}+\overline{\gamma^\circ}\right)
+\left(\overline{a^\circ}-\overline{\beta^\circ}+\overline{c^\circ}\right)
\right) \\
&=
2 \overline{\alpha} \cdot \left(\overline{\alpha^\circ} + \overline{a^\circ}\right)
+
2 \overline{\beta} \cdot \left(\overline{\beta^\circ} + \overline{b^\circ}\right)
+
2 \overline{\gamma} \cdot \left(\overline{\gamma^\circ} + \overline{c^\circ}\right).
\end{aligned}\end{equation*}

Let us examine the first term in the last line.
Recalling the definition of $\overline{\alpha^\circ}=(\overline{\alpha^\circ_1}, \overline{\alpha^\circ_2}, \overline{\alpha^\circ_3})$, its $u$-component $\overline{\alpha^\circ_1}$ is twice the area of the convex region $O*P_1(\alpha^\circ)$ in the $vw$-plane.
Similarly, $\overline{a^\circ_1}$ is twice $|O*P_1(a^\circ)|$.
Since $K^\circ \in \hatK$, the convex domain $P_1(K^\circ)$ in the $vw$-plane
is also centrally symmetric, and hence
\begin{equation*}
\overline{\alpha^\circ_1}+\overline{a^\circ_1} = 2(|O*P_1(\alpha^\circ)|+|O*P_1(a^\circ)|)
= |P_1(K^\circ)|.
\end{equation*}
Again, since the curve $\alpha \subset K$ lies in the $yz$-plane,
we have $\overline{\alpha} = (\overline{\alpha_1},0,0)$, where
\begin{equation*}
\overline{\alpha_1}
=
\int_0^1
\begin{vmatrix}
 \alpha_2(t) &
 \alpha_3(t) \\
 \alpha_2'(t) &
 \alpha_3'(t) 
\end{vmatrix}
\,dt.
\end{equation*}
Here $\overline{\alpha_1}$ is nothing but twice the area of the convex region
$O*\alpha$ in the $yz$-plane.
From the assumption that $\overline{\alpha}=\overline{a}$ in the condition \eqref{eq:3},
we have $\overline{\alpha_1}=|Q_1(K)|/2$.
Therefore, the first term
$2 \overline{\alpha} \cdot \left(\overline{\alpha^\circ} + \overline{a^\circ}\right)$
is equal to
$|Q_1(K)|\,|P_1(K^\circ)|$.
Applying the same argument to the second and the third terms, we have
\begin{equation*}
\overline{\beta^\circ_2}+\overline{b^\circ_2} = |P_2(K^\circ)|, \quad
\overline{\gamma^\circ_3}+\overline{c^\circ_3} = |P_3(K^\circ)|
\end{equation*}
and
\begin{equation*}
2 \overline{\beta} = (0,|Q_2(K)|,0), \quad
2 \overline{\gamma} = (0,0,|Q_3(K)|).
\end{equation*}

Consequently, we obtain
\begin{equation*}
\frac{9}{4} |K| |K^\circ| 
\geq
|Q_1(K)|\, |P_1(K^\circ)|+
|Q_2(K)|\, |P_2(K^\circ)|+
|Q_3(K)|\, |P_3(K^\circ)|.
\end{equation*}
It is well-known that convex domains $Q_i(K)$ and $P_i(K^\circ)$ are polar bodies of each other
(e.g. \cite{GMR}*{p.\,274}).
By Mahler's theorem \cite{Ma0} (see also \cite{R}, \cite{GMR}, \cite{MR}, and Section \ref{sec:2}),
we have
$|Q_i(K)|\, |P_i(K^\circ)| \geq 4^2/2! =8 \ (i=1,2,3)$. Thus
\begin{equation*}
|K| \, |K^\circ| \geq \frac{4}{9} \times 8 \times 3 = \frac{32}{3} =\frac{4^3}{3!}
\end{equation*}
holds.
Hence we obtain

\begin{proposition}
\label{prop:2}
If $K \in \hatK$ satisfies the condition \eqref{eq:3}, then $\mathcal{P}(K) \geq 32/3$.
\end{proposition}

\section{The general case}
\label{sec:5}

The purpose of this and the next sections is to find a suitable linear transformation $\mathcal A$ for any centrally symmetric convex body $K$ in $\hat{\mathcal K}$ such that ${\mathcal A}K$ satisfies the condition \eqref{eq:3}.
In other words, we will find a basis of $\R^3$ such that the hypothesis \eqref{eq:3} is satisfied.

\subsection{Linear transforms}
\label{sec:5.2}

In order to prove \eqref{eq:1} for general $K \in \K^3_0$,
we first consider some linear transform of $K \in \hat{\mathcal K}$.
For this purpose, we introduce some notations.
For $\theta, \phi, \psi \in \R$, 
$X(\theta)$,
$Y(\phi)$,
$Z(\psi)$ denote rotations in $\R^3$ as follows:
\begin{equation*}\begin{aligned}
 X(\theta)
&=
\begin{pmatrix}
 1 & 0 & 0 \\
0 & \cos \theta & -\sin \theta \\
0 & \sin \theta & \cos \theta
\end{pmatrix}, \quad
Y(\phi)=
\begin{pmatrix}
\cos \phi & 0 & \sin \phi \\
0 & 1 & 0 \\
-\sin \phi & 0 & \cos \phi
\end{pmatrix}, \\
Z(\psi)
&=
\begin{pmatrix}
\cos \psi & -\sin \psi & 0 \\
\sin \psi & \cos \psi & 0\\
0 & 0 & 1
\end{pmatrix}.
\end{aligned}\end{equation*}
Hereafter, we put $K(\theta, \phi, \psi):=X(\theta)Y(\phi)Z(\psi)K$.

In this section, we choose a linear transform $\mathcal{A}=\mathcal{A}(K)$
(depends on $K$) which satisfies
\begin{equation}
\label{eq:22}
\begin{aligned}
|O*\beta(\mathcal{A}K)| = |O*b(\mathcal{A}K)|, 
&\quad
|O*\gamma(\mathcal{A}K)| = |O*c(\mathcal{A}K)|, \\
|(\mathcal{A}K)_{+++}|+|(\mathcal{A}K)_{-++}| 
&= 
|(\mathcal{A}K)_{--+}|+|(\mathcal{A}K)_{+-+}|,
\end{aligned}
\end{equation}
where $\mathcal{A}K=\mathcal{A}(K)K$.
We use the following spherical coordinates:
\begin{equation*}
 P(\xi, \eta) =(\cos \xi, \sin \xi \cos \eta, \sin \xi \sin \eta) \in S^2.
\end{equation*}
First, we choose the angle $\Theta(K) \in (0,\pi)$ which satisfies that
\begin{equation}
\label{eq:21}
\int_0^{\Theta(K)} d \eta \int_0^{\pi} \rho_{K}^3(P(\xi,\eta)) \sin \xi \,d \xi
=
\int_{\Theta(K)}^\pi d \eta \int_0^{\pi} \rho_{K}^3(P(\xi,\eta)) \sin \xi \,d \xi.
\end{equation}
Since the left-hand side is increasing and the right-hand side is decreasing with respect to $\Theta(K)$,
we see that $\Theta(K)$ is uniquely determined.
Actually, we have the following
\begin{lemma}
\label{lem:12}
For each $K \in \hatK$, there exists the unique $\Theta(K) \in (0,\pi)$
satisfying the relation \eqref{eq:21}.
Moreover, $\Theta(\theta, \phi, \psi):=\Theta(K(\theta,\phi,\psi))$ is smooth with respect to $(\theta, \phi, \psi) \in \R^3$.
\end{lemma}

\begin{proof}
We put
\begin{equation*}\begin{aligned}
I(\Theta,\theta,\phi,\psi)
&:=
\int_0^\Theta d\eta \int_0^\pi \rho_{K(\theta,\phi,\psi)}^3(P(\xi,\eta)) \sin \xi \,d\xi \\
&\quad
-
\int_\Theta^\pi d\eta \int_0^\pi \rho_{K(\theta,\phi,\psi)}^3(P(\xi,\eta)) \sin \xi \,d\xi.
\end{aligned}\end{equation*}
Since $K \in \hatK$, by the definition of $K(\theta, \phi, \psi)$, we see that $I$ is smooth in $\R^4$.
On the other hand, since $\rho_{K}(P(\xi,\eta))>0$, we have 
\begin{equation*}\begin{aligned}
I(0,\theta,\phi,\psi)
&=
-
\int_0^\pi d\eta \int_0^\pi \rho_{K(\theta,\phi,\psi)}^3(P(\xi,\eta)) \sin \xi \,d\xi <0, \\
I(\pi,\theta,\phi,\psi)
&=
\int_0^\pi d\eta \int_0^\pi \rho_{K(\theta,\phi,\psi)}^3(P(\xi,\eta)) \sin \xi \,d\xi >0, \\
\frac{\partial I}{\partial \Theta}(\Theta,\theta,\phi,\psi)
&=
2 \int_0^\pi \rho_{K(\theta,\phi,\psi)}^3(P(\xi,\Theta)) \sin \xi \,d\xi >0.
\end{aligned}\end{equation*}
By the intermediate value theorem, there exists the unique $\Theta(\theta,\phi,\psi)=\Theta(K(\theta,\phi,\psi))$ such that
$I(\Theta(\theta,\phi,\psi),\theta,\phi,\psi)=0$. Moreover, by the implicit function theorem, we see that $\Theta(\theta,\phi,\psi)$ is smooth in $\R^3$.
\end{proof}
Next, we introduce the quantities $\Phi(K) \in (0,\pi)$ and $\Psi(K) \in (0,\pi)$ defined by
\begin{equation*}\begin{aligned}
 \int_0^{\Phi(K)} \rho_{K}^2(P(\xi,0)) \,d\xi
&=
 \int_{\Phi(K)}^\pi \rho_{K}^2(P(\xi,0)) \,d\xi, \\
 \int_0^{\Psi(K)} \rho_{K}^2(P(\xi,\Theta(K))) \,d\xi 
&=
 \int_{\Psi(K)}^\pi \rho_{K}^2(P(\xi,\Theta(K))) \,d\xi.
\end{aligned}\end{equation*}
Similarly as $\Theta(K)$, the above $\Phi(K)$ and $\Psi(K)$ are uniquely determined for each $K \in \hatK$, and 
$\Phi(\theta,\phi,\psi) :=\Phi(K(\theta,\phi,\psi))$ and
$\Psi(\theta,\phi,\psi) :=\Psi(K(\theta,\phi,\psi))$ are smooth functions
on $\R^3$.
Now, we define $\mathcal{A}=\mathcal{A}(K)$ as
\begin{equation*}\begin{aligned}
\linearop(K)&:=
\begin{pmatrix}
1 & \frac{1}{\tan (\Phi(K))} & \frac{1}{\sin(\Theta(K)) \tan (\Psi(K))} \\
0 & 1 & \frac{1}{\tan(\Theta(K))} \\
0 & 0 & 1
\end{pmatrix}^{-1}.
\end{aligned}\end{equation*}

\begin{proposition}
\label{prop:3}
For every $K \in \hatK$, the condition \eqref{eq:22} holds
for $\mathcal{A}K$ which is the image of $K$ by the above linear transform.
\end{proposition}

\begin{proof}
Denote by $\chi_A$ the characteristic function:
\begin{equation*}
\chi_A(P)=
\begin{cases}
 1 & \text{ if } P \in A, \\
 0 & \text{ if } P \not\in A.
\end{cases} 
\end{equation*}
Then we have
\begin{equation*}\begin{aligned}
& |(\linearop K)_{+++}|+|(\linearop K)_{-++}|
-|(\linearop K)_{--+}|-|(\linearop K)_{+-+}| \\
&=
\int_{\{\tilde{y} > 0, \tilde{z} > 0\}} \chi_{\linearop K}(\tilde{P}) \,d\tilde{x}d\tilde{y}d\tilde{z}
-
\int_{\{\tilde{y} < 0, \tilde{z} > 0\}} \chi_{\linearop K}(\tilde{P}) \,d\tilde{x}d\tilde{y}d\tilde{z} \\
&=
\int_{\{\tilde{y} > 0, \tilde{z} > 0\}} \chi_{K}(\linearop^{-1} \tilde{P}) \,d\tilde{x}d\tilde{y}d\tilde{z}
-
\int_{\{\tilde{y} < 0, \tilde{z} > 0\}} \chi_{K}(\linearop^{-1} \tilde{P}) \,d\tilde{x}d\tilde{y}d\tilde{z} \\
&=
\int_{\{
y -z/\tan(\Theta(K)) > 0, z > 0\}} \chi_{K}(P) \,dxdydz
-
\int_{\{y -z/\tan(\Theta(K)) < 0, z > 0\}} \chi_{K}(P) \,dxdydz,
\end{aligned}\end{equation*} 
where we used the substitution $P=\linearop^{-1} \tilde{P}$,
that is 
\begin{equation*}
\begin{pmatrix}
\tilde{x} \\ \tilde{y} \\ \tilde{z}
\end{pmatrix}
=
\linearop P=
\begin{pmatrix}
x -\frac{y}{\tan (\Phi(K))} -\frac{z}{\sin(\Theta(K)) \tan (\Psi(K))}+\frac{z}{\tan (\Phi(K)) \tan(\Theta(K))} \\
y -\frac{z}{\tan(\Theta(K))} \\
z
\end{pmatrix}.
\end{equation*}
By using the polar coordinates $x =r \cos \xi$, $y = r \sin \xi \cos \eta$, $z=r \sin \xi \sin \eta$, we have
\begin{equation*}\begin{aligned}
&y -
\frac{z}{\tan(\Theta(K))}
=
r \sin \xi \cos \eta - \frac{r \sin \xi \sin \eta}{\tan(\Theta(K))} = z\left(
\frac{1}{\tan \eta} - \frac{1}{\tan(\Theta(K))}\right).
\end{aligned}\end{equation*}
Under the condition $r > 0$ and $\sin \xi > 0$, 
$z>0$ if and only if $0 < \eta < \pi$. 
Thus we have
\begin{equation*}\begin{aligned}
&y -
\frac{z}{\tan(\Theta(K))}>0, z>0
\text{ if and only if }
0 < \eta < \Theta(K), \\
&y-\frac{z}{\tan(\Theta(K))}<0, z>0
\text{ if and only if } 
\Theta(K)<\eta<\pi.
\end{aligned}\end{equation*}
Therefore, by Lemma \ref{lem:12}, we have
\begin{equation*}\begin{aligned}
& |(\linearop K)_{+++}|+|(\linearop K)_{-++}|
-|(\linearop K)_{--+}|-|(\linearop K)_{+-+}| \\
&=
\frac{1}{3}
\int_0^{\Theta(K)} d \eta \int_0^{\pi} \rho_{K}^3(P(\xi,\eta)) \sin \xi \,d \xi
-
\frac{1}{3}
\int_{\Theta(K)}^\pi d \eta \int_0^{\pi} \rho_{K}^3(P(\xi,\eta)) \sin \xi \,d \xi=0.
\end{aligned}\end{equation*}

Next, since
\begin{equation*}\begin{aligned}
|O*\beta(\linearop K)|
&=
\int_{\{\tilde{x}>0, \tilde{y}=0, \tilde{z}>0\}} \chi_{\linearop K}(\tilde{P}) \,d\tilde{x}d\tilde{z}, \\
|O*b(\linearop K)|
&=
\int_{\{\tilde{x}>0, \tilde{y}=0, \tilde{z}<0\}} \chi_{\linearop K}(\tilde{P}) \,d\tilde{x}d\tilde{z}
=
\int_{\{\tilde{x}<0, \tilde{y}=0, \tilde{z}>0\}} \chi_{\linearop K}(\tilde{P}) \,d\tilde{x}d\tilde{z},
\end{aligned}\end{equation*}
by the substitution $P=\linearop^{-1} \tilde{P}$, we have
\begin{equation*}\begin{aligned}
& 
|O*\beta(\linearop K)|
-
|O*b(\linearop K)|
\\
&=
\int_{\{\tilde{x}>0, \tilde{y}=0, \tilde{z}>0\}} \chi_{K}(\linearop^{-1} \tilde{P}) \,d\tilde{x}d\tilde{z}
-
\int_{\{\tilde{x}<0, \tilde{y}=0, \tilde{z}>0\}} \chi_{K}(\linearop^{-1} \tilde{P}) \,d\tilde{x}d\tilde{z} \\
&=
 \int_{\{x - z/(\sin(\Theta(K)) \tan(\Psi(K)))>0, y=z/\tan(\Theta(K)), z>0\}} \chi_{K}(P) \,dxdz \\
&
-\int_{\{x - z/(\sin(\Theta(K)) \tan(\Psi(K)))<0, y=z/\tan(\Theta(K)), z>0\}} \chi_{K}(P) \,dxdz.
\end{aligned}\end{equation*}
By using the above polar coordinates, 
if $y=z/\tan(\Theta(K))$ and $z>0$, then $\eta=\Theta(K) \in (0, \pi)$.
Since
\begin{equation*}
 x - \frac{z}{\sin(\Theta(K)) \tan(\Psi(K))} = r \sin \xi \left(
\frac{1}{\tan \xi} - \frac{1}{\tan (\Psi(K))}
\right),
\end{equation*}
we see that
$x - z/(\sin(\Theta(K)) \tan(\Psi(K)))>0$ if and only if $0 < \xi < \Psi(K)$, and
$x - z/(\sin(\Theta(K)) \tan(\Psi(K)))<0$ if and only if $\Psi(K)< \xi < \pi$.
Thus, by the definition of $\Psi(K)$, we obtain
\begin{equation*}\begin{aligned}
&
|O*\beta(\linearop K)|
-
|O*b(\linearop K)|
\\
&=
\sin (\Theta(K))
\left(
\int_{\{r>0, 0 < \xi < \Psi(K)\}} \chi_{K}(P(\xi, \Theta(K))) r \,drd\xi 
\right. \\
& \hspace{9em}\left.
-
\int_{\{r>0, \Psi(K)< \xi < \pi\}} \chi_{K}(P(\xi, \Theta(K))) r \,drd\xi
\right)
 \\
&=
\sin (\Theta(K))
\left(
\frac{1}{2}\int_{0}^{\Psi(K)} \rho_{K}^2(P(\xi, \Theta(K))) \,d\xi
-
\frac{1}{2}\int_{\Psi(K)}^{\pi} \rho_{K}^2(P(\xi, \Theta(K))) \,d\xi
\right)=0.
\end{aligned}\end{equation*}

Finally, we show $|O*\gamma(\linearop K)|=|O*c(\linearop K)|$ similarly.
Since
\begin{equation*}\begin{aligned}
|O*\gamma(\linearop K)|
&=
\int_{\{\tilde{x}>0, \tilde{y}>0, \tilde{z}=0\}} \chi_{\linearop K}(\tilde{P}) \,d\tilde{x}d\tilde{y}, \\
|O*c(\linearop K)|
&=
\int_{\{\tilde{x}<0, \tilde{y}>0, \tilde{z}=0\}} \chi_{\linearop K}(\tilde{P}) \,d\tilde{x}d\tilde{y}, 
\end{aligned}\end{equation*}
by the substitution $P=\linearop^{-1} \tilde{P}$, we have
\begin{equation*}\begin{aligned}
&
|O*\gamma(\linearop K)|
-
|O*c(\linearop K)|
\\
&=
\int_{\{\tilde{x}>0, \tilde{y}>0, \tilde{z}=0\}} \chi_{K}(\linearop^{-1} \tilde{P}) \,d\tilde{x}d\tilde{y}
-
\int_{\{\tilde{x}<0, \tilde{y}>0, \tilde{z}=0\}} \chi_{K}(\linearop^{-1} \tilde{P}) \,d\tilde{x}d\tilde{y} \\
&=
 \int_{\{x - y/\tan(\Phi(K))>0, y>0, z=0\}} \chi_{K}(P) \,dxdy 
-\int_{\{x - y/\tan(\Phi(K))<0, y>0, z=0\}} \chi_{K}(P) \,dxdy.
\end{aligned}\end{equation*}
By using the same polar coordinates, 
for the case $y>0$, $z=0$, we have $\eta=0$ and
\begin{equation*}
x - \frac{y}{\tan(\Phi(K))} = r \sin \xi \left(
\frac{1}{\tan \xi} - \frac{1}{\tan (\Phi(K))}
\right).
\end{equation*}
Hence, $x - y/\tan(\Phi(K))>0$ if and only if $0 < \xi < \Phi(K)$,
and $x - y/\tan(\Phi(K))<0$ if and only if $\Phi(K)< \xi < \pi$.
Thus, by the definition of $\Phi(K)$, we obtain
\begin{equation*}\begin{aligned}
& 
|O*\gamma(\linearop K)|
-
|O*c(\linearop K)| \\
&=
\int_{\{r>0, 0 < \xi < \Phi(K)\}} \chi_{K}(P(\xi,0)) r \,drd\xi
-
\int_{\{r>0, \Phi(K)< \xi < \pi\}} \chi_{K}(P(\xi,0)) r \,drd\xi
 \\
&=
\frac{1}{2}\int_{0}^{\Phi(K)} \rho_{K}^2(P(\xi,0)) \,d\xi
-
\frac{1}{2}\int_{\Phi(K)}^{\pi} \rho_{K}^2(P(\xi,0)) \,d\xi
=0.
\end{aligned}\end{equation*}
\end{proof}

\subsection{Definition of a smooth map $(F,G,H)$}
\label{sec:5.3}

For simplicity, we denote
$\linearop K:=\mathcal{A}(K) K$
and
\begin{equation*}
\linearop K(\theta, \phi, \psi):=
\mathcal{A}(K(\theta, \phi, \psi))K(\theta, \phi, \psi)
=
\mathcal{A}(X(\theta)Y(\phi)Z(\psi)K)X(\theta)Y(\phi)Z(\psi)K.
\end{equation*}
For every $K \in \hatK$,
we define real numbers $F, G, H$ by
\begin{equation*}\begin{aligned}
F(K) &:= |O*\alpha(\linearop K)|- |O*a(\linearop K)|, \\
G(K) &:= |(\linearop K)_{+++}|+ |(\linearop K)_{--+}| - |(\linearop K)_{-++}| - |(\linearop K)_{+-+}|, \\
H(K) &:= |(\linearop K)_{+++}|+ |(\linearop K)_{+-+}| - |(\linearop K)_{-++}| - |(\linearop K)_{--+}|,
\end{aligned}\end{equation*}
and introduce the following functions on $\R^3$:
\begin{equation*}\begin{aligned}
F(\theta, \phi, \psi)&:=
F({\linearop K}(\theta, \phi, \psi)), \\
G(\theta, \phi, \psi)&:=
G({\linearop K}(\theta, \phi, \psi)), \\
H(\theta, \phi, \psi)&:=
H({\linearop K}(\theta, \phi, \psi)).
\end{aligned}\end{equation*}
If we can find a zero $(\theta, \phi, \psi)$ of $(F,G,H)$,
then ${\linearop K}(\theta, \phi, \psi)$ satisfies the condition \eqref{eq:3} by Proposition \ref{prop:3}.
To show the existence of such a zero, 
we consider the following region in $\R^3$:
\begin{equation*}
 D = \left\{(\theta, \phi, \psi) \in \R^3; 0 \leq \phi \leq \pi, 0 \leq \psi \leq \pi, 0 \leq \theta \leq \pi - \Theta(0,\phi,\psi) \right\}.
\end{equation*}
Note that $\Theta$ is a smooth function introduced in the previous subsection and $0 < \Theta(0,\phi,\psi) < \pi$.
We divide the boundary $\partial D$ into the following six parts:
\begin{equation*}\begin{aligned}
M_1 &:=
\left\{
(0, \phi, \psi) \in \R^3; 
0 \leq \phi \leq \pi, 0 \leq \psi \leq \pi\right\}, \\
M_2 &:=
\left\{
(\pi-\Theta(0,\phi,\psi), \phi, \psi) \in \R^3; 
0 \leq \phi \leq \pi, 0 \leq \psi \leq \pi\right\}, \\
M_3 &:=
\left\{(\theta, 0, \psi) \in \R^3; 0 \leq \psi \leq \pi, 0 \leq \theta \leq \pi-\Theta(0,0,\psi) \right\}, \\ 
M_4 &:=
\left\{(\theta, \pi, \psi) \in \R^3; 0 \leq \psi \leq \pi, 0 \leq \theta \leq \pi-\Theta(0,\pi,\psi) \right\}, \\ 
M_5 &:=
\left\{(\theta, \phi, 0) \in \R^3; 0 \leq \phi \leq \pi, 0 \leq \theta \leq \pi-\Theta(0,\phi,0) \right\}, \\ 
M_6 &:=
\left\{(\theta, \phi, \pi) \in \R^3; 0 \leq \phi \leq \pi, 0 \leq \theta \leq \pi-\Theta(0,\phi,\pi) \right\}.
\end{aligned}\end{equation*}
Then, we have
\begin{equation*}
 \partial D = M_1 \cup \dots \cup M_6
\end{equation*}
and $M_i \cap M_j$ is empty or a curve for each $i \not= j$.
In the case where $(F,G,H)$ has no zeros on $\partial D$,
we can use the degree of a map
\begin{equation*}
\mathscr{F}:=\frac{(F,G,H)}{\sqrt{F^2+G^2+H^2}}: \partial D \rightarrow S^2
\end{equation*}
to find a zero in the interior of $D$.
In order to calculate the degree of $\mathscr{F}$, in the rest of this section, we show the identities 
\eqref{eq:8}, \eqref{eq:10}, and \eqref{eq:11} in Proposition \ref{prop:7} below.

\subsection{Some formulas}
\label{sec:5.4}

We prepare some formulas.
\begin{lemma}
\label{lem:2}
For $K \in \hatK$ and $\xi,\eta,\zeta \in \R$,
\begin{equation*}
\rho_{(X(\zeta)K)}(P(\xi,\eta))=\rho_{K}(P(\xi,\eta-\zeta)),\ 
\rho_{K}(P(\xi,\eta\pm \pi))=\rho_{K}(P(\pi-\xi,\eta)).
\end{equation*}
\end{lemma}

\begin{proof}
By the definition of $\mu_{K}$, we have
\begin{equation*}\begin{aligned}
 \mu_{(X(\zeta)K)}(P) 
&=
\min\{t \geq 0; P \in t X(\zeta)K\} \\
&=
\min\{t \geq 0; X(-\zeta)P \in t K\}=
 \mu_{K}(X(-\zeta)P).
\end{aligned}\end{equation*}
By the definition of $\rho_K$, we get
\begin{equation*}
\rho_{(X(\zeta)K)}(P)=\rho_{K}(X(-\zeta)P).
\end{equation*}
Moreover, for $P=P(\xi,\eta) \in S^2$,
$X(-\zeta)P(\xi, \eta)=P(\xi,\eta-\zeta)$,
which proves the first equation.

Next, by a direct calculation, we have $P(\xi,\eta\pm \pi)=-P(\pi-\xi,\eta)$.
Since $K$ is centrally symmetric, we obtain
\begin{equation*}
\rho_{K}(P(\xi,\eta\pm \pi))
=
\rho_{K}(-P(\pi-\xi,\eta))
=
\rho_{K}(P(\pi-\xi,\eta)).
\end{equation*}
\end{proof}

\begin{lemma}
\label{lem:3}
 $\Theta(K) + \Theta(X(\pi-\Theta(K))K)=\pi.$
\end{lemma}

\begin{proof}
For $\eta_0, \eta_1 \in \R$, by Lemma \ref{lem:2}, we have
\begin{align}
\notag
& \int_{\eta_0}^{\eta_1} d \eta \int_0^{\pi} \rho_{K}^3(P(\xi,\eta)) \sin \xi \,d \xi \\
\notag
&= 
\int_{\eta_0}^{\eta_1} d \eta \int_0^{\pi} \rho_{(X(\pi-\Theta(K))K)}^3(P(\xi,\eta+\pi-\Theta(K))) \sin \xi \,d \xi \\
\label{eq:9}
&= 
\int_{\eta_0+\pi-\Theta(K)}^{\eta_1+\pi-\Theta(K)} d \tilde{\eta} \int_0^{\pi} \rho_{(X(\pi-\Theta(K))K)}^3(P(\xi,\tilde{\eta})) \sin \xi \,d \xi,
\end{align}
where we used the substitution $\tilde{\eta}=\eta+\pi-\Theta(K)$.
Applying \eqref{eq:9} with $(\eta_0, \eta_1)=(0,\Theta(K)), (\Theta(K),\pi)$ to 
the definition \eqref{eq:21} of $\Theta(K)$, we obtain
\begin{equation*}\begin{aligned}
& 
\int_{\pi-\Theta(K)}^{\pi} d \tilde{\eta} \int_0^{\pi} \rho_{(X(\pi-\Theta(K))K)}^3(P(\xi,\tilde{\eta})) \sin \xi \,d \xi \\
&= 
\int_{\pi}^{2 \pi-\Theta(K)} d \tilde{\eta} \int_0^{\pi} \rho_{(X(\pi-\Theta(K))K)}^3(P(\xi,\tilde{\eta})) \sin \xi \,d \xi \\
&= 
\int_{0}^{\pi-\Theta(K)} d \tilde{\eta} \int_0^{\pi} \rho_{(X(\pi-\Theta(K))K)}^3(P(\xi,\tilde{\eta}+\pi)) \sin \xi \,d \xi \\
&= 
\int_{0}^{\pi-\Theta(K)} d \tilde{\eta} \int_0^{\pi} \rho_{(X(\pi-\Theta(K))K)}^3(P(\pi-\xi,\tilde{\eta})) \sin \xi \,d \xi \\
&= 
\int_{0}^{\pi-\Theta(K)} d \tilde{\eta} \int_0^{\pi} \rho_{(X(\pi-\Theta(K))K)}^3(P(\tilde{\xi},\tilde{\eta})) \sin \tilde{\xi} \,d \tilde{\xi},
\end{aligned}\end{equation*}
where we used  Lemma \ref{lem:2} and the substitution $\tilde{\xi}=\pi-\xi$.
Since $\pi-\Theta(K) \in (0, \pi)$, the uniqueness of $\Theta(X(\pi-\Theta(K))K))$ implies $\Theta(X(\pi-\Theta(K))K))= \pi-\Theta(K)$ from \eqref{eq:21}.
\end{proof}

\subsection{Identities relating to the rotation $X(\pi-\Theta(K))$ and $X(\pi)$}

Here we examine the behavior of the functions $F$, $G$, and $H$ under the rotation $X$ of $K$.
\begin{lemma}
\label{lem:4}
Put $L=X(\pi-\Theta(K))K$. 
Then
\begin{equation*}
\Theta(L)=\pi-\Theta(K), \quad
\Phi(L)=\pi-\Psi(K), \quad
\Psi(L)=\Phi(K).
\end{equation*}
\end{lemma}

\begin{proof}
For simplicity, we put
$\tilde{\Theta}=\Theta(L)$, 
$\Theta=\Theta(K)$,
$\tilde{\Phi}=\Phi(L)$, 
$\Phi=\Phi(K)$,
$\tilde{\Psi}=\Psi(L)$, 
$\Psi=\Psi(K)$.
Note that $\tilde{\Theta}=\pi-\Theta$ by Lemma \ref{lem:3}, which means the first formula.

By the definition of $\tilde{\Phi}$, we have
\begin{equation*}
\int_0^{\tilde{\Phi}} \rho_{L}^2(P(\xi,0)) \,d\xi =
\int_{\tilde{\Phi}}^\pi \rho_{L}^2(P(\xi,0)) \,d\xi.
\end{equation*}
By Lemma \ref{lem:2}, we get
\begin{equation*}
\rho_{L}^2(P(\xi,0))
=
\rho_{X(\pi-\Theta) K}^2(P(\xi,0)) 
=
\rho_{K}^2(P(\xi,\Theta-\pi)) 
=
\rho_{K}^2(P(\pi-\xi,\Theta))
\end{equation*}
and
\begin{equation*}
\int_0^{\tilde{\Phi}} \rho_{L}^2(P(\xi,0)) \,d\xi
=
\int_0^{\tilde{\Phi}} \rho_{K}^2(P(\pi-\xi,\Theta)) \,d\xi
=
\int_{\pi-\tilde{\Phi}}^\pi \rho_{K}^2(P(\tilde{\xi},\Theta)) \,d\tilde{\xi},
\end{equation*}
where we used the substitution $\tilde{\xi}=\pi-\xi$.
Similarly,
\begin{equation*}
\int_{\tilde{\Phi}}^\pi \rho_{L}^2(P(\xi,0)) \,d\xi
=
\int_{\tilde{\Phi}}^\pi \rho_{K}^2(P(\pi-\xi,\Theta)) \,d\xi
=
\int_0^{\pi-\tilde{\Phi}} \rho_{K}^2(P(\tilde{\xi},\Theta)) \,d\tilde{\xi}.
\end{equation*}
Consequently, we have
\begin{equation*}
\int_{\pi-\tilde{\Phi}}^\pi \rho_{K}^2(P(\tilde{\xi},\Theta)) \,d\tilde{\xi}
=
\int_0^{\pi-\tilde{\Phi}} \rho_{K}^2(P(\tilde{\xi},\Theta)) \,d\tilde{\xi}.
\end{equation*}
By the uniqueness of $\Psi=\Psi(K)$, we see $\Psi=\pi-\tilde{\Phi}$.

Similarly, by the definition of $\tilde{\Psi}$,
\begin{equation*}
 \int_0^{\tilde{\Psi}} \rho_{L}^2(P(\xi,\tilde{\Theta})) \,d\xi =
 \int_{\tilde{\Psi}}^\pi \rho_{L}^2(P(\xi,\tilde{\Theta})) \,d\xi
\end{equation*}
holds.
By Lemma \ref{lem:2}, we have
\begin{equation*}
\rho_{L}^2(P(\xi,\tilde{\Theta}))
=
\rho_{X(\pi-\Theta) K}^2(P(\xi,\pi-\Theta))
=
\rho_{K}^2(P(\xi,0)).
\end{equation*}
Thus we get
\begin{equation*}
 \int_0^{\tilde{\Psi}} \rho_{K}^2(P(\xi,0)) \,d\xi =
 \int_{\tilde{\Psi}}^\pi \rho_{K}^2(P(\xi,0)) \,d\xi.
\end{equation*}
By the uniqueness of $\Phi=\Phi(K)$, we obtain $\tilde{\Psi}=\Phi$.
\end{proof}

\begin{lemma}
\label{lem:5}
\begin{equation*}\begin{aligned}
& F(X(\pi-\Theta(K))K)=-F(K), \\
& G(X(\pi-\Theta(K))K)=-H(K), \\
& H(X(\pi-\Theta(K))K)=G(K).
\end{aligned}\end{equation*}
\end{lemma}

\begin{proof}
We use the same notations as in Lemma \ref{lem:4}.
Putting $\tilde{\linearop}=\linearop(L)$, $\linearop=\linearop(K)$, we have
\begin{equation*}
\tilde{\linearop}L
=
\tilde{\linearop} X(\pi-\Theta) K
=
\left(
\tilde{\linearop} X(\pi-\Theta)
\mathcal{A}^{-1}
\right) \mathcal{A} K.
\end{equation*}
By a direct calculation using Lemma \ref{lem:4},
\begin{equation*}
(\tilde{\linearop} X(\pi-\Theta) \linearop^{-1})^{-1}
=\linearop X(\Theta-\pi) \tilde{\linearop}^{-1}=
\begin{pmatrix}
1 & 0 & 0 \\
0 & 0 & 1/\sin \Theta \\
0 & -\sin \Theta & 0
\end{pmatrix}.
\end{equation*}
Thus, 
for any $\tilde{P}=(\tilde{x}, \tilde{y}, \tilde{z}) \in \R^3$, 
we have
\begin{equation*}\begin{aligned}
\chi_{\tilde{\linearop}L}(\tilde{P})
&=
\chi_{(\tilde{\linearop} X(\pi-\Theta) \linearop^{-1} \linearop K)}(\tilde{P}) \\
&=
\chi_{\linearop K}((\tilde{\linearop} X(\pi-\Theta) \linearop^{-1})^{-1} \tilde{P}) \\
&=
\chi_{\linearop K}
\left(
\begin{pmatrix}
1 & 0 & 0 \\
0 & 0 & 1/\sin \Theta \\
0 & -\sin \Theta & 0
\end{pmatrix}
\begin{pmatrix}
\tilde{x} \\ \tilde{y} \\ \tilde{z}
\end{pmatrix}
\right)
=
\chi_{\linearop K}
\left(
\begin{pmatrix}
\tilde{x} \\ \tilde{z}/\sin \Theta \\ - \tilde{y} \sin \Theta
\end{pmatrix}
\right).
\end{aligned}\end{equation*}
By the substitution $x=\tilde{x}$, $y=\tilde{z}/\sin \Theta$, $z=-\tilde{y} \sin \Theta$,
\begin{equation*}\begin{aligned}
|(\tilde{\linearop}L)_{+++}|
&=
\int_{\{\tilde{x} > 0, \tilde{y} > 0, \tilde{z} > 0\}} \chi_{\tilde{\linearop}L}(\tilde{P}) \,d\tilde{x}d\tilde{y}d\tilde{z} \\
&=
\int_{\{\tilde{x} > 0, \tilde{y} > 0, \tilde{z} > 0\}} \chi_{\linearop K}
\left(
\begin{pmatrix}
\tilde{x} \\ \tilde{z}/\sin \Theta \\ - \tilde{y} \sin \Theta
\end{pmatrix}
\right)
 \,d\tilde{x}d\tilde{y}d\tilde{z} \\
&=
\int_{\{x > 0, y > 0, z < 0\}} \chi_{\linearop K}(P)
 \,dxdydz
=
|(\linearop K)_{++-}|
=
|(\linearop K)_{--+}|,
\end{aligned}\end{equation*}
where the last equality comes from the fact that $\linearop K$ is centrally symmetric.
Similarly, we obtain
\begin{equation*}\begin{aligned}
|(\tilde{\linearop}L)_{-++}|
&=
\int_{\{\tilde{x} < 0, \tilde{y} > 0, \tilde{z} > 0\}} \chi_{\tilde{\linearop}L}(\tilde{P}) \,d\tilde{x}d\tilde{y}d\tilde{z} \\
&=
\int_{\{x < 0, y > 0, z < 0\}} \chi_{\linearop K}(P)
 \,dxdydz
=
|(\linearop K)_{-+-}|
=
|(\linearop K)_{+-+}|, \\
|(\tilde{\linearop}L)_{--+}|
&=
\int_{\{\tilde{x} < 0, \tilde{y} < 0, \tilde{z} > 0\}} \chi_{\tilde{\linearop}L}(\tilde{P}) \,d\tilde{x}d\tilde{y}d\tilde{z} \\
&=
\int_{\{x < 0, y > 0, z > 0\}} \chi_{\linearop K}(P)
 \,dxdydz
=
|(\linearop K)_{-++}|, \\
|(\tilde{\linearop}L)_{+-+}|
&=
\int_{\{\tilde{x} > 0, \tilde{y} < 0, \tilde{z} > 0\}} \chi_{\tilde{\linearop}L}(\tilde{P}) \,d\tilde{x}d\tilde{y}d\tilde{z} \\
&=
\int_{\{x > 0, y > 0, z > 0\}} \chi_{\linearop K}(P)
 \,dxdydz
=
|(\linearop K)_{+++}|.
\end{aligned}\end{equation*}
Since $d\tilde{y}d\tilde{z}=dydz$, we also obtain
\begin{equation*}\begin{aligned}
|O*\alpha(\tilde{\linearop}L)|
&=
\int_{\{\tilde{x} = 0, \tilde{y} > 0, \tilde{z} > 0\}} \chi_{\tilde{\linearop}L}(\tilde{P}) \,d\tilde{y}d\tilde{z} \\
&=
\int_{\{x = 0, y > 0, z < 0\}} \chi_{\linearop K}(P) \,dydz 
=|O*-a(\linearop K)|
=|O*a(\linearop K)|, \\
|O*a(\tilde{\linearop}L)|
&=
\int_{\{\tilde{x} = 0, \tilde{y} < 0, \tilde{z} > 0\}} \chi_{\tilde{\linearop}L}(\tilde{P}) \,d\tilde{y}d\tilde{z} \\
&=
\int_{\{x = 0, y > 0, z > 0\}} \chi_{\linearop K}(P) \,dydz 
=|O*\alpha(\linearop K)|.
\end{aligned}\end{equation*}
Therefore,
\begin{equation*}\begin{aligned}
F(L)
&=
|O*\alpha(\tilde{\linearop}L)|-|O*a(\tilde{\linearop}L)| \\
&=
|O*a(\linearop K)|-|O*\alpha(\linearop K)|=-F(K), \\
G(L)
&=
|(\tilde{\linearop}L)_{+++}|+|(\tilde{\linearop}L)_{--+}|-|(\tilde{\linearop}L)_{-++}|-|(\tilde{\linearop}L)_{+-+}| \\
&=
|(\linearop K)_{--+}|+|(\linearop K)_{-++}|-|(\linearop K)_{+-+}|-|(\linearop K)_{+++}|=-H(K), \\
H(L)
&=
|(\tilde{\linearop}L)_{+++}|+|(\tilde{\linearop}L)_{+-+}|-|(\tilde{\linearop}L)_{-++}|-|(\tilde{\linearop}L)_{--+}| \\
&=
|(\linearop K)_{--+}|+|(\linearop K)_{+++}|-|(\linearop K)_{+-+}|-|(\linearop K)_{-++}|=G(K).
\end{aligned}\end{equation*}
\end{proof}

Using the above lemmas, 
we get identities relating to the rotation $X(\pi)$.
\begin{lemma}
\label{lem:6}
\begin{equation*}
\Theta(X(\pi) K)=\Theta(K), \quad
\Phi(X(\pi) K)=\pi-\Phi(K), \quad
\Psi(X(\pi) K)=\pi-\Psi(K),
\end{equation*}
\begin{equation*}
 F(X(\pi)K)=F(K), \quad
 G(X(\pi)K)=-G(K), \quad
 H(X(\pi)K)=-H(K).
\end{equation*}
\end{lemma}

\begin{proof}
Put $L=X(\pi-\Theta(K))K$.
By Lemma \ref{lem:4}, we have $\Theta(L)+\Theta(K)=\pi$.
Thus,
\begin{equation*}\begin{aligned}
X(\pi)K
&=
X(2 \pi - \Theta(L) - \Theta(K)) K \\
&=
X(\pi - \Theta(L))
X(\pi - \Theta(K)) K \\
&=
X(\pi-\Theta(L)) L,
\end{aligned}\end{equation*}
and so we can apply Lemmas \ref{lem:4} and \ref{lem:5} to 
$X(\pi) K =X(\pi-\Theta(L)) L$.
By Lemma \ref{lem:4}, we have
\begin{equation*}\begin{aligned}
\Theta(X(\pi)K)
&=
\pi-\Theta(L)
=
\pi-(\pi-\Theta(K))
=\Theta(K), \\
\Phi(X(\pi)K)
&
=
\pi-\Psi(L)
=
\pi-\Phi(K), \\
\Psi(X(\pi)K)
&=
\Phi(L)
=
\pi-\Psi(K).
\end{aligned}\end{equation*}
It follows from Lemma \ref{lem:5} that
\begin{equation*}\begin{aligned}
F(X(\pi)K)
&=F(X(\pi-\Theta(L))L)
=-F(L)
=F(K), \\
G(X(\pi)K)
&=G(X(\pi-\Theta(L))L)
=-H(L)
=-G(K), \\
H(X(\pi)K)
&=H(X(\pi-\Theta(L))L)
=G(L)
=-H(K).
\end{aligned}\end{equation*}
\end{proof}

\subsection{Identities relating to the rotation $Y(\pi)$}

\begin{lemma}
\label{lem:7}
Put $L=Y(\pi)K$.
Then
\begin{equation*}
\Theta(L)=\pi-\Theta(K), \quad
\Phi(L)=\pi-\Phi(K), \quad
\Psi(L)=\Psi(K).
\end{equation*}
\end{lemma}

\begin{proof}
For simplicity, we put
$\tilde{\Theta}=\Theta(L)$,
$\Theta=\Theta(K)$,
$\tilde{\Phi}=\Phi(L)$,
$\Phi=\Phi(K)$,
$\tilde{\Psi}=\Psi(L)$,
$\Psi=\Psi(K)$.
We first claim that
\begin{equation}
\label{eq:5}
 \rho_{L}(P(\xi,\eta))=\rho_{K}(P(\xi,\pi-\eta)).
\end{equation}
Actually, 
$Y(-\pi) P(\xi,\eta)=-P(\xi,\pi-\eta)$ holds by a direct calculation.
Thus we have
\begin{equation*}
\rho_{L}(P(\xi,\eta))
=
\rho_{Y(\pi)K}(P(\xi,\eta))
=
\rho_{K}(Y(-\pi)P(\xi,\eta))
=
\rho_{K}(-P(\xi,\pi-\eta)).
\end{equation*}
Since $K$ is centrally symmetric, the claim \eqref{eq:5} holds.

By \eqref{eq:5} and the definition of $\tilde{\Theta}$, 
\begin{equation*}
\int_0^{\tilde{\Theta}} d \eta \int_0^{\pi} \rho_{K}^3(P(\xi,\pi-\eta)) \sin \xi \,d \xi
=
\int_{\tilde{\Theta}}^\pi d \eta \int_0^{\pi} \rho_{K}^3(P(\xi,\pi-\eta)) \sin \xi \,d \xi.
\end{equation*}
By the substitution $\tilde{\eta}=\pi-\eta$,
\begin{equation*}
\int_{\pi-\tilde{\Theta}}^\pi d \tilde{\eta} \int_0^{\pi} \rho_{K}^3(P(\xi,\tilde{\eta})) \sin \xi \,d \xi
=
\int_0^{\pi-\tilde{\Theta}} d \tilde{\eta} \int_0^{\pi} \rho_{K}^3(P(\xi,\tilde{\eta})) \sin \xi \,d \xi
\end{equation*}
By the definition and uniqueness of $\Theta$, we get $\Theta=\pi-\tilde{\Theta}$.
By the definition of $\tilde{\Phi}$, 
\begin{equation*}
\int_0^{\tilde{\Phi}} \rho_{L}^2(P(\xi,0)) \,d\xi =
\int_{\tilde{\Phi}}^\pi \rho_{L}^2(P(\xi,0)) \,d\xi.
\end{equation*}
By Lemma \ref{lem:2}, $\rho_{K}(P(\xi,\pi-\eta))=\rho_{K}(P(\pi-\xi,-\eta))$ holds.
It then follows from \eqref{eq:5} that
\begin{equation*}\begin{aligned}
\int_0^{\tilde{\Phi}} \rho_{L}^2(P(\xi,0)) \,d\xi
&=
\int_0^{\tilde{\Phi}} \rho_{K}^2(P(\pi-\xi,0)) \,d\xi
=
\int_{\pi-\tilde{\Phi}}^\pi \rho_{K}^2(P(\tilde{\xi},0)) \,d\tilde{\xi}, \\
\int_{\tilde{\Phi}}^\pi \rho_{L}^2(P(\xi,0)) \,d\xi
&=
\int_{\tilde{\Phi}}^\pi \rho_{K}^2(P(\pi-\xi,0)) \,d\xi
=
\int_0^{\pi-\tilde{\Phi}} \rho_{K}^2(P(\tilde{\xi},0)) \,d\tilde{\xi},
\end{aligned}\end{equation*}
where we used the substitution $\tilde{\xi}=\pi-\xi$.
Therefore, by the definition of $\Phi$, $\Phi=\pi-\tilde{\Phi}$.
Similarly, by the definition of $\tilde{\Psi}$,
\begin{equation*}
 \int_0^{\tilde{\Psi}} \rho_{L}^2(P(\xi,\tilde{\Theta})) \,d\xi =
 \int_{\tilde{\Psi}}^\pi \rho_{L}^2(P(\xi,\tilde{\Theta})) \,d\xi.
\end{equation*}
By \eqref{eq:5} and $\Theta=\pi-\tilde{\Theta}$, we have
$\rho_{L}^2(P(\xi,\tilde{\Theta})) = \rho_{K}^2(P(\xi,\pi-\tilde{\Theta})) = \rho_{K}^2(P(\xi,\Theta))$.
Thus
\begin{equation*}
 \int_0^{\tilde{\Psi}} \rho_{K}^2(P(\xi,\Theta)) \,d\xi =
 \int_{\tilde{\Psi}}^\pi \rho_{K}^2(P(\xi,\Theta)) \,d\xi,
\end{equation*}
which means that $\tilde{\Psi}=\Psi$.
\end{proof}

\begin{lemma}
\label{lem:8}
\begin{equation*}
F(Y(\pi)K)=-F(K), \quad
G(Y(\pi)K)=-G(K), \quad
H(Y(\pi)K)=H(K).
\end{equation*} 
\end{lemma}

\begin{proof}
We use the same notations as in Lemma \ref{lem:7}.
Putting $\tilde{\linearop}=\linearop(L)$, $\linearop=\linearop(K)$, 
we have 
\begin{equation*}
\tilde{\linearop}L=\tilde{\linearop}Y(\pi)K=(\tilde{\linearop} Y(\pi) \linearop^{-1}) \linearop K.
\end{equation*}
Thus
\begin{equation*}
\chi_{\tilde{\linearop}L}(P)
=
\chi_{\linearop K}((\tilde{\linearop} Y(\pi) \linearop^{-1})^{-1}P)
=
\chi_{\linearop K}(\linearop Y(-\pi)\tilde{\linearop}^{-1} P).
\end{equation*}
By Lemma \ref{lem:7} and a direct calculation,
\begin{equation*}
\linearop Y(-\pi)\tilde{\linearop}^{-1}
=
\begin{pmatrix}
-1 & 0 & 0 \\
0 & 1 & 0 \\
0 & 0 & -1
\end{pmatrix}.
\end{equation*}
Thus, for any $\tilde{P}=(\tilde{x}, \tilde{y}, \tilde{z}) \in \R^3$, we have
\begin{equation*}\begin{aligned}
|(\tilde{\linearop}L)_{+++}|
&=
\int_{\{\tilde{x} > 0, \tilde{y} > 0, \tilde{z} > 0\}} \chi_{\tilde{\linearop}L}(\tilde{P}) \,d\tilde{x}d\tilde{y}d\tilde{z} \\
&=
\int_{\{\tilde{x} > 0, \tilde{y} > 0, \tilde{z} > 0\}} \chi_{\linearop K}
\left(
\begin{pmatrix}
-1 & 0 & 0 \\
0 & 1 & 0 \\
0 & 0 & -1
\end{pmatrix}
\begin{pmatrix}
\tilde{x} \\ \tilde{y} \\ \tilde{z}
\end{pmatrix}
\right)
 \,d\tilde{x}d\tilde{y}d\tilde{z} \\
&=
\int_{\{\tilde{x} > 0, \tilde{y} > 0, \tilde{z} > 0\}} \chi_{\linearop K}
\left(
\begin{pmatrix}
-\tilde{x} \\ \tilde{y} \\ -\tilde{z}
\end{pmatrix}
\right)
 \,d\tilde{x}d\tilde{y}d\tilde{z} \\
&=
\int_{\{x < 0, y > 0, z < 0\}} \chi_{\linearop K}(P)
 \,dxdydz
=
|(\linearop K)_{-+-}|
=
|(\linearop K)_{+-+}|,
\end{aligned}\end{equation*}
where we used the substitution $x=-\tilde{x}$, $y=\tilde{y}$, $z=-\tilde{z}$.
Similarly, we have
\begin{equation*}\begin{aligned}
|(\tilde{\linearop}L)_{-++}|
&=
\int_{\{\tilde{x} < 0, \tilde{y} > 0, \tilde{z} > 0\}} \chi_{\tilde{\linearop}L}(\tilde{P}) \,d\tilde{x}d\tilde{y}d\tilde{z} \\
&=
\int_{\{x > 0, y > 0, z < 0\}} \chi_{\linearop K}(P)
 \,dxdydz
=
|(\linearop K)_{++-}|
=
|(\linearop K)_{--+}|, \\
|(\tilde{\linearop}L)_{--+}|
&=
\int_{\{\tilde{x} < 0, \tilde{y} < 0, \tilde{z} > 0\}} \chi_{\tilde{\linearop}L}(\tilde{P}) \,d\tilde{x}d\tilde{y}d\tilde{z} \\
&=
\int_{\{x > 0, y < 0, z < 0\}} \chi_{\linearop K}(P)
 \,dxdydz
=
|(\linearop K)_{+--}|
=
|(\linearop K)_{-++}|, \\
|(\tilde{\linearop}L)_{+-+}|
&=
\int_{\{\tilde{x} > 0, \tilde{y} < 0, \tilde{z} > 0\}} \chi_{\tilde{\linearop}L}(\tilde{P}) \,d\tilde{x}d\tilde{y}d\tilde{z} \\
&=
\int_{\{x < 0, y < 0, z < 0\}} \chi_{\linearop K}(P)
 \,dxdydz
=
|(\linearop K)_{---}|
=
|(\linearop K)_{+++}|.
\end{aligned}\end{equation*}
Moreover, we get
\begin{equation*}\begin{aligned}
|O*\alpha(\tilde{\linearop}L)|
&=
\int_{\{\tilde{x} = 0, \tilde{y} > 0, \tilde{z} > 0\}} \chi_{\tilde{\linearop}L}(\tilde{P}) \,d\tilde{y}d\tilde{z} \\
&=
\int_{\{x = 0, y > 0, z < 0\}} \chi_{\linearop K}(P) \,dydz 
=|O*-a(\linearop K)|
=|O*a(\linearop K)|, \\
|O*a(\tilde{\linearop}L)|
&=
\int_{\{\tilde{x} = 0, \tilde{y} < 0, \tilde{z} > 0\}} \chi_{\tilde{\linearop}L}(\tilde{P}) \,d\tilde{y}d\tilde{z} \\
&=
\int_{\{x = 0, y < 0, z < 0\}} \chi_{\linearop K}(P) \,dydz 
=|O*-\alpha(\linearop K)|
=|O*\alpha(\linearop K)|.
\end{aligned}\end{equation*}
Therefore, 
\begin{equation*}\begin{aligned}
F(L)
&=
|O*\alpha(\tilde{\linearop}L)|-|O*a(\tilde{\linearop}L)| \\
&=
|O*a(\linearop K)|-|O*\alpha(\linearop K)|=-F(K), \\
G(L)
&=
|(\tilde{\linearop}L)_{+++}|+|(\tilde{\linearop}L)_{--+}|-|(\tilde{\linearop}L)_{-++}|-|(\tilde{\linearop}L)_{+-+}| \\
&=
|(\linearop K)_{+-+}|+|(\linearop K)_{-++}|-|(\linearop K)_{--+}|-|(\linearop K)_{+++}|=-G(K), \\
H(L)
&=
|(\tilde{\linearop}L)_{+++}|+|(\tilde{\linearop}L)_{+-+}|-|(\tilde{\linearop}L)_{-++}|-|(\tilde{\linearop}L)_{--+}| \\
&=
|(\linearop K)_{+-+}|+|(\linearop K)_{+++}|-|(\linearop K)_{--+}|-|(\linearop K)_{-++}|=H(K).
\end{aligned}\end{equation*}
\end{proof}

\subsection{Identities relating to the rotation $Z(\pi)$}

\begin{lemma}
\label{lem:9}
Put $L=Z(\pi)K$. Then
\begin{equation*}
\Theta(L)=\pi-\Theta(K), \quad
\Phi(L)=\Phi(K), \quad
\Psi(L)=\pi-\Psi(K).
\end{equation*}
\end{lemma}

\begin{proof}
We put
$\tilde{\Theta}=\Theta(L)$, 
$\Theta=\Theta(K)$,
$\tilde{\Phi}=\Phi(L)$, 
$\Phi=\Phi(K)$,
$\tilde{\Psi}=\Psi(L)$, 
$\Psi=\Psi(K)$.
Since
$Z(-\pi) P(\xi,\eta)=P(\pi-\xi,\pi-\eta)$,
we have
\begin{equation*}
\rho_{L}(P(\xi,\eta))
=
\rho_{Z(\pi)K}(P(\xi,\eta))
=
\rho_{K}(Z(-\pi)P(\xi,\eta))
=
\rho_{K}(P(\pi-\xi,\pi-\eta)).
\end{equation*}
It follows from Lemma \ref{lem:2} that
\begin{equation}
\label{eq:6}
 \rho_{L}(P(\xi,\eta))=\rho_{K}(P(\pi-\xi,\pi-\eta))
=\rho_{K}(P(\xi,-\eta)).
\end{equation}

By \eqref{eq:6} and the definition of $\tilde{\Theta}$,
we have
\begin{equation*}
\int_0^{\tilde{\Theta}} d \eta \int_0^{\pi} \rho_{K}^3(P(\pi-\xi,\pi-\eta)) \sin \xi \,d \xi
=
\int_{\tilde{\Theta}}^\pi d \eta \int_0^{\pi} \rho_{K}^3(P(\pi-\xi,\pi-\eta)) \sin \xi \,d \xi.
\end{equation*}
By the substitution $\tilde{\xi}=\pi-\xi$, $\tilde{\eta}=\pi-\eta$,
\begin{equation*}
\int_{\pi-\tilde{\Theta}}^\pi d \tilde{\eta} \int_0^{\pi} \rho_{K}^3(P(\tilde{\xi},\tilde{\eta})) \sin \tilde{\xi} \,d \tilde{\xi}
=
\int_0^{\pi-\tilde{\Theta}} d \tilde{\eta} \int_0^{\pi} \rho_{K}^3(P(\tilde{\xi},\tilde{\eta})) \sin \tilde{\xi} \,d \tilde{\xi}.
\end{equation*}
By the definition of $\Theta$, we get $\Theta=\pi-\tilde{\Theta}$.
Furthermore, we have
\begin{equation*}
\int_0^{\tilde{\Phi}} \rho_{K}^2(P(\xi,0)) \,d\xi =
\int_0^{\tilde{\Phi}} \rho_{L}^2(P(\xi,0)) \,d\xi =
\int_{\tilde{\Phi}}^\pi \rho_{L}^2(P(\xi,0)) \,d\xi
=\int_{\tilde{\Phi}}^\pi \rho_{K}^2(P(\xi,0)) \,d\xi.
\end{equation*}
By the definition of $\Phi$, we obtain $\Phi=\tilde{\Phi}$.
Finally, the definition of $\tilde{\Psi}$ means
\begin{equation*}
 \int_0^{\tilde{\Psi}} \rho_{L}^2(P(\xi,\tilde{\Theta})) \,d\xi =
 \int_{\tilde{\Psi}}^\pi \rho_{L}^2(P(\xi,\tilde{\Theta})) \,d\xi.
\end{equation*}
By \eqref{eq:6} and $\Theta=\pi-\tilde{\Theta}$, we obtain
\begin{equation*}
\rho_{L}^2(P(\xi,\tilde{\Theta}))
=
\rho_{K}^2(P(\pi-\xi,\pi-\tilde{\Theta}))
=
\rho_{K}^2(P(\pi-\xi,\Theta)).
\end{equation*}
Thus
\begin{equation*}\begin{aligned}
\int_0^{\tilde{\Psi}} \rho_{L}^2(P(\xi,\tilde{\Theta})) \,d\xi
&=
\int_0^{\tilde{\Psi}} \rho_{K}^2(P(\pi-\xi,\Theta)) \,d\xi
=
\int_{\pi-\tilde{\Psi}}^\pi \rho_{K}^2(P(\tilde{\xi},\Theta)) \,d\tilde{\xi}, \\
\int_{\tilde{\Psi}}^\pi \rho_{L}^2(P(\xi,\tilde{\Theta})) \,d\xi
&=
\int_{\tilde{\Psi}}^\pi \rho_{K}^2(P(\pi-\xi,\Theta)) \,d\xi
=
\int_0^{\pi-\tilde{\Psi}} \rho_{K}^2(P(\tilde{\xi},\Theta)) \,d\tilde{\xi}.
\end{aligned}\end{equation*}
Therefore, $\pi-\tilde{\Psi}=\Psi$.
\end{proof}

\begin{lemma}
\label{lem:10} 
\begin{equation*}\begin{aligned}
F(Z(\pi)K)=-F(K), \quad
G(Z(\pi)K)=G(K), \quad
H(Z(\pi)K)=-H(K).
\end{aligned}\end{equation*}
\end{lemma}

\begin{proof}
We use the same notations as in Lemma \ref{lem:9}.
Putting $\tilde{\linearop}=\linearop(L)$, $\linearop=\linearop(K)$, we have
\begin{equation*}
\tilde{\linearop}L=\tilde{\linearop}Z(\pi)K=(\tilde{\linearop} Z(\pi) \linearop^{-1}) \linearop K.
\end{equation*}
Then
\begin{equation*}
\chi_{\tilde{\linearop}L}(P)
=
\chi_{\linearop K}((\tilde{\linearop} Z(\pi) \linearop^{-1})^{-1}P)
=
\chi_{\linearop K}(\linearop Z(-\pi)\tilde{\linearop}^{-1} P).
\end{equation*}
By Lemma \ref{lem:9} and a direct calculation,
\begin{equation*}
 \linearop Z(-\pi)\tilde{\linearop}^{-1} =
\begin{pmatrix}
-1 & 0 & 0 \\
0 & -1 & 0 \\
0 & 0 & 1
\end{pmatrix}.
\end{equation*}
Thus, for any $\tilde{P}=(\tilde{x}, \tilde{y}, \tilde{z}) \in \R^3$, 
we see
\begin{equation*}\begin{aligned}
|(\tilde{\linearop}L)_{+++}|
&=
\int_{\{\tilde{x} > 0, \tilde{y} > 0, \tilde{z} > 0\}} \chi_{\tilde{\linearop}L}(\tilde{P}) \,d\tilde{x}d\tilde{y}d\tilde{z} \\
&=
\int_{\{\tilde{x} > 0, \tilde{y} > 0, \tilde{z} > 0\}} \chi_{\linearop K}
\left(
\begin{pmatrix}
-1 & 0 & 0 \\
0 & -1 & 0 \\
0 & 0 & 1
\end{pmatrix}
\begin{pmatrix}
\tilde{x} \\ \tilde{y} \\ \tilde{z}
\end{pmatrix}
\right)
 \,d\tilde{x}d\tilde{y}d\tilde{z} \\
&=
\int_{\{\tilde{x} > 0, \tilde{y} > 0, \tilde{z} > 0\}} \chi_{\linearop K}
\left(
\begin{pmatrix}
-\tilde{x} \\ -\tilde{y} \\ \tilde{z}
\end{pmatrix}
\right)
 \,d\tilde{x}d\tilde{y}d\tilde{z} \\
&=
\int_{\{x < 0, y < 0, z > 0\}} \chi_{\linearop K}(P)
 \,dxdydz
=
|(\linearop K)_{--+}|,
\end{aligned}\end{equation*}
where we used the substitution $x=-\tilde{x}$, $y=-\tilde{y}$, $z=\tilde{z}$.
Similarly, we obtain
\begin{equation*}\begin{aligned}
|(\tilde{\linearop}L)_{-++}|
&=
\int_{\{\tilde{x} < 0, \tilde{y} > 0, \tilde{z} > 0\}} \chi_{\tilde{\linearop}L}(\tilde{P}) \,d\tilde{x}d\tilde{y}d\tilde{z} \\
&=
\int_{\{x > 0, y < 0, z > 0\}} \chi_{\linearop K}(P)
 \,dxdydz
=
|(\linearop K)_{+-+}|, \\
|(\tilde{\linearop}L)_{--+}|
&=
\int_{\{\tilde{x} < 0, \tilde{y} < 0, \tilde{z} > 0\}} \chi_{\tilde{\linearop}L}(\tilde{P}) \,d\tilde{x}d\tilde{y}d\tilde{z} \\
&=
\int_{\{x > 0, y > 0, z > 0\}} \chi_{\linearop K}(P)
 \,dxdydz
=
|(\linearop K)_{+++}|, \\
|(\tilde{\linearop}L)_{+-+}|
&=
\int_{\{\tilde{x} > 0, \tilde{y} < 0, \tilde{z} > 0\}} \chi_{\tilde{\linearop}L}(\tilde{P}) \,d\tilde{x}d\tilde{y}d\tilde{z} \\
&=
\int_{\{x < 0, y > 0, z > 0\}} \chi_{\linearop K}(P)
 \,dxdydz
=
|(\linearop K)_{-++}|.
\end{aligned}\end{equation*}
Moreover, we have
\begin{equation*}\begin{aligned}
|O*\alpha(\tilde{\linearop}L)|
&=
\int_{\{\tilde{x} = 0, \tilde{y} > 0, \tilde{z} > 0\}} \chi_{\tilde{\linearop}L}(\tilde{P}) \,d\tilde{y}d\tilde{z} \\
&=
\int_{\{x = 0, y < 0, z > 0\}} \chi_{\linearop K}(P) \,dydz
=|O*a(\linearop K)|, \\
|O*a(\tilde{\linearop}L)|
&=
\int_{\{\tilde{x} = 0, \tilde{y} < 0, \tilde{z} > 0\}} \chi_{\tilde{\linearop}L}(\tilde{P}) \,d\tilde{y}d\tilde{z} \\
&=
\int_{\{x = 0, y > 0, z > 0\}} \chi_{\linearop K}(P) \,dydz
=|O*\alpha(\linearop K)|
\end{aligned}\end{equation*}
Therefore,
\begin{equation*}\begin{aligned}
F(L)
&=
|O*\alpha(\tilde{\linearop}L)|-|O*a(\tilde{\linearop}L)| \\
&=
|O*a(\linearop K)|-|O*\alpha(\linearop K)|=-F(K), \\
G(L)
&=
|(\tilde{\linearop}L)_{+++}|+|(\tilde{\linearop}L)_{--+}|-|(\tilde{\linearop}L)_{-++}|-|(\tilde{\linearop}L)_{+-+}| \\
&=
|(\linearop K)_{--+}|+|(\linearop K)_{+++}|-|(\linearop K)_{+-+}|-|(\linearop K)_{-++}|=G(K), \\
H(L)
&=
|(\tilde{\linearop}L)_{+++}|+|(\tilde{\linearop}L)_{+-+}|-|(\tilde{\linearop}L)_{-++}|-|(\tilde{\linearop}L)_{--+}| \\
&=
|(\linearop K)_{--+}|+|(\linearop K)_{-++}|-|(\linearop K)_{+-+}|-|(\linearop K)_{+++}|=-H(K).
\end{aligned}\end{equation*}
\end{proof}

\subsection{Key identities on $\partial D$}
\label{sec:5.8}

By using the lemmas prepared in the above subsections, we show Proposition \ref{prop:7} below, which states key identities on $\partial D$ for the degree calculation in the next section.
To begin with, we introduce the map $\Gamma$ defined by
\begin{equation*}
\Gamma_\psi(\theta):=\pi-\Theta(\theta,0,\psi)+\theta.
\end{equation*}

\begin{lemma}
\label{lem:11}
$\Gamma$ satisfies the following properties:
\begin{enumerate}
 \item $\theta \mapsto \Gamma_\psi(\theta)$ is increasing.
 \item $\Gamma_\psi\left([0,\pi-\Theta(0,0,\psi)]\right)=[\pi-\Theta(0,0,\psi), \pi]$.
 \item $\Gamma_\psi, (\Gamma_\psi)^{-1}$ is smooth with respect to $\theta$ and $\psi$. 
\end{enumerate}
\end{lemma}

\begin{proof}
By Lemma \ref{lem:12}, $\Gamma_\psi(\theta)$ is smooth with respect to $\theta$ and $\psi$. 
We claim that
\begin{equation}
\label{eq:19}
\frac{\partial \Gamma_\psi}{\partial \theta}(\theta) = - \frac{\partial \Theta}{\partial \theta}(\theta,0,\psi) +1 >0.
\end{equation}
If the inequality \eqref{eq:19} holds, then the properties (i), (ii), and (iii) are easily verified.
Indeed, (i) follows immediately.
Note that $\Gamma_\psi(0)=\pi-\Theta(0,0,\psi)$.
By Lemma \ref{lem:3}, we have
\begin{equation*}\begin{aligned}
\Gamma_\psi(\pi-\Theta(0,0,\psi))
&=
\pi-\Theta(\pi-\Theta(0,0,\psi),0,\psi) + \pi-\Theta(0,0,\psi) \\
&=
\pi-\Theta(X(\pi-\Theta(0,0,\psi))K(0,0,\psi)) +\pi-\Theta(0,0,\psi)
=\pi,
\end{aligned}\end{equation*}
which proves (ii),
because $\Gamma_\psi$ is increasing by (i).
Moreover, $\Gamma_\psi^{-1}$ is smooth by the implicit function theorem.

Therefore, it suffices to prove \eqref{eq:19}.
By the definition of $\Theta$,
\begin{equation*}
0 =
\int_0^{\Theta(\theta,0,\psi)} d\eta \int_0^\pi \rho_{K(\theta,0,\psi)}^3(P(\xi,\eta)) \sin \xi \,d\xi
-
\int_{\Theta(\theta,0,\psi)}^\pi d\eta \int_0^\pi \rho_{K(\theta,0,\psi)}^3(P(\xi,\eta)) \sin \xi \,d\xi.
\end{equation*}
By Lemma \ref{lem:2}, we have
\begin{equation*}
 \rho_{K(\theta,0,\psi)}(P(\xi,\eta))
=
 \rho_{X(\theta)K(0,0,\psi)}(P(\xi,\eta))
=
 \rho_{K(0,0,\psi)}(P(\xi,\eta-\theta)).
\end{equation*}
Thus, by the substitution $\tilde{\eta}=\eta-\theta$, we obtain
\begin{equation*}\begin{aligned}
0 &=
\int_{-\theta}^{\Theta(\theta,0,\psi)-\theta} d\tilde{\eta} \int_0^\pi \rho_{K(0,0,\psi)}^3(P(\xi,\tilde{\eta})) \sin \xi \,d\xi \\
&-
\int_{\Theta(\theta,0,\psi)-\theta}^{\pi-\theta} d\tilde{\eta} \int_0^\pi \rho_{K(0,0,\psi)}^3(P(\xi,\tilde{\eta})) \sin \xi \,d\xi.
\end{aligned}\end{equation*}
By differentiating it with respect to $\theta$, 
\begin{equation*}\begin{aligned}
 0&=
2 
\left(
\frac{\partial \Theta}{\partial \theta}(\theta,0,\psi) - 1
\right)
\int_0^\pi \rho_{K(0,0,\psi)}^3(P(\xi,\Theta(\theta,0,\psi)-\theta)) \sin \xi \,d\xi \\
&+
\int_0^\pi \rho_{K(0,0,\psi)}^3(P(\xi,-\theta)) \sin \xi \,d\xi
+
\int_0^\pi \rho_{K(0,0,\psi)}^3(P(\xi,\pi-\theta)) \sin \xi \,d\xi,
\end{aligned}\end{equation*}
which yields the inequality \eqref{eq:19}.
\end{proof}

\begin{proposition}
\label{prop:7}
Between $(0, \phi, \psi) \in M_1$ and $(\pi-\Theta(0,\phi,\psi),\phi,\psi) \in M_2$,
the following formulas hold:
\begin{equation}
\label{eq:8}
\begin{aligned}
F(\pi-\Theta(0,\phi,\psi),\phi,\psi)&=-F(0,\phi,\psi), \\
G(\pi-\Theta(0,\phi,\psi),\phi,\psi)&=-H(0,\phi,\psi), \\
H(\pi-\Theta(0,\phi,\psi),\phi,\psi)&=G(0,\phi,\psi).
\end{aligned}
\end{equation}
Between $(\tilde{\theta},0,\psi) \in M_3$ and $(\theta,\pi,\psi) \in M_4$,
the following formulas hold:
\begin{equation}
\label{eq:10}
\begin{aligned}
F(\theta,\pi,\psi)
&
=F(\tilde{\theta},0,\psi), \\
G(\theta,\pi,\psi)
&
=-H(\tilde{\theta},0,\psi), \\
H(\theta,\pi,\psi)
&
=-G(\tilde{\theta},0,\psi),
\end{aligned}
\end{equation}
where $\tilde{\theta} =(\Gamma_\psi)^{-1}(\pi-\theta)\in [0, \pi-\Theta(0,0,\psi)]$.
Between $(\theta,\pi-\phi,0) \in M_5$ and $(\theta,\phi,\pi) \in M_6$,
the following formulas hold:
\begin{equation}
\label{eq:11}
\begin{aligned}
F(\theta,\phi,\pi)
&=
-F(-\theta,-\phi,0)
=F(\theta,\pi-\phi,0), \\
G(\theta,\phi,\pi)
&=
G(-\theta,-\phi,0)
=-G(\theta,\pi-\phi,0), \\
H(\theta,\phi,\pi)
&=
-H(-\theta,-\phi,0)
=-H(\theta,\pi-\phi,0).
\end{aligned}
\end{equation}
\end{proposition}
\begin{proof}
First, we show \eqref{eq:8}.
We have
\begin{equation*}\begin{aligned}
 K(\pi-\Theta(\theta,\phi,\psi)+\theta, \phi,\psi)
&= X(\pi-\Theta(\theta,\phi,\psi)) X(\theta) Y(\phi) Z(\psi) K \\
&= X(\pi-\Theta(\theta,\phi,\psi)) K(\theta, \phi, \psi).
\end{aligned}\end{equation*}
Noting $\Theta(\theta,\phi,\psi)=\Theta(K(\theta,\phi,\psi))$, 
Lemma \ref{lem:5} implies
\begin{equation}
\label{eq:7}
\begin{aligned}
F(\pi-\Theta(\theta,\phi,\psi)+\theta,\phi,\psi)&=-F(\theta,\phi,\psi), \\
G(\pi-\Theta(\theta,\phi,\psi)+\theta,\phi,\psi)&=-H(\theta,\phi,\psi), \\
H(\pi-\Theta(\theta,\phi,\psi)+\theta,\phi,\psi)&=G(\theta,\phi,\psi).
\end{aligned}
\end{equation}
Putting $\theta=0$, we get \eqref{eq:8}.

Next, let us consider the case $(\theta,\pi,\psi) \in M_4$.
Since
$X(\theta) Y(\pi)=Y(\pi) X(-\theta)$,
$Y(\phi) Z(\pi)=Z(\pi) Y(-\phi)$, and 
$X(\theta) Z(\pi)=Z(\pi) X(-\theta)$,
we have
\begin{equation*}
K(\theta, \pi,\psi)
= X(\theta) Y(\pi) Z(\psi) K
= Y(\pi) X(-\theta) Y(0) Z(\psi) K
= Y(\pi) K(-\theta, 0, \psi).
\end{equation*}
By Lemma \ref{lem:8}, we have
\begin{equation*}\begin{aligned}
F(\theta,\pi,\psi)=-F(-\theta,0,\psi), \quad
G(\theta,\pi,\psi)=-G(-\theta,0,\psi), \quad
H(\theta,\pi,\psi)=H(-\theta,0,\psi).
\end{aligned}\end{equation*}
Since $F,G,H$ are $2\pi$-periodic with respect to $\theta$,
we use Lemma \ref{lem:6} to obtain
\begin{equation*}\begin{aligned}
-F(-\theta,0,\psi)
&=
-F(\pi+\pi-\theta,0,\psi)
=-F(\pi-\theta,0,\psi), \\
-G(-\theta,0,\psi)
&=
-G(\pi+\pi-\theta,0,\psi)
=
G(\pi-\theta,0,\psi), \\
H(-\theta,0,\psi)
&=
H(\pi+\pi-\theta,0,\psi)
=
-H(\pi-\theta,0,\psi).
\end{aligned}\end{equation*}
Here, by the definition of $M_4$, $0 \leq \theta \leq \pi-\Theta(0,\pi,\psi)$ holds, and hence
$\Theta(0,\pi,\psi) \leq \pi-\theta \leq \pi$.
On the other hand, by Lemma \ref{lem:7}, we have
\begin{equation}
\label{eq:44}
 \Theta(0,\pi,\psi)=\pi- \Theta(0,0,\psi).
\end{equation}
Thus $\pi-\Theta(0,0,\psi) \leq \pi-\theta \leq \pi$.
For each $\pi-\theta \in [\pi-\Theta(0,0,\psi), \pi]$,
putting $\tilde{\theta} =(\Gamma_\psi)^{-1}(\pi-\theta)\in [0, \pi-\Theta(0,0,\psi)]$
(see Lemma \ref{lem:11}), by using \eqref{eq:7}, we get
\begin{equation*}\begin{aligned}
-F(\pi-\theta,0,\psi)
&=-F(\Gamma_\psi(\tilde{\theta}),0,\psi)
=
-F(\pi-\Theta(\tilde{\theta},0,\psi)+\tilde{\theta},0,\psi)
=
F(\tilde{\theta},0,\psi), \\
G(\pi-\theta,0,\psi)
&=G(\Gamma_\psi(\tilde{\theta}),0,\psi)
=
G(\pi-\Theta(\tilde{\theta},0,\psi)+\tilde{\theta},0,\psi)
=
-H(\tilde{\theta},0,\psi), \\
-H(\pi-\theta,0,\psi)
&=-H(\Gamma_\psi(\tilde{\theta}),0,\psi)
=
-H(\pi-\Theta(\tilde{\theta},0,\psi)+\tilde{\theta},0,\psi)
=
-G(\tilde{\theta},0,\psi),
\end{aligned}\end{equation*}
which imply \eqref{eq:10}.

Finally, we consider the case $(\theta,\phi,\pi) \in M_6$.
Note that
\begin{equation*}\begin{aligned}
 K(\theta, \phi,\pi)
&= X(\theta) Y(\phi) Z(\pi) K
= X(\theta) Z(\pi) Y(-\phi) K \\
&= Z(\pi) X(-\theta) Y(-\phi) Z(0) K
= Z(\pi) K(-\theta,-\phi,0), \\
K(-\theta,-\phi,0)
&= X(-\theta) Y(2\pi-\phi) Z(0) K
= Y(\pi) X(\theta) Y(\pi-\phi) Z(0) K \\
&= Y(\pi) K(\theta, \pi-\phi, 0).
\end{aligned}\end{equation*}
Thus, by Lemmas \ref{lem:10} and \ref{lem:8}, we obtain
\begin{equation*}\begin{aligned}
F(\theta,\phi,\pi)
&=
-F(-\theta,-\phi,0)
=F(\theta,\pi-\phi,0), \\
\notag
G(\theta,\phi,\pi)
&=
G(-\theta,-\phi,0)
=-G(\theta,\pi-\phi,0), \\
H(\theta,\phi,\pi)
&=
-H(-\theta,-\phi,0)
=-H(\theta,\pi-\phi,0).
\end{aligned}\end{equation*}
Since $K(0,\phi,\pi)=Z(\pi) Y(\pi) K(0,\pi-\phi,0)$,
by the definition of $\Theta$,
we have
\begin{equation}
\label{eq:45} 
\Theta(0,\phi,\pi)=\Theta(0,\pi-\phi,0).
\end{equation}
It means that
$(\theta,\phi,\pi) \in M_6$ $(\theta \in [0,\pi-\Theta(0,\phi,\pi)])$ if and only if
$(\theta,\pi-\phi,0) \in M_5$ $(\theta \in [0,\pi-\Theta(0,\pi-\phi,0)])$.
Hence \eqref{eq:11} is verified.
\end{proof}

\section{Calculation of degree}

\label{sec:6}

\subsection{Setting}

We use new coordinates $(s, \phi, \psi)$ instead of $(\theta, \phi, \psi)$, where
\begin{equation*}
 s= \frac{\theta}{\pi-\Theta(0,\phi,\psi)}.
\end{equation*}
For the new coordinates $(s, \phi, \psi)$, we see that
\begin{equation*}\begin{aligned}
M_1 &=
\{0\} \times [0,\pi]\times[0,\pi], &
M_3 &=
[0,1]\times\{0\}\times[0,\pi], &
M_5 &=
[0,1]\times[0,\pi]\times\{0\}, \\
M_2 &=
\{1\} \times [0,\pi]\times[0,\pi], &
M_4 &=
[0,1]\times\{\pi\}\times[0,\pi], &
M_6 &=
[0,1]\times[0,\pi]\times\{\pi\}.
\end{aligned}\end{equation*}
We define
\begin{equation*}
\begin{aligned}
\hat{F}(s, \phi, \psi) &:=
F((\pi-\Theta(0,\phi,\psi))s, \phi, \psi), \\
\hat{G}(s, \phi, \psi) &:=
G((\pi-\Theta(0,\phi,\psi))s, \phi, \psi), \\
\hat{H}(s, \phi, \psi) &:=
H((\pi-\Theta(0,\phi,\psi))s, \phi, \psi).
\end{aligned}
\end{equation*}
Note that
\begin{equation*}
 D= \left\{(s, \phi, \psi) \in \R^3; 0 \leq s \leq 1, 0 \leq \phi \leq \pi, 0 \leq \psi \leq \pi\right\}
\end{equation*}
and $(\hat{F},\hat{G},\hat{H}) \in C^\infty(D, \R^3)$.

Then $(\hat{F},\hat{G},\hat{H})|_{\partial D}$ is a triplet of continuous functions on $\partial D=M_1 \cup \dots \cup M_6$, of class $C^{\infty}$ in the interior of each $M_i$ for $i=1,\dots,6$.
Hence, $(\hat{F},\hat{G},\hat{H})|_{\partial D}$ can be regarded as a piecewise smooth vector field on $\partial D$.
It follows from Proposition \ref{prop:7}, \eqref{eq:44}, and \eqref{eq:45} that
$(\hat{F},\hat{G},\hat{H})$ satisfies the identities:
\begin{equation}
\label{eq:12}
\begin{aligned}
\hat{F}(1,\phi,\psi)&=-\hat{F}(0,\phi,\psi), \\
\hat{G}(1,\phi,\psi)&=-\hat{H}(0,\phi,\psi), \\
\hat{H}(1,\phi,\psi)&=\hat{G}(0,\phi,\psi)
\end{aligned}
\end{equation}
for $\phi, \psi \in [0,\pi]$ and
\begin{equation}
\label{eq:13}
\begin{aligned}
\hat{F}(s,\pi,\psi)
&=
\hat{F}(T_\psi(s),0,\psi), \\
\hat{G}(s,\pi,\psi)
&=
-\hat{H}(T_\psi(s),0,\psi), \\
\hat{H}(s,\pi,\psi)
&=
-\hat{G}(T_\psi(s),0,\psi)
\end{aligned}
\end{equation}
for $s \in [0,1], \psi \in [0,\pi]$
and
\begin{equation}
\label{eq:14}
\begin{aligned}
\hat{F}(s,\phi,\pi)
&=
\hat{F}(s,\pi-\phi,0), \\
\hat{G}(s,\phi,\pi)
&=
-\hat{G}(s,\pi-\phi,0), \\
\hat{H}(s,\phi,\pi)
&=
-\hat{H}(s,\pi-\phi,0)
\end{aligned}
\end{equation}
for $s \in [0,1], \phi \in [0,\pi]$,
where $T_\psi:[0,1] \to [0,1]$ is defined by 
\begin{equation*}
 T_\psi(s) :=
\frac{\Gamma_\psi^{-1}\left(\pi-\Theta(0,0,\psi)s\right)}{\pi-\Theta(0,0,\psi)}.
\end{equation*}
By Lemma \ref{lem:11}, $T_\psi(s)$ is smooth with respect to $\psi$ and $s$, 
decreasing with respect to $s$, $T_\psi(0)=1$, and $T_\psi(1)=0$.

As we mentioned in Section \ref{sec:5.3}, if we can find a zero $(\theta,\phi,\psi)$
in $D$ of the vector field $(F,G,H)$, then the equations \eqref{eq:3} in Section \ref{sec:3.5} are satisfied and the proof of 3-dimensional symmetric Mahler conjecture is completed.
In Sections \ref{sec:6.2} and \ref{sec:6.3}, we consider the case that $(F,G,H)$ has no zeros on the boundary $\partial D$.
Now, assume that $(F,G,H) \not=0$ on $\partial D$, that is, $(\hat{F},\hat{G},\hat{H}) \not=0$ on $\partial D$.
Then we shall calculate the degree of
\begin{equation*}
 \mathscr{F}:=\frac{(\hat{F},\hat{G},\hat{H})}{\sqrt{\hat{F}^2+\hat{G}^2+\hat{H}^2}}:\partial D \to S^2.
\end{equation*}

\subsection{Degree of a map}
\label{sec:6.1+}

The degree of a map is one of the basic tools to find a zero of a vector field.
Here we briefly review the necessary facts about the degree of maps (see e.g., \cite{OR}).

Let $M,N$ be connected oriented smooth manifolds of dimension $n$.
For a smooth map $f:M \to N$, the integer
\begin{equation}
\label{eq:BKdeg}
 \sum_{x \in f^{-1}(a)}\sgn \det(d f)_x
\end{equation}
does not depend on the choice of the regular value $a \in N$, which is called {\it the Brouwer-Kronecker degree} of $f$ and denoted by $\deg(f)$.
Since $\deg(f)$ has the property of the homotopy invariance, the notion can be defined for continuous maps.
Note that since $\deg(f)$ can be described in terms of integration of differential forms \cite[Corollary 2.4]{OR}, we can also use it when $M$ is piecewise smooth and $f:M \to N$ is continuous and smooth on the smooth part of $M$ as long as $f^{-1}(a)$ is contained in the smooth part of $M$.
In our setting, $M=\partial D \subset \R^3$ and $f=\mathscr{F}:\partial D \to S^2$ is continuous, and smooth on the interior of $M_i \ (i=1,\ldots,6)$.

Another notion of degree we will need below is the Euclidean degree.
Let $D_0 \subset \R^{m+1}$ be a bounded open set.
Let $g:D_0 \cup \partial D_0 \to \R^{m+1}$ be a smooth map.
If $a \in \R^{m+1} \setminus g(\partial D_0)$ is a regular value of $g|_{D_0}$, then $g^{-1}(a)$ is a finite set and
\begin{equation}
\label{eq:Edeg}
 d(g,D_0,a) := \sum_{x \in g^{-1}(a)}\sgn \det\left(\frac{\partial g_i}{\partial x_j}(x)\right)
\end{equation}
is called {\it the Euclidean degree} of $f$.
This is constant for all regular values contained in the same connected component of $\R^{m+1} \setminus g(\partial D_0)$.

In our setting, the key idea for proving that $\deg(\mathscr{F}) \neq 0$ is to reduce the calculation to the case of the Euclidean degree of $g=(\hat{G},\hat{H}):M_1 \to \R^2$ (see Proposition \ref{prop:5}).
And the latter is reduced to the calculation of the winding number of it (see Proposition \ref{prop:6}).

\subsection{Reduction}
\label{sec:6.2}

For notational convenience, in Sections \ref{sec:6.2}, \ref{sec:6.3}, and Appendix \ref{sec:A}, we denote the above $\hat{F}$, $\hat{G}$, and $\hat{H}$ by $F$, $G$, and $H$, respectively, and consider
\begin{equation*}
 \mathscr{F}=\frac{(F,G,H)}{\sqrt{F^2+G^2+H^2}}:\partial D \to S^2.
\end{equation*}

In order to calculate $\deg \mathscr{F}$ by using \eqref{eq:BKdeg}, we first need to modify the map $\mathscr{F}$ so as to satisfy the property that $(\pm 1,0,0)$ are its regular values.
\begin{proposition}
\label{prop:4}
There exists a map $\tilde{\mathscr{F}}: \partial D \to S^2 \subset \R^3$ satisfying the following properties \textup{(i)--(iv)}:
\begin{enumerate}
 \item $\tilde{\mathscr{F}}$ is homotopic to $\mathscr{F}$.
Especially, $\deg \tilde{\mathscr{F}} = \deg \mathscr{F}$ holds.
 \item $\tilde{\mathscr{F}}(\partial M_i) \not\ni (\pm 1,0,0)$ for $i=1,\dots,6$.
 \item $(\pm 1,0,0)$ are regular values of $\tilde{\mathscr{F}}$.
 \item $\tilde{\mathscr{F}}$ satisfies \eqref{eq:12}, \eqref{eq:13}, and \eqref{eq:14}.
\end{enumerate}
\end{proposition}
We give a proof of Proposition \ref{prop:4} in Appendix \ref{sec:A}.
Owing to Proposition \ref{prop:4},
in what follows we may assume that
\begin{equation}
\label{eq:46}
\mathscr{F}(\partial M_i) \not\ni (\pm 1,0,0) \text{ for } i=1,\dots,6, 
\end{equation}
and $(\pm 1,0,0)$ are regular values of $\mathscr{F}$.
Note that  $\mathscr{F}(s, \phi, \psi)=(\pm 1,0,0)$ if and only if $G(s, \phi, \psi)=H(s, \phi, \psi)=0$ and $\sgn F(s, \phi, \psi)=\pm 1$.
Hence, the inverse image of $(\pm 1,0,0)$ by $\mathscr{F} \rvert_{M_1}$ corresponds
to the zeros of the vector field $(G,H)$ on $M_1$.
By the condition \eqref{eq:46}, $(G,H) \not= (0,0)$ on $\partial M_1$ and the map
\begin{equation*}
\mathscr{G}:=
 \frac{(G,H)}{\sqrt{G^2+H^2}} : \partial M_1 \rightarrow S^1
\end{equation*}
is well-defined.
\begin{proposition}
\label{prop:5}
For the above maps
$\mathscr{F}:\partial D \rightarrow S^2$ and $\mathscr{G}: \partial M_1 \rightarrow S^1 \subset \R^2$, we have
\begin{equation*}
\deg \mathscr{F} = w(\mathscr{G},(0,0)) \pmod{2},
\end{equation*}
where $w(\mathscr{G},(0,0))$ is the winding number.
\end{proposition}

\begin{proof}
Since $(1,0,0)$ is a regular value of $\mathscr{F}$, by \eqref{eq:BKdeg} we have
\begin{equation*}
\deg \mathscr{F}
=
\sum_{P \in \mathscr{F}^{-1}(1,0,0)} \sgn \det \left(d \mathscr{F}\right)_P.
\end{equation*}
Since the inverse image $\mathscr{F}^{-1}(1,0,0)$ is a finite set,
we find out the parity of $\deg \mathscr{F}$ by counting the elements.
For each $i=1,\ldots,6$,
denote by $k_i$ the cardinality of
$\mathscr{F}^{-1}(1,0,0) \cap \interior M_i$, that is
\begin{equation}
\label{eq:50}
k_i := \# \left\{
(s,\phi,\psi) \in \interior M_i; \mathscr{F}(s,\phi,\psi)=(1,0,0)
\right\}.
\end{equation}
Similarly $k_-$ denotes the cardinality of $\mathscr{F}^{-1}(-1,0,0) \cap \interior M_1$, that is
\begin{equation}
\label{eq:51}
k_- := \# \left\{
(s,\phi,\psi) \in \interior M_1; \mathscr{F}(s,\phi,\psi)=(-1,0,0)
\right\}.
\end{equation}
By \eqref{eq:46},
we have $\deg \mathscr{F}=k_1 + \dots +k_6 \pmod{2}$.

If $P \in \mathscr{F}^{-1}(1,0,0) \cap M_4$,
that is $P=(s, \pi, \psi) \in \mathscr{F}^{-1}(1,0,0)$ for some $s \in [0,1]$ and $\psi \in [0, \pi]$, then for $\tilde{P}=(T_\psi(s),0,\psi) \in M_3$
we have $\tilde{P} \in \mathscr{F}^{-1}(1,0,0)$ by \eqref{eq:13}.
Since $(s, \pi, \psi) \mapsto (T_\psi(s),0,\psi)$ is a $C^{\infty}$-diffeomorphism from $\interior M_4$ to $\interior M_3$, 
we see $k_4 \leq k_3$.
Using the same argument for $P \in M_3 \cap \mathscr{F}^{-1}(1,0,0)$, 
we obtain $k_3 \leq k_4$, and hence $k_3=k_4$.
Similarly, using the correspondence between $M_5$ and $M_6$ by \eqref{eq:14}, we get $k_5=k_6$.
Furthermore, by \eqref{eq:12} the set
$\mathscr{F}^{-1}(1,0,0) \cap M_2$ exactly corresponds to $\mathscr{F}^{-1}(-1,0,0) \cap M_1$, which yields $k_2 = k_-$.
Consequently, we obtain
\begin{equation*}
 k_1 + \dots + k_6 = k_1 + k_- + 2(k_3 + k_5)
\end{equation*}
and
\begin{equation*}
\deg \mathscr{F} = k_1 + k_- \pmod{2}.
\end{equation*}
By \eqref{eq:50} and \eqref{eq:51}, $k_1 + k_-$
can be viewed as the number of zeros of the vector field $(G,H)$ on $M_1$.
Note that $P \in \interior M_1$ is a regular point of $(G,H)$ with $(G, H)(P)=(0,0)$ 
if and only if 
$P \in \interior M_1$ is a regular point of $\mathscr{F}$ with $\mathscr{F}(P)=(\pm 1, 0,0)$.
Since $(G,H) \not= (0,0)$ on $\partial M_1$, $(0,0)$ is a regular value of $(G,H)$ because
$(\pm 1,0,0)$ are regular values of $\mathscr{F}$.
Thus, by \eqref{eq:Edeg} we have
\begin{equation*}
d((G,H), \interior M_1, (0,0)) = \sum_{(G,H)(P)=(0,0)} \sgn \det (d (G,H))_P = k_1 + k_- \pmod{2}.
\end{equation*}
Moreover, it is known that
\begin{equation*}
 d((G,H), \interior M_1, (0,0)) = w((G,H),(0,0)),
\end{equation*}
where the right-hand side is the winding number of $(G,H): \partial M_1 \rightarrow \R^2 \setminus \{(0,0)\}$ around $(0,0)$.
Since $(G,H): \partial M_1 \rightarrow \R^2 \setminus \{(0,0)\}$ is homotopic to 
\begin{equation*}
 \mathscr{G}=\frac{(G,H)}{\sqrt{G^2+H^2}}: \partial M_1 \rightarrow S^1 \subset \R^2 \setminus \{(0,0)\},
\end{equation*}
we have $w((G,H),(0,0))=w(\mathscr{G},(0,0))$, which completes the proof.
\end{proof}

\subsection{Calculation of $w(\mathscr{G},(0,0))$}
\label{sec:6.3}

Proposition \ref{prop:5} reduces the problem to the following
\begin{proposition}
\label{prop:6}
$w(\mathscr{G},(0,0))$ is odd.
\end{proposition}

\begin{proof}
To calculate the winding number,
we represent each point $P_t$ on $\partial M_1$ as 
\begin{equation*}
 P_t =
\begin{cases}
 (0,t,0) & \text{ if } 0 \leq t \leq \pi, \\
 (0,\pi,t-\pi) & \text{ if } \pi \leq t \leq 2 \pi, \\
 (0,3\pi-t,\pi) & \text{ if } 2 \pi \leq t \leq 3 \pi, \\
 (0,0,4 \pi-t) & \text{ if } 3 \pi \leq t \leq 4 \pi,
\end{cases}
\end{equation*}
which parametrizes $\partial M_1$.
Put
\begin{equation*}
\mathscr{G}(P_t)=(\tilde{G}(P_t), \tilde{H}(P_t))
:=
\frac{(G(P_t), H(P_t))}
{\sqrt{G^2(P_t) + H^2(P_t)}}
=(\cos s, \sin s).
\end{equation*}
Then $s$ is a piecewise smooth function of $t$.
We have
\begin{equation*}
w(\mathscr{G},(0,0))= \frac{1}{2 \pi} \int_0^{4 \pi}W(t) \,dt,
\end{equation*}
where $W(t):=
\tilde{G}(P_t)
\tilde{H}'(P_t)
-
\tilde{G}'(P_t)
\tilde{H}(P_t).
$
First,
by \eqref{eq:12} and \eqref{eq:13}, we have
\begin{equation*}\begin{aligned}
G(0,\pi,\psi)
&=
-H(1,0,\psi)
=-G(0,0,\psi), \\
H(0,\pi,\psi)
&=
-G(1,0,\psi) =
H(0,0,\psi), \\
(G^2+H^2)(0,\pi, \psi)
&=
(G^2+H^2)(0,0, \psi).
\end{aligned}\end{equation*}
Thus
\begin{equation*}
\tilde{G}(0,\pi,\psi)=-\tilde{G}(0,0,\psi), \quad
\tilde{H}(0,\pi,\psi)=\tilde{H}(0,0,\psi),
\end{equation*}
which yield that, for $3 \pi \leq t \leq 4 \pi$,
\begin{equation*}\begin{aligned}
\tilde{G}(P_t)
&=
\tilde{G}(0,0,4\pi-t)
=
-\tilde{G}(0,\pi,4\pi-t), \\
\tilde{H}(P_t)
&=
\tilde{H}(0,0,4\pi-t)
=
\tilde{H}(0,\pi,4\pi-t).
\end{aligned}\end{equation*}
So, we obtain
\begin{equation*}\begin{aligned}
&\int_{3 \pi}^{4 \pi} W(t) \,dt \\
&=
\int_{3 \pi}^{4 \pi}
\tilde{G}(0,\pi,4\pi-t)
\tilde{H}_\psi(0,\pi,4\pi-t)
-
\tilde{G}_\psi(0,\pi,4\pi-t)
\tilde{H}(0,\pi,4\pi-t) \,dt \\
&=
-
\int_{2 \pi}^{\pi}
\tilde{G}(0,\pi,\tau-\pi)
\tilde{H}_\psi(0,\pi,\tau-\pi)
-
\tilde{G}_\psi(0,\pi,\tau-\pi)
\tilde{H}(0,\pi,\tau-\pi) \,d\tau,
\end{aligned}\end{equation*}
where we used the substitution $4 \pi-t=\tau-\pi$.
On the other hand,
for $\pi \leq t \leq 2 \pi$,
we have $\tilde{G}(t)=\tilde{G}(0,\pi,t-\pi)$, $\tilde{H}(t)=\tilde{H}(0,\pi,t-\pi)$.
So, we obtain
\begin{equation*}\begin{aligned}
&\int_{\pi}^{2\pi} W(t) \,dt \\
&=
\int_{\pi}^{2\pi}
\tilde{G}(0,\pi,t-\pi)
\tilde{H}_\psi(0,\pi,t-\pi)
-
\tilde{G}_\psi(0,\pi,t-\pi)
\tilde{H}(0,\pi,t-\pi) \,dt.
\end{aligned}\end{equation*}
Hence
\begin{equation*}
\int_{3 \pi}^{4 \pi} W(t) \,dt
=
\int_{\pi}^{2\pi} W(t) \,dt.
\end{equation*}

Next, by \eqref{eq:14}, we have
\begin{equation*}\begin{aligned}
G(0,\phi,\pi)
&=
-G(0,\pi-\phi,0), \\
H(0,\phi,\pi)
&=
-H(0,\pi-\phi,0), \\
(G^2+H^2)(0,\phi,\pi)
&=
(G^2+H^2)(0,\pi-\phi,0).
\end{aligned}\end{equation*}
Thus
\begin{equation}
\label{eq:18}
\tilde{G}(0,\phi,\pi)=-\tilde{G}(0,\pi-\phi,0), \quad
\tilde{H}(0,\phi,\pi)=-\tilde{H}(0,\pi-\phi,0),
\end{equation}
which yield that, for $2 \pi \leq t \leq 3 \pi$,
\begin{equation*}\begin{aligned}
\tilde{G}(P_t)
&=
\tilde{G}(0,3\pi-t,\pi)
=
-\tilde{G}(0,t-2\pi,0), \\
\tilde{H}(P_t)
&=
\tilde{H}(0,3\pi-t,\pi)
=
-\tilde{H}(0,t-2\pi,0).
\end{aligned}\end{equation*}
Therefore,
\begin{equation*}\begin{aligned}
&\int_{2 \pi}^{3 \pi} W(t) \,dt \\
&=
\int_{2 \pi}^{3 \pi}
\tilde{G}(0,t-2\pi,0)
\tilde{H}_\phi(0,t-2\pi,0)
-
\tilde{G}_\phi(0,t-2\pi,0)
\tilde{H}(0,t-2\pi,0)
\,dt \\
&=
\int_{0}^{\pi}
\tilde{G}(0,\tau,0)
\tilde{H}_\phi(0,\tau,0)
-
\tilde{G}_\phi(0,\tau,0)
\tilde{H}(0,\tau,0)
\,d\tau,
\end{aligned}\end{equation*}
where we used the substitution $\tau=t-2\pi$.
On the other hand, for $0 \leq t \leq \pi$, 
$\tilde{G}(P_t)=\tilde{G}(0,t,0)$, $\tilde{H}(P_t)=\tilde{H}(0,t,0)$.
Thus,
\begin{equation*}
\int_{0}^{\pi} W(t) \,dt =
\int_{0}^{\pi}
\tilde{G}(0,t,0)
\tilde{H}_\phi(0,t,0)
-
\tilde{G}_\phi(0,t,0)
\tilde{H}(0,t,0) \,dt.
\end{equation*}
Hence,
\begin{equation*}
\int_{2 \pi}^{3 \pi} W(t) \,dt
=
\int_{0}^{\pi} W(t) \,dt.
\end{equation*}
Consequently, we get
\begin{equation*}
 \int_0^{4 \pi} W(t) \,dt = 2 \int_0^{2 \pi} W(t) \,dt.
\end{equation*}

Since
$\tilde{G}(P_t)=\cos s$, $\tilde{H}(P_t)=\sin s$, and
$W(t)=\frac{d s}{d t}$, we have
\begin{equation*}
2 \int_0^{2 \pi} W(t) \,dt
=
2 \int_{s_0}^{s_1} ds =2 (s_1-s_0),
\end{equation*}
where $s_0$ and $s_1$ are determined by
\begin{equation*}
\begin{pmatrix}
\cos s_0 \\ \sin s_0
\end{pmatrix}
=
\begin{pmatrix}
\tilde{G}(P_0) \\ \tilde{H}(P_0)
\end{pmatrix}, \quad
\begin{pmatrix}
\cos s_1 \\ \sin s_1
\end{pmatrix}
=
\begin{pmatrix}
\tilde{G}(P_{2\pi}) \\ \tilde{H}(P_{2\pi})
\end{pmatrix},
\end{equation*}
respectively.
By \eqref{eq:18}, we see
\begin{equation*}
\begin{pmatrix}
\tilde{G}(P_{2\pi}) \\ \tilde{H}(P_{2\pi})
\end{pmatrix}
=
\begin{pmatrix}
\tilde{G}(0,\pi,\pi) \\ \tilde{H}(0,\pi,\pi)
\end{pmatrix}
=
-
\begin{pmatrix}
\tilde{G}(0,0,0) \\ \tilde{H}(0,0,0)
\end{pmatrix}
=
-
\begin{pmatrix}
\tilde{G}(P_0) \\ \tilde{H}(P_0)
\end{pmatrix}.
\end{equation*}
This means that there exists $m \in \Z$ such that $s_1=s_0+(2m+1)\pi$.
Hence
\begin{equation*}
\int_0^{4 \pi} W(t) \,dt
=
 2(2m+1)\pi,
\end{equation*}
which yields $w(\mathscr{G},(0,0))=2m+1$.
\end{proof}

\subsection{Proof of the main theorem (inequality part)}

Let us prove the inequality part of Theorem \ref{thm:main}.
A convex body $K \in \K_0^3$ can be approximated by an element of $\hatK$ as follows (see e.g., \cite[pp.\,38]{Mi}, \cite[pp.\,438]{Sc2}).

\begin{proposition}
\label{prop:13}
For any centrally symmetric convex body $K \in \K_0^3$ and any $\varepsilon > 0$, there exists a convex body $L \in \K^3$ with the following properties:
\begin{itemize}
 \item[\upshape{(a)}] $\delta(K,L) < \varepsilon$,
 \item[\upshape{(b)}] $L$ has $C^\infty$ support function and is strongly convex,
 \item[\upshape{(c)}] $L$ is also centrally symmetric,
\end{itemize}
where $\delta$ denotes the Hausdorff distance on $\K^3$.
In particular, $L \in \hatK$.
\end{proposition}

Since the volume is continuous with respect to the Hausdorff distance,
it suffices to show the inequality $\eqref{eq:1}$ for $K \in \hatK$.
Fix $K \in \hatK$. We show the following
\begin{claim}
 $(\hat{F}, \hat{G}, \hat{H})$ has a zero on $D$.
\end{claim}

\begin{proof}
In the case where $(\hat{F}, \hat{G}, \hat{H})$ has a zero on $\partial D$, the claim trivially holds.
If not, 
\begin{equation*}
\mathscr{F}= \frac{(\hat{F}, \hat{G}, \hat{H})}{\sqrt{\hat{F}^2+\hat{G}^2+\hat{H}^2}}: \partial D \rightarrow S^2
\end{equation*}
is well-defined and $\deg \mathscr{F}$ is odd
by Propositions \ref{prop:5} and \ref{prop:6}.
By, for example, \cite{OR}*{pp.\,157--158, Propositions 4.4 and 4.6},
$(\hat{F},\hat{G},\hat{H})$ has a zero in $D$.
\end{proof}
We denote the zero of $(\hat{F},\hat{G},\hat{H})$ by $(s_0,\phi_0,\psi_0)$.
Putting $\theta_0=(\pi-\Theta(0,\phi_0,\psi_0))s_0$,
it means that 
$(F,G,H)=(0,0,0)$ at $(\theta_0,\phi_0,\psi_0)$.
Thus the convex body
\begin{equation*}
\linearop K(\theta_0,\phi_0,\psi_0)=\mathcal{A}(X(\theta_0) Y(\phi_0) Z(\psi_0)K) X(\theta_0) Y(\phi_0) Z(\psi_0)K
\end{equation*}
satisfies the key equality \eqref{eq:3}.
Since $\mathcal{A}(X(\theta_0) Y(\phi_0) Z(\psi_0)K) X(\theta_0) Y(\phi_0) Z(\psi_0)$ is a linear transformation, 
by Proposition \ref{prop:2} we obtain
\begin{equation*}
\mathcal{P}(K)
=
 \mathcal{P}\left(
\linearop K(\theta_0,\phi_0,\psi_0)
\right)
\geq \frac{32}{3}.
\end{equation*}
Consequently, every $K \in \hatK$ satisfies \eqref{eq:1}.
Thus we complete the proof of the Mahler conjecture
for any centrally symmetric convex body $K \in \K_0^3$.

\section{The equality case}
\label{sec:7}

In this section, we show the equality part of Theorem \ref{thm:main}. That is,
\begin{theorem}
\label{prop:14}
Let $K$ be a centrally symmetric convex body in $\R^3$ with $\mathcal{P}(K)=32/3$.
Then, $K$ or $K^\circ$ is a parallelepiped.
\end{theorem}

\subsection{Dual face}
\label{sec:7.1}

Let $K \in \K^n_0$ be a centrally symmetric convex polytope in $\R^n$.
Then $K^\circ$ is also a centrally symmetric convex polytope.
For a face $F$ of $K$, its {\it dual face} $F^\circ \subset K^\circ$ is defined by
\begin{equation*}
 F^\circ:= \left\{
Q \in K^\circ; P \cdot Q=1 \text{ for any } P \in F
\right\}.
\end{equation*}
For a $k$-dimensional face $F$, the dimension of $F^\circ$ is $n-k-1$.

Assume $K \in \K^2_0$ is a polygon.
If a face $F \subset K$ is a line segment which contains two points $(a,b)$ and $(c,d)$,
then its dual face $F^\circ$ is a vertex of $K^\circ$.
Putting $F^\circ=\{(\alpha,\beta)\}$,
by the definition of $F^\circ$, we obtain that
\begin{equation*}
\begin{pmatrix}
a & b \\
c & d
\end{pmatrix}
\begin{pmatrix}
 \alpha \\ \beta
\end{pmatrix}
=
\begin{pmatrix}
 1 \\ 1
\end{pmatrix}.
\end{equation*}
Hence, 
\begin{equation*}
\begin{pmatrix}
 \alpha \\ \beta
\end{pmatrix}
=
\begin{pmatrix}
a & b \\
c & d
\end{pmatrix}^{-1}
\begin{pmatrix}
 1 \\ 1
\end{pmatrix}
=
\frac{1}{ad-bc}
\begin{pmatrix}
 d -b \\ a-c
\end{pmatrix}.
\end{equation*}

Assume $K \in \K^3_0$ is a polytope.
If a $2$-dimensional face $F$ of $K$ contains three points $(x_i, y_i, z_i)$\, $(i=1,2,3)$ which do not on the same straight line,
then its dual face $F^\circ$ is a vertex of $K^\circ$.
Putting $F^\circ=\{(u,v,w)\}$, we similarly obtain
\begin{equation*}
\begin{pmatrix}
 x_1 & y_1 & z_1 \\
 x_2 & y_2 & z_2 \\
 x_3 & y_3 & z_3
\end{pmatrix} 
\begin{pmatrix}
 u \\ v \\ w
\end{pmatrix}
=
\begin{pmatrix}
 1 \\ 1 \\ 1
\end{pmatrix}.
\end{equation*}
Hence,
\begin{equation*}
\begin{pmatrix}
 u \\ v \\ w
\end{pmatrix}
=
\begin{pmatrix}
 x_1 & y_1 & z_1 \\
 x_2 & y_2 & z_2 \\
 x_3 & y_3 & z_3
\end{pmatrix}^{-1}
\begin{pmatrix}
 1 \\ 1 \\ 1
\end{pmatrix}.
\end{equation*}

\subsection{Inequality}

For a general convex body $K$, we need the inequality similarly as Proposition \ref{prop:1}.
Let $A_i$ $(i=1,2,3)$ be points with $A_1 \cdot (A_2 \times A_3)>0$, and consider
$C_{i,j}=\mathcal{C}_K(A_i, A_j)$ for $(i,j)=(1,2),(2,3),(3,1)$.
\begin{lemma}
\label{lem:17}
Assume $K \in \K^3_0$.
Let $\mathcal{S}_K(C_{1,2}\cup C_{2,3} \cup C_{3,1})$ be a triangle on $\partial K$.
For any $P \in K$, we have
\begin{equation}
\label{eq:48}
\frac{1}{3}
\left(
P \cdot \overline{C_{1,2}} + P \cdot \overline{C_{2,3}} + P \cdot \overline{C_{3,1}} 
\right)
\leq \abs{O*\mathcal{S}_K(C_{1,2}\cup C_{2,3} \cup C_{3,1})},
\end{equation}
where 
$\overline{C_{i,j}}=(\overline{C_{i,j,1}},\overline{C_{i,j,2}},\overline{C_{i,j,3}})$ and
$\overline{C_{i,j,k}}=\sgn((A_i \times A_j)_k) |O*P_k(C_{i,j})|$.
\end{lemma}

\begin{proof}
Since \eqref{eq:48} depends on $K$ continuously, 
by the Hausdorff approximation, we may assume $K \in \hatK$.
By Proposition \ref{prop:1}, it is sufficient to show that 
\begin{equation}
\label{eq:49}
\int_{C_{1,2}} (P \times \vr) \cdot d\vr =
2 P \cdot \overline{C_{1,2}}.
\end{equation}
The curve $C_{1,2}$ is parametrized by
$l(t):= \rho_K((1-t)A_1 + t A_2) ((1-t)A_1 + t A_2)$ $(t \in [0,1])$.
By a direct calculation,
\begin{equation*}\begin{aligned}
&l(t) \times l'(t) \\
&=
\rho_K((1-t)A_1 + t A_2) ((1-t)A + t A_2) \\
&
\times
\left(
\left(
\nabla \rho_K((1-t)A_1 + t A_2) \cdot (A_2-A_1)
\right)((1-t)A_1 + t A_2) 
+
\rho_K((1-t)A_1 + t A_2) (A_2-A_1)
\right) \\
&=
\rho_K^2 ((1-t)A_1 + t A_2) (A_1 \times A_2).
\end{aligned}\end{equation*}
Thus, we have
\begin{equation*}
 \int_{C_{1,2}} (P \times \vr) \cdot d\vr
= P \cdot \int_0^1 (l(t) \times l'(t)) \,dt
= P \cdot (A_1 \times A_2) \int_0^1 \rho_K^2 ((1-t)A_1 + t A_2) \,dt
\end{equation*}
Since the absolute value of $\left(\int_0^1 l \times l' \,dt\right)_k$ is equal to $2 \abs{O*P_k(C_{1,2})}$ and $\rho_K((1-t)A_1 + t A_2)>0$,
we obtain \eqref{eq:49}.
\end{proof}

\subsection{Linear transformation of $K \in \K_0^3$}

To prove Theorem \ref{prop:14}, for a general centrally symmetric convex body $K$, we need to find a transformation $\mathcal{A}$ such that $\mathcal{A}K$ satisfies \eqref{eq:3}.
\begin{proposition}
\label{prop:10}
Let $K \in \K_0^3$ be a convex body which satisfies $\mathcal{P}(K)=32/3$.
Then there exists a linear transformation $\mathcal{B} \in GL(3)$ such that 
$L=\mathcal{B}K$ satisfies $\mathcal{P}(L)=32/3$, the condition \eqref{eq:3},
and $\pm (1,0,0), \pm (0,1,0), \pm (0,0,1) \in \partial L$.
\end{proposition}

\begin{proof}
Fix $K \in \K^3_0$.
By Proposition \ref{prop:13}, there exists $(K_m)_{m=1}^\infty \subset \hatK$
such that $K_m \rightarrow K$ as $m \rightarrow \infty$.
By the results in Sections \ref{sec:5} and \ref{sec:6}, for each $m \in \N$, 
there exists a linear transformation $B_m=
\mathcal{A}(K_m(\theta_m,\phi_m,\psi_m))X(\theta_m)Y(\phi_m)Z(\psi_m)
$ such that $B_m K_m$
satisfies the condition \eqref{eq:3}.
Hence, 
\begin{equation}
\begin{aligned}
\label{eq:52} 
& \abs{(B_m K_m)_{\pm\pm\pm}} =\frac{\abs{B_m K_m}}{8}, \\
& 
\overline{\alpha}(B_m K_m) =\overline{a}(B_m K_m), \quad
\overline{\beta}(B_m K_m) =\overline{b}(B_m K_m), \quad
\overline{\gamma}(B_m K_m) =\overline{c}(B_m K_m).
\end{aligned}
\end{equation}
Note that $(\theta_m)_{m=1}^\infty, (\phi_m)_{m=1}^\infty, (\psi_m)_{m=1}^\infty \subset [0,\pi]$ are bounded sequences.
\begin{claim}
 $(\mathcal{A}(K_m))_{m=1}^\infty$ and $(\mathcal{A}(K_m)^{-1})_{m=1}^\infty$ are bounded sequences.
\end{claim}
\begin{proof}
By the definition of $\mathcal{A}$ in Section \ref{sec:5.2}, 
it is sufficient to show that 
\begin{equation*}
\left(
\frac{1}{\tan \Theta_m}
\right)_{m=1}^\infty, \quad
\left(
\frac{1}{\tan \Phi_m}
\right)_{m=1}^\infty, \quad
\left(
\frac{1}{\sin \Theta_m \tan \Psi_m}
\right)_{m=1}^\infty
\end{equation*}
are bounded, where 
$\Theta_m=\Theta(\theta_m,\phi_m,\psi_m)$,
$\Phi_m=\Phi(\theta_m,\phi_m,\psi_m)$,
$\Psi_m=\Psi(\theta_m,\phi_m,\psi_m)$.

For simplicity, we put $\tilde{K}^m:=K_m(\theta_m,\phi_m,\psi_m)$.
Since $(K_m)_{m=1}^\infty$ converges to $K$,
there exists $r>1$ independent of $m$ such that
\begin{equation*}
B \left(0,\frac{1}{r}\right) \subset K_m \subset B(0,r),
\end{equation*}
where $B(0,r)$ is the open ball with center $0$ and radius $r$.
Since each $\tilde{K}^m$ is the image of $K_m$ by a rotation, the same inequality holds for $\tilde{K}^m$.
Hence, 
\begin{equation*}
 \frac{1}{r} \leq \rho_{\tilde{K}^m}(P) \leq r \ \text{ if } \ P \in S^2 \subset \R^3.
\end{equation*}
Thus, by the definition \eqref{eq:21} of $\Theta$, we have
\begin{equation*}
2 \Theta_m \frac{1}{r^3} \leq 2 (\pi- \Theta_m) r^3, \quad
2 \Theta_m r^3 \geq 2 (\pi- \Theta_m) \frac{1}{r^3}.
\end{equation*}
Hence, 
\begin{equation*}
\frac{1/r^3}{r^3+1/r^3} \pi \leq \Theta_m \leq \frac{r^3}{r^3+1/r^3} \pi.
\end{equation*}
Since
\begin{equation*}
0< \frac{1/r^3}{r^3+1/r^3} <\frac{r^3}{r^3+1/r^3} < 1,
\end{equation*}
we obtain that $(1/\tan \Theta_m)_{m=1}^\infty$ and $(1/\sin \Theta_m)_{m=1}^\infty$ are bounded.
Similarly, by the definitions of $\Phi$ and $\Psi$, we have
\begin{equation*}
\frac{1/r^2}{r^2+1/r^2} \pi \leq \Phi_m \leq \frac{r^2}{r^2+1/r^2} \pi, \quad
\frac{1/r^2}{r^2+1/r^2} \pi \leq \Psi_m \leq \frac{r^2}{r^2+1/r^2} \pi.
\end{equation*}
Hence, the claim holds.
\end{proof}

Thus, taking subsequence if necessary, there exists a linear transformation $B \in GL(3)$ such that 
$\lim_{m \rightarrow \infty}B_m=B$ in $GL(3)$, and $\lim_{m \rightarrow \infty} B_m K_m= BK$ in the Hausdorff distance.
Thus $\lim_{m \rightarrow \infty} \rho_{B_m K_m} = \rho_{BK}$ uniformly in $S^2 \subset \R^3$.
Since $(B_m K_m)_{+++}$ and $\overline{\alpha}(B_m K_m)$, etc. are expressed by the integral of $\rho_{B_m K_m}$ on closed subsets on $S^2$, taking as $m \rightarrow \infty$ in \eqref{eq:52}, 
we obtain that
\begin{equation*}\begin{aligned}
& \abs{(B K)_{\pm\pm\pm}} =\frac{\abs{B K}}{8}, \\
& 
\overline{\alpha}(B K) =\overline{a}(B K), \quad
\overline{\beta}(B K) =\overline{b}(B K), \quad
\overline{\gamma}(B K) =\overline{c}(B K).
\end{aligned}\end{equation*}
Moreover, since each $B_m K_m$ satisfies \eqref{eq:3}, we obtain
\begin{equation*}\begin{aligned}
& \frac{9}{4} |B_m K_m| |(B_m K_m)^\circ| \\
& \geq
|Q_1(B_m K_m)|\, |P_1((B_m K_m)^\circ)| \\
&\quad+
|Q_2(B_m K_m)|\, |P_2((B_m K_m)^\circ)|+
|Q_3(B_m K_m)|\, |P_3((B_m K_m)^\circ)|
\end{aligned}\end{equation*}
as in Section \ref{sec:3.5}.
Since $\lim_{m \rightarrow \infty} (B_m K_m)^\circ=(BK)^\circ$ in the Hausdorff distance,
\begin{equation*}\begin{aligned}
&\frac{9}{4}\mathcal{P}(BK)= \frac{9}{4} |B K| |(B K)^\circ| \\
&\geq
|Q_1(B K)|\, |P_1((B K)^\circ)|+
|Q_2(B K)|\, |P_2((B K)^\circ)|+
|Q_3(B K)|\, |P_3((B K)^\circ)|
\end{aligned}\end{equation*}
holds.
By Theorem \ref{prop:12},
$|Q_i(B K)|\, |P_i((B K)^\circ)| \geq 8$ $(i=1,2,3)$.
Since $\mathcal{P}(BK)=\mathcal{P}(K)=32/3$, we obtain $|Q_i(B K)|\, |P_i((B K)^\circ)|=8$ $(i=1,2,3)$.
If necessary, we choose a diagonal matrix $C$ to satisfy that $\pm (1,0,0), \pm (0,1,0), \pm (0,0,1) \in \partial (CBK)$.
Then, $\mathcal{B}=CB$ satisfies the required conditions.
\end{proof}

\subsection{The equality case}

\newcommand{\Olda}{\alpha}
\newcommand{\Oldb}{\beta}
\newcommand{\Oldc}{\gamma}
\newcommand{\Newa}{{a}}
\newcommand{\Newb}{{b}}
\newcommand{\Newc}{{c}}

\begin{proposition}
\label{prop:11}
Let $L \in \K_0^3$ with $\mathcal{P}(L)=32/3$.
Assume that $L$ satisfies the condition \eqref{eq:3},
$\pm (1,0,0), \pm (0,1,0), \pm (0,0,1) \in \partial L$, 
and $\abs{Q_i(L)} \abs{P_i(L^\circ)} =8$ $(i=1,2,3)$.
Then, $L$ or $L^\circ$ is a parallelepiped.
\end{proposition}

\begin{proof}
For each $i=1,2,3$, 
we have $(Q_i(L))^\circ = P_i(L^\circ)$ and $\abs{Q_i(L)} \abs{P_i(L^\circ)} =8$.
Moreover, \eqref{eq:3} holds.
Thus, we can apply Proposition \ref{prop:9} to each $Q_i(L)$.
As a result, there exist $\Olda,\Oldb,\Oldc \in (-1,1]$ such that
\begin{equation}
\label{eq:28}
\begin{aligned}
Q_1(L)
&=
\conv \left\{
\frac{\pm 1}{1+\Olda^2}(0,1-\Olda, 1+\Olda), 
\frac{\pm 1}{1+\Olda^2}(0,-1-\Olda, 1-\Olda)
\right\}, \\
P_1(L^\circ)
&=
\conv \left\{
\pm(0,1,\Olda), \pm(0,-\Olda,1)
\right\}, \\
Q_2(L)
&=
\conv \left\{
\frac{\pm 1}{1+\Oldb^2}(1+\Oldb,0,1-\Oldb), 
\frac{\pm 1}{1+\Oldb^2}(1-\Oldb,0,-1-\Oldb)
\right\}, \\
P_2(L^\circ)
&=
\conv \left\{
\pm(\Oldb,0,1), \pm(1,0,-\Oldb)
\right\}, \\
Q_3(L)
&=
\conv \left\{
\frac{\pm 1}{1+\Oldc^2}(1-\Oldc, 1+\Oldc,0), 
\frac{\pm 1}{1+\Oldc^2}(-1-\Oldc, 1-\Oldc,0)
\right\}, \\
P_3(L^\circ)
&=
\conv \left\{
\pm(1,\Oldc,0), \pm(-\Oldc,1,0)
\right\}.
\end{aligned}
\end{equation}
Hence, 
\begin{equation*}
\abs{Q_1(L)}=\frac{4}{1+\Olda^2}, \quad
\abs{Q_2(L)}=\frac{4}{1+\Oldb^2}, \quad
\abs{Q_3(L)}=\frac{4}{1+\Oldc^2}.
\end{equation*}
Then
\begin{equation*}
\abs{L_{+++} \cap \{x=0\}}=\frac{1}{1+\Olda^2}, \quad
\abs{L_{+++} \cap \{y=0\}}=\frac{1}{1+\Oldb^2}, \quad
\abs{L_{+++} \cap \{z=0\}}=\frac{1}{1+\Oldc^2}.
\end{equation*}
It follows from Lemma \ref{lem:17} that
\begin{equation}
\label{eq:29}
 \frac{\abs{L}}{8} =\abs{L_{+++}}
\geq \frac{1}{3} \left(
\frac{1}{1+\Olda^2},
\frac{1}{1+\Oldb^2},
\frac{1}{1+\Oldc^2}
\right)\cdot(x,y,z)
\end{equation}
for any $(x,y,z) \in L$.
Hence, $(8/3\abs{L})(1/(1+\Olda^2),1/(1+\Oldb^2),1/(1+\Oldc^2)) \in L^\circ$ holds.
Similarly, we obtain that
\begin{equation}
\label{eq:34}
\frac{8}{3\abs{L}} \left(
\frac{\pm 1}{1+\Olda^2},
\frac{\pm 1}{1+\Oldb^2},
\frac{\pm 1}{1+\Oldc^2}
\right) \in L^\circ.
\end{equation}
On the other hand, since $(\pm 1,0,0), (0,\pm 1,0), (0,0,\pm 1) \in L$,
\begin{equation*}\begin{aligned}
 (\pm 1,0,0) \cdot (u,v,w) &= \pm u \leq 1, \\ 
 (0,\pm 1,0) \cdot (u,v,w) &= \pm v \leq 1, \\
 (0,0,\pm 1) \cdot (u,v,w) &= \pm w \leq 1
\end{aligned}\end{equation*}
hold for any $(u,v,w) \in L^\circ$.
Hence, 
\begin{equation}
\label{eq:36}
L^\circ \subset [-1,1]\times[-1,1]\times[-1,1].
\end{equation}

Now we determine the shapes of $L$ and $L^\circ$.
Without loss of generalities, it suffices to consider the following four cases:
(i)\ $\Olda,\Oldb,\Oldc \in (-1,1)$, \ (ii)\ $\Olda,\Oldb \in (-1,1)$, $\Oldc=1$, \ (iii)\ $\Olda=\Oldb=1$, $\Oldc \in (-1,1)$, \ 
(iv)\ $\Olda=\Oldb=\Oldc=1$.

\paragraph{Case\,(i) $\Olda,\Oldb,\Oldc \in (-1,1)$;}

By the assumption $(0,0,1) \in \partial L$, 
we have $L^\circ \cap \{w=1\} \not= \emptyset$.
By \eqref{eq:28},
\begin{equation*}\begin{aligned}
P_1(L^\circ \cap \{w=1\})&=P_1(L^\circ) \cap \{w=1\}=\{(0,-\Olda,1)\}, \\
P_2(L^\circ \cap \{w=1\})&=P_2(L^\circ) \cap \{w=1\}=\{(\Oldb,0,1)\}
\end{aligned}\end{equation*}
hold. Thus, $L^\circ \cap \{w=1\}=\{(\Oldb,-\Olda,1)\}$.
Similarly, we have
\begin{align}
\notag
L^\circ \cap \{u=\pm 1\}&=\{\pm (1,\Oldc,-\Oldb)\}, \\
\notag
L^\circ \cap \{v=\pm 1\}&=\{\pm (-\Oldc,1,\Olda)\}, \\
\label{eq:31}
L^\circ \cap \{w=\pm 1\}&=\{\pm (\Oldb,-\Olda,1)\}.
\end{align}
By \eqref{eq:36}, we obtain
\begin{equation}
\label{eq:30}
A^\circ_\pm:=\pm (1,\Oldc,-\Oldb), \,
B^\circ_\pm:=\pm (-\Oldc,1,\Olda), \,
C^\circ_\pm:=\pm (\Oldb,-\Olda,1) \in \partial L^\circ.
\end{equation}
Note that the points in \eqref{eq:30} are mutually distinct, because $\Olda,\Oldb,\Oldc \in (-1,1)$.
Moreover, these six points are in different faces of the cube $[-1,1]^3$.
We use these points to divide $L^\circ$.

The line segment connecting from $A^\circ_+$ to $B^\circ_+$ is on $\partial L^\circ$.
Indeed, if $Q=\tau A^\circ_+ + (1-\tau)B^\circ_+$ is an interior point of $L^\circ$ for some $\tau \in (0,1)$, 
then there exists a open ball $B(Q, \epsilon)$ with the center $Q$ and radius $\epsilon>0$ such that $B(Q, \epsilon) \subset L^\circ$.
Thus $P_3(B(Q, \epsilon)) \subset P_3(L^\circ)$, and $P_3(Q)$ is an interior point of $P_3(L^\circ)$.
On the other hand, by \eqref{eq:28}, the line segment connecting from $(1,\Oldc,0)$ to $(-\Oldc,1,0)$ is a part of $\partial P_3(L^\circ)$.
It is a contradiction.
Therefore, $\mathcal{C}_L(A^\circ_+, B^\circ_+)$ is the line segment connecting from $A^\circ_+$ to $B^\circ_+$.
Similarly, we see that the segments connecting $C^\circ_+$, $A^\circ_+$, $B^\circ_+$ cyclically is a triangle on $\partial L^\circ$,
because $(C^\circ_+\times A^\circ_+) \cdot B^\circ_+=1+\Olda^2+\Oldb^2+\Oldc^2>0$.
We denote the triangle by $\mathcal{T}(C^\circ_+, A^\circ_+, B^\circ_+)$.
Put
\begin{equation*}\begin{aligned}
\begin{aligned}
L^\circ_{+++}:=&O*\mathcal{T}(C^\circ_+, A^\circ_+, B^\circ_+), &
L^\circ_{-++}:=&O*\mathcal{T}(C^\circ_+, B^\circ_+, A^\circ_-), \\
L^\circ_{--+}:=&O*\mathcal{T}(C^\circ_+, A^\circ_-, B^\circ_-), &
L^\circ_{+-+}:=&O*\mathcal{T}(C^\circ_+, B^\circ_-, A^\circ_+).
\end{aligned}
\end{aligned}\end{equation*}
Then, by the symmetry of $L^\circ$, we have $\abs{L^\circ}=2(\abs{L^\circ_{+++}}+\abs{L^\circ_{-++}}+\abs{L^\circ_{--+}}+\abs{L^\circ_{+-+}})$.
Applying Lemma \ref{lem:17} to $L^\circ_{+++}$,
we have
\begin{equation*}\begin{aligned}
& 
\abs{L^\circ_{+++}} 
\geq 
\frac{1}{6}
\begin{vmatrix}
u & v & w \\
\Oldb & -\Olda & 1 \\
1 & \Oldc & -\Oldb
\end{vmatrix}
+
\frac{1}{6}
\begin{vmatrix}
u & v & w \\
1 & \Oldc & -\Oldb \\
-\Oldc & 1 & \Olda 
\end{vmatrix}
+
\frac{1}{6}
\begin{vmatrix}
u & v & w \\
-\Oldc & 1 & \Olda \\
\Oldb & -\Olda & 1 
\end{vmatrix} \\
&= 
\frac{1}{6}
%(1+\Olda^2+\Olda\Oldb+\Olda\Oldc+\Oldb-\Oldc, 1+\Oldb^2+\Olda\Oldb+\Oldb\Oldc+\Oldc-\Olda, 1+\Oldc^2+\Olda\Oldc+\Oldb\Oldc+\Olda-\Oldb)
\begin{pmatrix}
1+\Olda^2+\Olda\Oldb+\Olda\Oldc+\Oldb-\Oldc \\
1+\Oldb^2+\Olda\Oldb+\Oldb\Oldc+\Oldc-\Olda \\ 
1+\Oldc^2+\Olda\Oldc+\Oldb\Oldc+\Olda-\Oldb
\end{pmatrix}^\top
 \cdot (u,v,w)
\end{aligned}\end{equation*}
for any $(u,v,w) \in L^\circ$.
Similarly, we obtain
\begin{equation*}\begin{aligned}
&\abs{L^\circ_{-++}} 
\geq 
\frac{1}{6}
\begin{vmatrix}
u & v & w \\
\Oldb & -\Olda & 1 \\
-\Oldc & 1 & \Olda 
\end{vmatrix}
+
\frac{1}{6}
\begin{vmatrix}
u & v & w \\
-\Oldc & 1 & \Olda \\
-1 & -\Oldc & \Oldb
\end{vmatrix}
+
\frac{1}{6}
\begin{vmatrix}
u & v & w \\
-1 & -\Oldc & \Oldb \\
\Oldb & -\Olda & 1 
\end{vmatrix} \\
&= \frac{1}{6}
%(-1-\Olda^2+\Olda\Oldb+\Olda\Oldc+\Oldb-\Oldc, 1+\Oldb^2-\Olda\Oldb+\Oldb\Oldc-\Oldc-\Olda, 1+\Oldc^2-\Olda\Oldc+\Oldb\Oldc+\Olda+\Oldb) 
\begin{pmatrix}
-1-\Olda^2+\Olda\Oldb+\Olda\Oldc+\Oldb-\Oldc \\
 1+\Oldb^2-\Olda\Oldb+\Oldb\Oldc-\Oldc-\Olda \\
 1+\Oldc^2-\Olda\Oldc+\Oldb\Oldc+\Olda+\Oldb
\end{pmatrix}^\top
\cdot (u,v,w), \\
&\abs{L^\circ_{--+}} 
\geq 
\frac{1}{6}
\begin{vmatrix}
u & v & w \\
\Oldb & -\Olda & 1 \\
-1 & -\Oldc & \Oldb
\end{vmatrix}
+
\frac{1}{6}
\begin{vmatrix}
u & v & w \\
-1 & -\Oldc & \Oldb \\
\Oldc & -1 & -\Olda
\end{vmatrix}
+
\frac{1}{6}
\begin{vmatrix}
u & v & w \\
\Oldc & -1 & -\Olda \\
\Oldb & -\Olda & 1 
\end{vmatrix} \\
&= \frac{1}{6}
%(-1-\Olda^2-\Olda\Oldb+\Olda\Oldc+\Oldb+\Oldc, -1-\Oldb^2-\Olda\Oldb+\Oldb\Oldc-\Oldc-\Olda, 1+\Oldc^2-\Olda\Oldc-\Oldb\Oldc-\Olda+\Oldb)
\begin{pmatrix}
-1-\Olda^2-\Olda\Oldb+\Olda\Oldc+\Oldb+\Oldc \\
-1-\Oldb^2-\Olda\Oldb+\Oldb\Oldc-\Oldc-\Olda \\
 1+\Oldc^2-\Olda\Oldc-\Oldb\Oldc-\Olda+\Oldb
\end{pmatrix}^\top
\cdot (u,v,w), \\
&\abs{L^\circ_{+-+}} 
\geq 
\frac{1}{6}
\begin{vmatrix}
u & v & w \\
\Oldb & -\Olda & 1 \\
\Oldc & -1 & -\Olda
\end{vmatrix}
+
\frac{1}{6}
\begin{vmatrix}
u & v & w \\
\Oldc & -1 & -\Olda \\
1 & \Oldc & -\Oldb
\end{vmatrix}
+
\frac{1}{6}
\begin{vmatrix}
u & v & w \\
1 & \Oldc & -\Oldb \\
\Oldb & -\Olda & 1 
\end{vmatrix} \\
&
= 
\frac{1}{6}
%(1+\Olda^2-\Olda\Oldb+\Olda\Oldc+\Oldb+\Oldc, -1-\Oldb^2+\Olda\Oldb+\Oldb\Oldc+\Oldc-\Olda, 1+\Oldc^2+\Olda\Oldc-\Oldb\Oldc-\Olda-\Oldb) 
\begin{pmatrix}
 1+\Olda^2-\Olda\Oldb+\Olda\Oldc+\Oldb+\Oldc \\
-1-\Oldb^2+\Olda\Oldb+\Oldb\Oldc+\Oldc-\Olda \\
 1+\Oldc^2+\Olda\Oldc-\Oldb\Oldc-\Olda-\Oldb 
\end{pmatrix}^\top
\cdot (u,v,w)
\end{aligned}\end{equation*}
for any $(u,v,w) \in L^\circ.$
Therefore, the four points
\begin{equation*}\begin{aligned}
S_{+++}:=& \frac{1}{6\abs{L^\circ_{+++}}}
%(1+\Olda^2+\Olda\Oldb+\Olda\Oldc+\Oldb-\Oldc, 1+\Oldb^2+\Olda\Oldb+\Oldb\Oldc+\Oldc-\Olda, 1+\Oldc^2+\Olda\Oldc+\Oldb\Oldc+\Olda-\Oldb)
\begin{pmatrix}
1+\Olda^2+\Olda\Oldb+\Olda\Oldc+\Oldb-\Oldc \\
1+\Oldb^2+\Olda\Oldb+\Oldb\Oldc+\Oldc-\Olda \\
1+\Oldc^2+\Olda\Oldc+\Oldb\Oldc+\Olda-\Oldb
\end{pmatrix}^\top, \\ 
S_{-++}:=& \frac{1}{6\abs{L^\circ_{-++}}} 
%(-1-\Olda^2+\Olda\Oldb+\Olda\Oldc+\Oldb-\Oldc, 1+\Oldb^2-\Olda\Oldb+\Oldb\Oldc-\Oldc-\Olda, 1+\Oldc^2-\Olda\Oldc+\Oldb\Oldc+\Olda+\Oldb)
\begin{pmatrix}
-1-\Olda^2+\Olda\Oldb+\Olda\Oldc+\Oldb-\Oldc \\
 1+\Oldb^2-\Olda\Oldb+\Oldb\Oldc-\Oldc-\Olda \\
 1+\Oldc^2-\Olda\Oldc+\Oldb\Oldc+\Olda+\Oldb
\end{pmatrix}^\top, \\ 
S_{--+}:=& \frac{1}{6\abs{L^\circ_{--+}}} 
%(-1-\Olda^2-\Olda\Oldb+\Olda\Oldc+\Oldb+\Oldc, -1-\Oldb^2-\Olda\Oldb+\Oldb\Oldc-\Oldc-\Olda, 1+\Oldc^2-\Olda\Oldc-\Oldb\Oldc-\Olda+\Oldb)
\begin{pmatrix}
-1-\Olda^2-\Olda\Oldb+\Olda\Oldc+\Oldb+\Oldc \\
-1-\Oldb^2-\Olda\Oldb+\Oldb\Oldc-\Oldc-\Olda \\
 1+\Oldc^2-\Olda\Oldc-\Oldb\Oldc-\Olda+\Oldb
\end{pmatrix}^\top, \\ 
S_{+-+}:=& \frac{1}{6\abs{L^\circ_{+-+}}} 
%(1+\Olda^2-\Olda\Oldb+\Olda\Oldc+\Oldb+\Oldc, -1-\Oldb^2+\Olda\Oldb+\Oldb\Oldc+\Oldc-\Olda, 1+\Oldc^2+\Olda\Oldc-\Oldb\Oldc-\Olda-\Oldb)
\begin{pmatrix}
 1+\Olda^2-\Olda\Oldb+\Olda\Oldc+\Oldb+\Oldc \\
-1-\Oldb^2+\Olda\Oldb+\Oldb\Oldc+\Oldc-\Olda \\
 1+\Oldc^2+\Olda\Oldc-\Oldb\Oldc-\Olda-\Oldb 
\end{pmatrix}^\top
\end{aligned}\end{equation*}
are contained in $L$.
Combining the fact that $S_{+++} \in L$ with \eqref{eq:29}, we have
\begin{equation}
\label{eq:47}
\abs{L_{+++}} \geq \frac{1}{18\abs{L^\circ_{+++}}}
\left(3+\frac{\Olda\Oldb+\Olda\Oldc+\Oldb-\Oldc}{1+\Olda^2}+\frac{\Olda\Oldb+\Oldb\Oldc+\Oldc-\Olda}{1+\Oldb^2}+\frac{\Olda\Oldc+\Oldb\Oldc+\Olda-\Oldb}{1+\Oldc^2}\right).
\end{equation}
Similarly, we obtain 
\begin{equation*}\begin{aligned}
&\abs{L_{-++}} \\
&\geq \frac{1}{18\abs{L^\circ_{-++}}}
\left(3-\frac{\Olda\Oldb+\Olda\Oldc+\Oldb-\Oldc}{1+\Olda^2}+\frac{-\Olda\Oldb+\Oldb\Oldc-\Oldc-\Olda}{1+\Oldb^2}+\frac{-\Olda\Oldc+\Oldb\Oldc+\Olda+\Oldb}{1+\Oldc^2}\right), \\
&\abs{L_{--+}} \\
&\geq \frac{1}{18\abs{L^\circ_{--+}}}
\left(3-\frac{-\Olda\Oldb+\Olda\Oldc+\Oldb+\Oldc}{1+\Olda^2}-\frac{-\Olda\Oldb+\Oldb\Oldc-\Oldc-\Olda}{1+\Oldb^2}+\frac{-\Olda\Oldc-\Oldb\Oldc-\Olda+\Oldb}{1+\Oldc^2}\right), \\
&\abs{L_{+-+}} \\
&\geq \frac{1}{18\abs{L^\circ_{+-+}}}
\left(3+\frac{-\Olda\Oldb+\Olda\Oldc+\Oldb+\Oldc}{1+\Olda^2}-\frac{\Olda\Oldb+\Oldb\Oldc+\Oldc-\Olda}{1+\Oldb^2}+\frac{\Olda\Oldc-\Oldb\Oldc-\Olda-\Oldb}{1+\Oldc^2}\right).
\end{aligned}\end{equation*}
Thus,
\begin{equation*}
\abs{L_{+++}} \abs{L^\circ_{+++}}
+ \abs{L_{-++}} \abs{L^\circ_{-++}}
+ \abs{L_{--+}} \abs{L^\circ_{--+}}
+ \abs{L_{+-+}} \abs{L^\circ_{+-+}}
\geq \frac{12}{18} = \frac{2}{3}.
\end{equation*}
On the other hand, since the assumption \eqref{eq:3} for $L$ means $\abs{L_{\pm\pm\pm}}=\abs{L}/8$, we have
\begin{equation*} 
\begin{aligned}
&\abs{L_{+++}} \abs{L^\circ_{+++}}
+ \cdots +
\abs{L_{+-+}} \abs{L^\circ_{+-+}}
=
\frac{\abs{L}}{8} \left(
\abs{L^\circ_{+++}} + \abs{L^\circ_{-++}} + \abs{L^\circ_{--+}} + \abs{L^\circ_{+-+}}
\right) \\
&=
\frac{\abs{L} \abs{L^\circ}}{16}=\frac{32}{3} \frac{1}{16}= \frac{2}{3}.
\end{aligned}
\end{equation*}
Thus, the above inequality holds with equality.
Hence \eqref{eq:47} holds with equality. That is, \eqref{eq:29} for $S_{+++}$ holds with equality.
Taking the proof of Proposition \ref{prop:1} into account, we have
\begin{equation*}\begin{aligned}
&S_{+++} \in \{(x,y,z); x \geq 0, y \geq 0, z \geq 0\},\\
&
L_{+++} = \conv\left\{
L_{+++} \cap \{x=0\}, 
L_{+++} \cap \{y=0\}, 
L_{+++} \cap \{z=0\}, S_{+++}
\right\}.
\end{aligned}\end{equation*}
Hence $S_{+++} \in \partial L$.
We put $S_{+++}=(s_1,s_2,s_3)$ for simplicity.

Now, we consider the case $S_{+++} \not \in \{x=0\}$.
Then 
\begin{equation*}
\conv \left\{\frac{(0,1-\Olda, 1+\Olda)}{1+\Olda^2}, \frac{(0,-1-\Olda, 1-\Olda)}{1+\Olda^2}, S_{+++} \right\}
\end{equation*}
is a part of a face of $L$.
Its dual face is a vertex of $L^\circ$.
Then the dual face is
\begin{equation*}
\begin{pmatrix}
0 & (1-\Olda)/(1+\Olda^2) & (1+\Olda)/(1+\Olda^2) \\
0 & -(1+\Olda)/(1+\Olda^2) & (1-\Olda)/(1+\Olda^2) \\
s_1 & s_2 & s_3 \\
\end{pmatrix}^{-1}
\begin{pmatrix}
1 \\ 1 \\ 1
\end{pmatrix}
=
\begin{pmatrix}
\frac{-s_3+\Olda s_2+1}{s_1} \\ -\Olda \\ 1
\end{pmatrix}.
\end{equation*}
Comparing it with \eqref{eq:31}, we obtain
\begin{equation*}
\left(
\frac{-s_3+\Olda s_2+1}{s_1}, -\Olda, 1
\right)=C^\circ_+.
\end{equation*}
Hence $-s_3+\Olda s_2+1=\Oldb s_1$.
Then, 
\begin{align}
\notag
&6\abs{L^\circ_{+++}}
=
6\abs{L^\circ_{+++}}(\Oldb s_1 -\Olda s_2+s_3) \\
\notag
&=\Oldb(1+\Olda^2+\Olda\Oldb+\Olda\Oldc+\Oldb-\Oldc) - \Olda(1+\Oldb^2+\Olda\Oldb+\Oldb\Oldc+\Oldc-\Olda) \\
\notag
& \quad + 1+\Oldc^2+\Olda\Oldc+\Oldb\Oldc+\Olda-\Oldb \\
\label{eq:53}
&=1+\Olda^2+\Oldb^2+\Oldc^2. 
\end{align}
On the other hand, the volume of 
$\conv\left\{O, A^\circ_+, B^\circ_+, C^\circ_+ \right\} \subset L^\circ_{+++}$
is
\begin{equation*}
 \frac{1}{6} 
\begin{vmatrix}
1 & \Oldc & -\Oldb \\
-\Oldc & 1 & \Olda \\
\Oldb & -\Olda & 1
\end{vmatrix}
=\frac{1}{6} (1+\Olda^2+\Oldb^2+\Oldc^2).
\end{equation*}
Hence, 
\begin{equation}
\label{eq:32}
L^\circ_{+++}=
\conv\left\{O,
A^\circ_+, 
B^\circ_+, 
C^\circ_+
\right\}.
\end{equation}

Next, we consider the case $S_{+++} \in \{x=0\}$.
Since $S_{+++} \in \partial L$, we obtain $S_{+++} \in \partial Q_1(L)$.
Since $S_{+++}$ is in the first octant, $S_{+++}$ is on the line segment connecting from $(0,1-\Olda,1+\Olda)/(1+\Olda^2)$ to $(0,0,1)$,
or on the line segment connecting from $(0,1,0)$ to $(0,1-\Olda,1+\Olda)/(1+\Olda^2)$.
For the former case, we have
\begin{equation*}
 S_{+++} \cdot C^\circ_+ = (1-\tau) \frac{(0,1-\Olda, 1+\Olda)}{1+\Olda^2} \cdot C^\circ_+ + \tau (0,0,1)\cdot C^\circ_+= 1,
\end{equation*}
which yields $\Oldb s_1 - \Olda s_2 + s_3=1$, and
$6 \abs{L^\circ_{+++}}=1+\Olda^2+\Oldb^2+\Oldc^2$ holds as \eqref{eq:53}.
Thus, we obtain \eqref{eq:32} as well.
For the latter case,
we have 
\begin{equation*}
 S_{+++} \cdot B^\circ_+ = (1-\tau) (0,1,0)\cdot B^\circ_+ + \tau \frac{(0,1-\Olda, 1+\Olda)}{1+\Olda^2} \cdot B^\circ_+ =1.
\end{equation*}
Since $s_1=0$, we obtain $s_2 + \Olda s_3 =1$.
Thus
\begin{equation*}\begin{aligned}
&6\abs{L^\circ_{+++}}
=
6\abs{L^\circ_{+++}}(s_2 + \Olda s_3) \\
&=1+\Oldb^2+\Olda\Oldb+\Oldb\Oldc+\Oldc-\Olda+\Olda(1+\Oldc^2+\Olda\Oldc+\Oldb\Oldc+\Olda-\Oldb) \\
&=1+\Olda^2+\Oldb^2+\Oldc^2 +\Oldc(1+\Olda^2+\Olda\Oldb+\Olda\Oldc+\Oldb-\Oldc) \\
&=1+\Olda^2+\Oldb^2+\Oldc^2 + 6\Oldc \abs{L^\circ_{+++}} s_1 \\
&=1+\Olda^2+\Oldb^2+\Oldc^2,
\end{aligned}\end{equation*}
which yields \eqref{eq:32} as well.

By repeating the above arguments for $L^\circ_{-++}, L^\circ_{--+}, L^\circ_{+-+}$,
we obtain
\begin{equation*}
L^\circ=
\conv\left\{
\pm A^\circ_+, 
\pm B^\circ_+, 
\pm C^\circ_+
\right\}.
\end{equation*}
Thus,  $L^\circ$ is a centrally symmetric octahedron.
Then,
\begin{equation*}
 L=\conv \left\{\pm S_{+++}, \pm S_{-++}, \pm S_{--+}, \pm S_{+-+}\right\}
\end{equation*}
is a parallelepiped.

\paragraph{Case\,(ii) $\Olda,\Oldb \in (-1,1)$, $\Oldc=1$;}

By \eqref{eq:28}, we have
\begin{equation*}
 Q_3(L)
=
\conv \left\{
\pm (1,0,0), 
\pm (0,1,0)
\right\}, \quad
P_3(L^\circ)
=
[-1,1]\times [-1,1] \times\{0\}.
\end{equation*}
Thus, we obtain
\begin{equation*}
P_3(L^\circ \cap \{u=1\})
=
P_3(L^\circ) \cap \{u=1\}
=
\{1\}\times [-1,1] \times\{0\}.
\end{equation*}
Hence, for any $v \in [-1,1]$, there exists $w(v) \in [-1,1]$ such that $(1,v,w(v)) \in L^\circ$ and 
\begin{equation*}
 (1,v,w(v))\in L^\circ \cap \{u=1\}.
\end{equation*}
On the other hand, 
combining \eqref{eq:28} with $\Oldb \in (-1,1)$, we have
\begin{equation*}
 P_2(L^\circ \cap \{u=1\})=\{(1,0,-\Oldb)\}.
\end{equation*}
Consequently, we obtain $w(v)=-\Oldb$\, $(v \in [-1,1])$ and 
\begin{equation}
\label{eq:33}
\{1\} \times [-1,1] \times \{-\Oldb\} \subset L^\circ.
\end{equation}
Especially, we have $(1,-1,-\Oldb), (1,1,-\Oldb) \in L^\circ$.
Combining \eqref{eq:28} with $\Olda \in (-1,1)$, 
\begin{equation*}\begin{aligned}
P_1((1,1,-\Oldb))=(0,1,-\Oldb) &\in P_1(L^\circ)\cap\{v=1\}=\{(0,1,\Olda)\}, \\
P_1((1,-1,-\Oldb))=(0,-1,-\Oldb) &\in P_1(L^\circ)\cap\{v=-1\}=\{(0,-1,-\Olda)\}.
\end{aligned}\end{equation*}
Hence $-\Oldb=\Olda=-\Olda$. Thus, $\Olda=\Oldb=0$.
Especially, by \eqref{eq:33}, we obtain
\begin{equation*}
\{1\} \times [-1,1] \times \{0\} \subset L^\circ.
\end{equation*}
By the symmetry and convexity of $L^\circ$, we obtain
\begin{equation*}
 [-1,1] \times [-1,1] \times \{0\} \subset L^\circ.
\end{equation*}
Since $\Olda=\Oldb=0$, similarly as the above,
\begin{equation*}\begin{aligned}
P_1(L^\circ \cap \{w=1\})&=P_1(L^\circ) \cap \{w=1\}=\{(0,0,1)\}, \\
P_2(L^\circ \cap \{w=1\})&=P_2(L^\circ) \cap \{w=1\}=\{(0,0,1)\},
\end{aligned}\end{equation*}
which implies that $(0,0,1) \in L^\circ$.
In summary,
\begin{equation}
\label{eq:54}
\conv \left\{[-1,1] \times [-1,1] \times \{0\}, \pm (0,0,1)\right\} \subset L^\circ.
\end{equation}

Now, since $\Olda=\Oldb=0$, 
by \eqref{eq:34}, we have 
\begin{equation*}
\frac{8}{3\abs{L}} \left(1,1,\frac{1}{2}\right)\in L^\circ.
\end{equation*}
Since
\begin{equation*}
 P_2 \left(
\frac{8}{3\abs{L}} \left(1,1,\frac{1}{2}\right)
\right)
=
\frac{8}{3\abs{L}} \left(1,0,\frac{1}{2}\right) \in P_2(L^\circ) = \conv\{\pm(1,0,0), \pm(0,0,1)\},
\end{equation*}
we obtain
\begin{equation*}
 \frac{8}{3\abs{L}}\left(
1+\frac{1}{2}
\right) = \frac{4}{\abs{L}}\leq 1.
\end{equation*}
Since $\mathcal{P}(L)=32/3$, we have
\begin{equation}
\label{eq:35}
 \abs{L^\circ} = \frac{32}{3 \abs{L}} \leq \frac{8}{3}.
\end{equation}
On the other hand, by \eqref{eq:54},
\begin{equation*}
 \abs{L^\circ} \geq 
\abs{\conv \left\{[-1,1] \times [-1,1] \times \{0\}, \pm (0,0,1)\right\}}=\frac{8}{3}
\end{equation*}
holds.
Therefore, $\abs{L^\circ}=8/3$ and
\begin{equation*}
L^\circ =
\conv \left\{[-1,1] \times [-1,1] \times \{0\}, \pm (0,0,1)\right\}.
\end{equation*}
Hence, $L^\circ$ is a centrally symmetric octahedron and
\begin{equation*}
 L=\conv \left\{
(\pm 1,0,\pm 1), (0,\pm 1,\pm 1)
\right\}
\end{equation*}
is a parallelepiped.

\paragraph{Case\,(iii) $\Olda=\Oldb=1$, $\Oldc \in (-1,1)$;}

By \eqref{eq:28}, we have 
\begin{equation*}\begin{aligned}
P_1(L^\circ)
&=
\{0\}\times[-1,1]\times[-1,1], \quad
P_2(L^\circ)
=
[-1,1]\times\{0\}\times[-1,1], \\
P_3(L^\circ)
&=
\conv \left\{
\pm(1,\Oldc,0), \pm(-\Oldc,1,0)
\right\}.
\end{aligned}\end{equation*}
Since $\Oldc \in (-1,1)$, we see
\begin{equation*}\begin{aligned}
P_2(L^\circ \cap \{u=1\}) &= \{1\} \times \{0\} \times [-1,1], \\
P_3(L^\circ \cap \{u=1\}) &= \{(1,\Oldc,0)\}.
\end{aligned}\end{equation*}
Hence
\begin{equation*}
\{1\}\times\{\Oldc\}\times[-1,1] \subset L^\circ.
\end{equation*}
By the symmetry of $L^\circ$, we obtain
\begin{equation*}
\{-1\}\times\{-\Oldc\}\times[-1,1] \subset L^\circ.
\end{equation*}
Similarly,
\begin{equation*}\begin{aligned}
P_1(L^\circ \cap \{v=1\}) &= \{0\} \times \{1\} \times [-1,1], \\
P_3(L^\circ \cap \{v=1\}) &= \{(-\Oldc,1,0)\}.
\end{aligned}\end{equation*}
Thus 
\begin{equation*}
\{-\Oldc\}\times\{1\}\times[-1,1] \subset L^\circ.
\end{equation*}
By the symmetry of $L^\circ$, we get
\begin{equation*}
\{\Oldc\}\times\{-1\}\times[-1,1] \subset L^\circ.
\end{equation*}
By the convexity of $L^\circ$, 
\begin{equation*}
\conv \left\{
\pm(1,\Oldc,1), \pm(-\Oldc,1,1),
\pm(1,\Oldc,-1), \pm(-\Oldc,1,-1)
\right\}
\subset L^\circ.
\end{equation*}
On the other hand, we have 
$L^\circ \subset [-1,1]\times[-1,1]\times[-1,1]$ by \eqref{eq:36}.
Combining it with $P_3(L^\circ)=\conv \left\{\pm(1,\Oldc,0), \pm(-\Oldc,1,0)\right\}$, we obtain
\begin{equation*}
 L^\circ \subset 
\conv \left\{
\pm(1,\Oldc,1), \pm(-\Oldc,1,1),
\pm(1,\Oldc,-1), \pm(-\Oldc,1,-1)
\right\}.
\end{equation*}
Thus, 
\begin{equation*}
 L^\circ =
\conv \left\{
\pm(1,\Oldc,1), \pm(-\Oldc,1,1),
\pm(1,\Oldc,-1), \pm(-\Oldc,1,-1)
\right\}
\end{equation*}
is a parallelepiped.

\paragraph{Case\,(iv) $\Olda=\Oldb=\Oldc=1$;}

By \eqref{eq:28}, we have
\begin{equation*}\begin{aligned}
Q_1(L)
&=
\conv \left\{
\pm (0,1,0), \pm (0,0,1)
\right\}, &
P_1(L^\circ)
&=
\{0\}\times[-1,1]\times[-1,1], \\
Q_2(L)
&=
\conv \left\{
\pm (1,0,0), \pm (0,0,1)
\right\}, &
P_2(L^\circ)
&=
[-1,1]\times\{0\}\times[-1,1], \\
Q_3(L)
&=
\conv \left\{
\pm (1,0,0), \pm (0,1,0)
\right\}, &
P_3(L^\circ)
&=
[-1,1]\times[-1,1]\times\{0\}.
\end{aligned}\end{equation*}
Recall that $L^\circ \subset [-1,1]\times[-1,1]\times[-1,1]$.
If $(1,1,1) \in L^\circ$, then we have 
\begin{equation*}
 (1,1,1) \cdot (x,y,z) =x+y+z \leq 1
\end{equation*}
for any $(x,y,z) \in L$.
On the other hand, $(1,0,0), (0,1,0), (0,0,1) \in L$.
By the definition of $L_{+++}$, we obtain
\begin{equation*}
L_{+++} = \conv \left\{(1,0,0), (0,1,0), (0,0,1)\right\}.
\end{equation*}
Moreover, since
\begin{equation*}
 \abs{L_{-++}}=\frac{\abs{L}}{8}=\abs{L_{+++}}=\frac{1}{6},
\end{equation*}
$(-1,0,0), (0,1,0), (0,0,1) \in L$,
and the definition of $L_{-++}$, we have
\begin{equation*}
L_{-++} = \conv \left\{(-1,0,0), (0,1,0), (0,0,1)\right\}.
\end{equation*}
By similar arguments for $L_{--+}$ and $L_{+-+}$, we obtain
\begin{equation*}
 L= \conv \left\{\pm(1,0,0), \pm(0,1,0), \pm(0,0,1)\right\}.
\end{equation*}
Thus
\begin{equation*}
 L^\circ= [-1,1]\times[-1,1]\times[-1,1]
\end{equation*}
is a parallelepiped.
Similarly, if one of the eight points $(\pm 1, \pm 1, \pm 1)$ is in $L^\circ$, then $L$ is an octahedron and $L^\circ$ is a parallelepiped.

Thus, from now on we may assume that
$(\pm 1, \pm 1, \pm 1) \not\in L^\circ$.
Using the characterization of $Q_i(L)$ $(i=1,2,3)$, by Lemma \ref{lem:17}, 
we have
\begin{equation*}
\frac{\abs{L}}{8}=\abs{L_{+++}}\geq
\frac{1}{6}(1,1,1) \cdot (x,y,z)
\end{equation*}
for any $(x,y,z) \in L$.
Similar arguments about the other seven parts imply
\begin{equation*}
 \frac{4}{3\abs{L}} (\pm 1, \pm 1, \pm 1) \in L^\circ.
\end{equation*}
Next, by the characterization of $P_i(L^\circ)$ $(i=1,2,3)$, 
there exists $\Newa_i, \Newb_i, \Newc_i \in (-1,1)$\, $(i=1,2)$ such that
\begin{equation*}\begin{aligned}
A^\circ_1&:=(\Newa_1,1,1),&
A^\circ_2&:=(\Newa_2,-1,1), &
B^\circ_1&:=(1,\Newb_1,1),&
B^\circ_2&:=(1,\Newb_2,-1),\\
C^\circ_1&:=(1,1,\Newc_1),&
C^\circ_2&:=(-1,1,\Newc_2)
\end{aligned}\end{equation*}
are in $L^\circ$.
We consider three planes 
$H_{O A^\circ_1 A^\circ_2}$, 
$H_{O B^\circ_1 B^\circ_2}$, and 
$H_{O C^\circ_1 C^\circ_2}$ 
determined by $O, A^\circ_1, A^\circ_2$; 
$O, B^\circ_1, B^\circ_2$;  and 
$O, C^\circ_1, C^\circ_2$,
respectively.
We divide $L^\circ$ by these planes into 8 pieces $L^\circ_{\pm\pm\pm}$.
By the convexity of $L^\circ$, each line segment
connecting any two points in
$\{ \pm A^\circ_1, \pm A^\circ_2, \pm B^\circ_1, \pm B^\circ_2, \pm C^\circ_1, \pm C^\circ_2\}$
is in $L^\circ$.
Thus $H_{O A^\circ_1 A^\circ_2}$, 
$H_{O B^\circ_1 B^\circ_2}$, and 
$H_{O C^\circ_1 C^\circ_2}$ 
contain $\conv\left\{\pm A^\circ_1, \pm A^\circ_2\right\}$,
$\conv\left\{\pm B^\circ_1, \pm B^\circ_2\right\}$, and
$\conv\left\{\pm C^\circ_1, \pm C^\circ_2\right\}$, respectively.
We put 
\begin{equation*}\begin{aligned}
D^\circ&=(1,d_2,d_3):= B^\circ_1 B^\circ_2 \cap C^\circ_1 (-C^\circ_2), \\
E^\circ&=(e_1,1,e_3):= C^\circ_1 C^\circ_2 \cap A^\circ_1 (-A^\circ_2), \\
F^\circ&=(f_1,f_2,1):= A^\circ_1 A^\circ_2 \cap B^\circ_1 (-B^\circ_2).
\end{aligned}\end{equation*}
Then $D^\circ$, $E^\circ$, and $F^\circ$ are in $L^\circ$.
Applying Lemma \ref{lem:17} to
$L^\circ$
with $C_{1,2}=\mathcal{C}_{L^\circ}(F^\circ, D^\circ)$, $C_{2,3}=\mathcal{C}_{L^\circ}(D^\circ, E^\circ)$, $C_{3,1}=\mathcal{C}_{L^\circ}(E^\circ, F^\circ)$, and
$P=4(1,1,1)/3\abs{L} \in L^\circ$, we have 
\begin{equation}
\label{eq:55}
\begin{aligned}
\abs{L^\circ_{+++}} 
&\geq 
\frac{2}{9\abs{L}}
\begin{vmatrix}
1 & 1 & 1 \\
f_1 & f_2 & 1 \\
1 & \Newb_1 & 1
\end{vmatrix}
+
\frac{2}{9\abs{L}}
\begin{vmatrix}
1 & 1 & 1 \\
1 & \Newb_1 & 1 \\
1 & d_2 & d_3
\end{vmatrix}
+
\frac{2}{9\abs{L}}
\begin{vmatrix}
1 & 1 & 1 \\
1 & d_2 & d_3 \\
1 & 1 & \Newc_1
\end{vmatrix}\\
&+
\frac{2}{9\abs{L}}
\begin{vmatrix}
1 & 1 & 1 \\
1 & 1 & \Newc_1 \\
e_1 & 1 & e_3
\end{vmatrix}
+
\frac{2}{9\abs{L}}
\begin{vmatrix}
1 & 1 & 1 \\
e_1 & 1 & e_3 \\
\Newa_1 & 1 & 1
\end{vmatrix}
+
\frac{2}{9\abs{L}}
\begin{vmatrix}
1 & 1 & 1 \\
\Newa_1 & 1 & 1\\
f_1 & f_2 & 1
\end{vmatrix}\\
&=
\frac{2}{9\abs{L}}
(6 -2 \Newa_1 -2 \Newb_1 -2 \Newc_1 -d_2 -d_3 -e_1 -e_3 -f_1 -f_2 
\\
&\qquad +\Newc_1 d_2 +\Newb_1 d_3 +\Newc_1 e_1 +\Newa_1 e_3 +\Newb_1 f_1 +\Newa_1 f_2).
\end{aligned}
\end{equation}
Similarly, we obtain that
\begin{equation*}\begin{aligned}
\abs{L^\circ_{-++}} 
&\geq 
\frac{2}{9\abs{L}}
\begin{vmatrix}
-1 & 1 & 1 \\
f_1 & f_2 & 1 \\
\Newa_1 & 1 & 1
\end{vmatrix}
+
\frac{2}{9\abs{L}}
\begin{vmatrix}
-1 & 1 & 1 \\
\Newa_1 & 1 & 1 \\
e_1 & 1 & e_3
\end{vmatrix}
+
\frac{2}{9\abs{L}}
\begin{vmatrix}
-1 & 1 & 1 \\
e_1 & 1 & e_3 \\
-1 & 1 & \Newc_2
\end{vmatrix}\\
&+
\frac{2}{9\abs{L}}
\begin{vmatrix}
-1 & 1 & 1 \\
-1 & 1 & \Newc_2 \\
-1 & -d_2 & -d_3
\end{vmatrix}
+
\frac{2}{9\abs{L}}
\begin{vmatrix}
-1 & 1 & 1 \\
-1 & -d_2 & -d_3 \\
-1 & -\Newb_2 & 1
\end{vmatrix}
+
\frac{2}{9\abs{L}}
\begin{vmatrix}
-1 & 1 & 1 \\
-1 & -\Newb_2 & 1 \\
f_1 & f_2 & 1
\end{vmatrix}\\
&=
\frac{2}{9\abs{L}}
(6 +2 \Newa_1 +2 \Newb_2 -2 \Newc_2 +d_2 +d_3 +e_1 -e_3 +f_1 -f_2
\\
&\qquad +\Newb_2 d_3 -\Newa_1 e_3 +\Newb_2 f_1 -\Newa_1 f_2 -\Newc_2 d_2 - \Newc_2 e_1
), 
\end{aligned}\end{equation*}
\begin{equation*}\begin{aligned}
\abs{L^\circ_{--+}} 
&\geq 
\frac{2}{9\abs{L}}
\begin{vmatrix}
-1 & -1 & 1 \\
f_1 & f_2 & 1 \\
-1 & -\Newb_2 & 1
\end{vmatrix}
+
\frac{2}{9\abs{L}}
\begin{vmatrix}
-1 & -1 & 1 \\
-1 & -\Newb_2 & 1 \\
-1 & -d_2 & -d_3
\end{vmatrix}
+
\frac{2}{9\abs{L}}
\begin{vmatrix}
-1 & -1 & 1 \\
-1 & -d_2 & -d_3 \\
-1 & -1 & -\Newc_1
\end{vmatrix}\\
&+
\frac{2}{9\abs{L}}
\begin{vmatrix}
-1 & -1 & 1 \\
-1 & -1 & -\Newc_1 \\
-e_1 & -1 & -e_3
\end{vmatrix}
+
\frac{2}{9\abs{L}}
\begin{vmatrix}
-1 & -1 & 1 \\
-e_1 & -1 & -e_3 \\
\Newa_2 & -1 & 1
\end{vmatrix}
+
\frac{2}{9\abs{L}}
\begin{vmatrix}
-1 & -1 & 1 \\
\Newa_2 & -1 & 1\\
f_1 & f_2 & 1
\end{vmatrix}\\
&=
\frac{2}{9\abs{L}}
(
6 +2 \Newa_2 -2 \Newb_2 +2 \Newc_1 -d_2 +d_3 -e_1 +e_3 +f_1 +f_2
\\
&\qquad 
-\Newb_2 d_3 +\Newa_2 e_3 -\Newb_2 f_1 +\Newa_2 f_2 -\Newc_1 d_2 -\Newc_1 e_1
), \\
\abs{L^\circ_{+-+}} 
&\geq 
\frac{2}{9\abs{L}}
\begin{vmatrix}
1 & -1 & 1 \\
f_1 & f_2 & 1 \\
\Newa_2 & -1 & 1
\end{vmatrix}
+
\frac{2}{9\abs{L}}
\begin{vmatrix}
1 & -1 & 1 \\
\Newa_2 & -1 & 1 \\
-e_1 & -1 & -e_3
\end{vmatrix}
+
\frac{2}{9\abs{L}}
\begin{vmatrix}
1 & -1 & 1 \\
-e_1 & -1 & -e_3 \\
1 & -1 & -\Newc_2
\end{vmatrix}\\
&+
\frac{2}{9\abs{L}}
\begin{vmatrix}
1 & -1 & 1 \\
1 & -1 & -\Newc_2 \\
1 & d_2 & d_3
\end{vmatrix}
+
\frac{2}{9\abs{L}}
\begin{vmatrix}
1 & -1 & 1 \\
1 & d_2 & d_3 \\
1 & \Newb_1 & 1
\end{vmatrix}
+
\frac{2}{9\abs{L}}
\begin{vmatrix}
1 & -1 & 1 \\
1 & \Newb_1 & 1 \\
f_1 & f_2 & 1
\end{vmatrix}\\
&=
\frac{2}{9\abs{L}}
(6 -2 \Newa_2 +2 \Newb_1 +2 \Newc_2+d_2 -d_3 +e_1 +e_3 -f_1 +f_2 
\\
&\qquad
-\Newb_1 d_3 -\Newa_2 e_3-\Newb_1 f_1-\Newa_2 f_2 +\Newc_2 d_2 +\Newc_2 e_1 
).
\end{aligned}\end{equation*}
Taking these sums, we obtain
\begin{equation*}
\frac{\abs{L^\circ}}{2} 
= \abs{L^\circ_{+++}}+\abs{L^\circ_{-++}}+\abs{L^\circ_{--+}}+\abs{L^\circ_{+-+}} 
\geq \frac{16}{3\abs{L}}.
\end{equation*}
On the other hand, by the assumption, $\abs{L} \abs{L^\circ}=32/3$.
Thus \eqref{eq:55} holds with equality.
Especially, 
$L^\circ_{+++}$ is the cone with the vertex $4(1,1,1)/3\abs{L}$ over $O*(F^\circ B^\circ_1 D^\circ C^\circ_1 E^\circ A^\circ_1 F^\circ)$. It is not convex if $4/3\abs{L} <1$.
Thus $4/3\abs{L} \geq 1$. It means that $(1,1,1) \in L^\circ$, which is a contradiction.
Consequently, in the case $\Olda=\Oldb=\Oldc=1$, $(\pm 1, \pm 1, \pm 1) \not\in L^\circ$ does not occur, and $L^\circ=[-1,1]\times[-1,1]\times[-1,1]$. Hence $L^\circ$ is a parallelepiped.
\end{proof}

\begin{proof}[Proof of Theorem \ref{prop:14}]
Combining Proposition \ref{prop:10} and Proposition \ref{prop:11}, the conclusion follows immediately.
\end{proof}

\appendix

\section{Proof of Proposition \ref{prop:4}}
\label{sec:A}

To prove Proposition \ref{prop:4},
we have to modify the map $\mathscr{F}$ defined on $\partial D$.
We consider $\partial D$ as a hexahedron, 
and deform $\mathscr{F}$ near each vertices of $\partial D$ and deform near each edges of $\partial D$.
We introduce some auxiliary functions satisfy \eqref{eq:12}, \eqref{eq:13}, and \eqref{eq:14}.
We first check  that $G^2+H^2$ has same value at vertices.
\begin{lemma}
\label{lem:15}
$(G^2+H^2)$ takes same value at eight vertices on $\partial D$;
$(0,0,0)$, $(0,\pi,0)$, $(0, \pi, \pi)$, $(0, 0, \pi)$,
$(1,0,0)$, $(1,\pi,0)$, $(1, \pi, \pi)$, $(1, 0, \pi)$.
\end{lemma}

\begin{remark}
\label{rem:1}
More generally, if a vector field $(0, \check{G}, \check{H})$ satisfies \eqref{eq:12}, \eqref{eq:13}, and \eqref{eq:14}, 
then $\check{G}^2 + \check{H}^2$ takes same value at the eight vertices.
\end{remark}

\begin{proof}[Proof of Lemma \ref{lem:15}]
By \eqref{eq:12}, we have
\begin{equation*}\begin{aligned}
&(G^2+H^2)(1,0,0)=(G^2+H^2)(0,0,0), \\
&(G^2+H^2)(1,\pi,0)=(G^2+H^2)(0,\pi,0), \\
&(G^2+H^2)(1,\pi,\pi)=(G^2+H^2)(0,\pi,\pi), \\
&(G^2+H^2)(1,0,\pi)=(G^2+H^2)(0,0,\pi).
\end{aligned}\end{equation*}
By \eqref{eq:13}, we also obtain
\begin{equation*}\begin{aligned}
&(G^2+H^2)(0,\pi,0)= (G^2+H^2)(1,0,0), \\
&(G^2+H^2)(0,\pi,\pi)=(G^2+H^2)(1,0,\pi).
\end{aligned}\end{equation*}
By \eqref{eq:14}, 
\begin{equation*}
 (G^2+H^2)(0,0,\pi)=
 (G^2+H^2)(0,\pi,0).
\end{equation*}
Thus we see the conclusion.
\end{proof}

To modify the vector field $(F,G,H)$ on $D \subset \R^3$, we introduce a cut-off function
$\zeta(s, \phi, \psi)$.
Let $\zeta: D \rightarrow \R$ be a smooth function such that, radially symmetric with respect to the origin, $0 \leq \zeta \leq 1$, $\zeta(0,0,0)=1$, and $\supp \zeta$ is a small neighborhood of the origin.
We put
\begin{equation*}\begin{aligned}
G_0(s,\phi,\psi)
&:=
\zeta(s,\phi,\psi)
-
\zeta(s,\pi-\phi,\psi) \\
&\quad
-
\zeta(s,\pi-\phi,\pi-\psi)
+
\zeta(s,\phi,\pi-\psi), \\
H_0(s,\phi,\psi)
&:=
\zeta(T_\psi^{-1}(s),\phi,\psi)
-
\zeta(T_\psi(s),\pi-\phi,\psi) \\
&\quad
-
\zeta(T_\psi(s),\pi-\phi,\pi-\psi)
+
\zeta(T_\psi^{-1}(s),\phi,\pi-\psi).
\end{aligned}\end{equation*}

\begin{lemma}
\label{lem:16}
$(0, G_0, H_0)$ satisfies \eqref{eq:12}, \eqref{eq:13}, and \eqref{eq:14}.
\end{lemma}

\begin{proof}
Since $\supp \zeta$ is a small neighborhood of the origin $(0,0,0)$, we see $\zeta(1,\phi,\psi)=\zeta(s,\pi,\psi)=\zeta(s,\phi,\pi)=0$.
Combining them with the formulas of Lemma \ref{lem:13} below, we can check that $(0, G_0, H_0)$ satisfies \eqref{eq:12}, \eqref{eq:13}, and \eqref{eq:14} directly as follows:
\begin{equation*}\begin{aligned}
G_0(0,\phi,\psi)
&=
\zeta(0,\phi,\psi)
-
\zeta(0,\pi-\phi,\psi) \\
&\quad
-
\zeta(0,\pi-\phi,\pi-\psi)
+
\zeta(0,\phi,\pi-\psi), \\
H_0(0,\phi,\psi)
&=
\zeta(T_\psi^{-1}(0),\phi,\psi)
-
\zeta(T_\psi(0),\pi-\phi,\psi) \\
&\quad
-
\zeta(T_\psi(0),\pi-\phi,\pi-\psi)
+
\zeta(T_\psi^{-1}(0),\phi,\pi-\psi) \\
&=
0,
\end{aligned}\end{equation*}
\begin{equation*}\begin{aligned}
G_0(1,\phi,\psi)
&=
\zeta(1,\phi,\psi)
-
\zeta(1,\pi-\phi,\psi) \\
&\quad
-
\zeta(1,\pi-\phi,\pi-\psi)
+
\zeta(1,\phi,\pi-\psi) \\
&=
0, \\
H_0(1,\phi,\psi)
&=
\zeta(T_\psi^{-1}(1),\phi,\psi)
-
\zeta(T_\psi(1),\pi-\phi,\psi) \\
&\quad
-
\zeta(T_\psi(1),\pi-\phi,\pi-\psi)
+
\zeta(T_\psi^{-1}(1),\phi,\pi-\psi) \\
&=
\zeta(0,\phi,\psi)
-
\zeta(0,\pi-\phi,\psi) \\
&\quad
-
\zeta(0,\pi-\phi,\pi-\psi)
+
\zeta(0,\phi,\pi-\psi),
\end{aligned}\end{equation*}
\begin{equation*}\begin{aligned}
G_0(T_\psi(s),0,\psi)
&=
\zeta(T_\psi(s),0,\psi)
-
\zeta(T_\psi(s),\pi-0,\psi) \\
&\quad
-
\zeta(T_\psi(s),\pi-0,\pi-\psi)
+
\zeta(T_\psi(s),0,\pi-\psi) \\
&=
\zeta(T_\psi(s),0,\psi)
+
\zeta(T_\psi(s),0,\pi-\psi), \\
H_0(T_\psi(s),0,\psi)
&=
\zeta(T_\psi^{-1}(T_\psi(s)),0,\psi)
-
\zeta(T_\psi(T_\psi(s)),\pi-0,\psi) \\
&\quad
-
\zeta(T_\psi(T_\psi(s)),\pi-0,\pi-\psi)
+
\zeta(T_\psi^{-1}(T_\psi(s)),0,\pi-\psi) \\
&=
\zeta(s,0,\psi)
+
\zeta(s,0,\pi-\psi),
\end{aligned}\end{equation*}
\begin{equation*}\begin{aligned}
G_0(s,\pi,\psi)
&=
\zeta(s,\pi,\psi)
-
\zeta(s,\pi-\pi,\psi) \\
&\quad
-
\zeta(s,\pi-\pi,\pi-\psi)
+
\zeta(s,\pi,\pi-\psi) \\
&=
-
\zeta(s,0,\psi)
-
\zeta(s,0,\pi-\psi), \\
H_0(s,\pi,\psi)
&=
\zeta(T_\psi^{-1}(s),\pi,\psi)
-
\zeta(T_\psi(s),\pi-\pi,\psi) \\
&\quad
-
\zeta(T_\psi(s),\pi-\pi,\pi-\psi)
+
\zeta(T_\psi^{-1}(s),\pi,\pi-\psi) \\
&=
-
\zeta(T_\psi(s),0,\psi)
-
\zeta(T_\psi(s),0,\pi-\psi),
\end{aligned}\end{equation*}
\begin{equation*}\begin{aligned}
G_0(s,\pi-\phi,0)
&=
\zeta(s,\pi-\phi,0)
-
\zeta(s,\pi-(\pi-\phi),0) \\
&\quad
-
\zeta(s,\pi-(\pi-\phi),\pi-0)
+
\zeta(s,\pi-\phi,\pi-0) \\
&=
\zeta(s,\pi-\phi,0)
-
\zeta(s,\phi,0), \\
H_0(s,\pi-\phi,0)
&=
\zeta(T_0^{-1}(s),\pi-\phi,0)
-
\zeta(T_0(s),\pi-(\pi-\phi),0) \\
&\quad
-
\zeta(T_0(s),\pi-(\pi-\phi),\pi-0)
+
\zeta(T_0^{-1}(s),\pi-\phi,\pi-0) \\
&=
\zeta(T_0^{-1}(s),\pi-\phi,0)
-
\zeta(T_0(s),\phi,0),
\end{aligned}\end{equation*}
\begin{equation*}\begin{aligned}
G_0(s,\phi,\pi)
&=
\zeta(s,\phi,\pi)
-
\zeta(s,\pi-\phi,\pi) \\
&\quad
-
\zeta(s,\pi-\phi,\pi-\pi)
+
\zeta(s,\phi,\pi-\pi) \\
&=
-
\zeta(s,\pi-\phi,0)
+
\zeta(s,\phi,0), \\
H_0(s,\phi,\pi)
&=
\zeta(T_\pi^{-1}(s),\phi,\pi)
-
\zeta(T_\pi(s),\pi-\phi,\pi) \\
&\quad
-
\zeta(T_\pi(s),\pi-\phi,\pi-\pi)
+
\zeta(T_\pi^{-1}(s),\phi,\pi-\pi) \\
&=
-
\zeta(T_\pi(s),\pi-\phi,0)
+
\zeta(T_\pi^{-1}(s),\phi,0).
\end{aligned}\end{equation*}
\end{proof}

Fix $\epsilon>0$.
We also introduce cut-off functions
$\zeta_1(s,\phi,\psi)$,
$\zeta_2(s,\phi,\psi)$, and
$\zeta_3(s,\phi,\psi)$ such that
\begin{equation*}\begin{aligned}
& \zeta_1=1 \text{ on } [\epsilon,1-\epsilon] \times \{0\} \times \{0\}, \\
& \zeta_2=1 \text{ on } \{0\} \times [\epsilon, \pi-\epsilon] \times \{0\}, \\ 
& \zeta_3=1 \text{ on } \{0\} \times \{0\} \times [\epsilon, \pi-\epsilon],
\end{aligned}\end{equation*}
each $\supp \zeta_i$ $(i=1,2,3)$ is contained in $\epsilon/2$-neighborhood
 of $[\epsilon,1-\epsilon] \times \{0\} \times \{0\}$, $\{0\} \times [\epsilon, \pi-\epsilon] \times \{0\}$, and 
$\{0\} \times \{0\} \times [\epsilon, \pi-\epsilon]$ respectively,
and $0\leq \zeta_i \leq 1$ $(i=1,2,3)$.
We put
\begin{equation*}\begin{aligned}
\begin{pmatrix}
G_1 \\ H_1 
\end{pmatrix}
(s,\phi,\psi)
&:=
\begin{pmatrix}
\alpha_1 \\ \beta_1 
\end{pmatrix}
\zeta_1(s,\phi,\psi)
+
\begin{pmatrix}
- \beta_1 
\\
- \alpha_1
\end{pmatrix}
\zeta_1(T_\psi(s),\pi-\phi,\psi)
\\&\quad
+
\begin{pmatrix}
\beta_1 
\\
\alpha_1
\end{pmatrix}
\zeta_1(T_\psi^{-1}(s),\phi,\pi-\psi)
+
\begin{pmatrix}
-\alpha_1 
\\ 
-\beta_1 
\end{pmatrix}
\zeta_1(s,\pi-\phi,\pi-\psi), \\
\begin{pmatrix}
G_2 \\ H_2 
\end{pmatrix}
(s,\phi,\psi)
&:=
\begin{pmatrix}
\alpha_2 \\ \beta_2 
\end{pmatrix}
\zeta_2(s,\phi,\psi)
+
\begin{pmatrix}
-\alpha_2 
\\ 
-\beta_2 
\end{pmatrix}
\zeta_2(s,\pi-\phi,\pi-\psi)
\\&\quad
+
\begin{pmatrix}
-\beta_2 
\\
\alpha_2
\end{pmatrix}
\zeta_2(T_\psi^{-1}(s),\phi,\psi)
+
\begin{pmatrix}
\beta_2 
\\
-\alpha_2
\end{pmatrix}
\zeta_2(T_\psi(s),\pi-\phi,\pi-\psi), \\
\begin{pmatrix}
G_3 \\ H_3 
\end{pmatrix}
(s,\phi,\psi)
&:=
\begin{pmatrix}
\alpha_3 \\ \beta_3 
\end{pmatrix}
\zeta_3(s,\phi,\psi)
+
\begin{pmatrix}
-\beta_3 
\\
-\alpha_3
\end{pmatrix}
\zeta_3(T_\psi(s),\pi-\phi,\psi)
\\&\quad
+
\begin{pmatrix}
-\alpha_3
\\
\beta_3 
\end{pmatrix}
\zeta_3(s,\pi-\phi,\psi)
+
\begin{pmatrix}
-\beta_3 
\\
\alpha_3
\end{pmatrix}
\zeta_3(T_\psi^{-1}(s),\phi,\psi),
\end{aligned}\end{equation*}
where we choose the $\epsilon$ sufficiently small with $0< \epsilon<1/2$, and $\alpha_i$, $\beta_i$ $(i=1,2,3)$ are positive real numbers.
Then, we can assume that $\supp G_i \cap \supp G_j= \emptyset$ and $\supp H_i \cap \supp H_j= \emptyset$ for $i \not=j$.

\begin{lemma}
For each $i=1,2,3$, the vector field
$(0,G_i, H_i)$ on $D$ satisfies
\eqref{eq:12},
\eqref{eq:13}, and
\eqref{eq:14}.
\end{lemma}

\begin{proof}
By the definition of $\zeta_i$, we have
\begin{equation*}\begin{aligned}
&
\zeta_i(1,\cdot,\cdot)=
\zeta_i(\cdot,\pi,\cdot)=
\zeta_i(\cdot,\cdot,\pi)=0 \quad (i=1,2,3), \\
&
\zeta_1(0, \cdot, \cdot)=
\zeta_2(\cdot,0,\cdot)=
\zeta_3(\cdot,\cdot,0)=
0.
\end{aligned}\end{equation*}
Combining them with the formulas of Lemma \ref{lem:13} below, we can check that $(0, G_1, H_1)$ satisfies \eqref{eq:12}, \eqref{eq:13}, and \eqref{eq:14} directly as follows: 
\begin{equation*}\begin{aligned}
\begin{pmatrix}
G_1 \\ H_1 
\end{pmatrix}
(0,\phi,\psi)
&=
\begin{pmatrix}
\alpha_1 \\ \beta_1 
\end{pmatrix}
\zeta_1(0,\phi,\psi)
+
\begin{pmatrix}
-\beta_1 
\\
-\alpha_1
\end{pmatrix}
\zeta_1(T_\psi(0),\pi-\phi,\psi)
\\&\quad
+
\begin{pmatrix}
\beta_1 
\\
\alpha_1
\end{pmatrix}
\zeta_1(T_\psi^{-1}(0),\phi,\pi-\psi)
+
\begin{pmatrix}
-\alpha_1 
\\ 
-\beta_1 
\end{pmatrix}
\zeta_1(0,\pi-\phi,\pi-\psi) 
=
\begin{pmatrix}
0 \\ 0
\end{pmatrix}, \\
\begin{pmatrix}
G_1 \\ H_1 
\end{pmatrix}
(1,\phi,\psi)
&=
\begin{pmatrix}
\alpha_1 \\ \beta_1 
\end{pmatrix}
\zeta_1(1,\phi,\psi)
+
\begin{pmatrix}
-\beta_1 
\\
-\alpha_1
\end{pmatrix}
\zeta_1(T_\psi(1),\pi-\phi,\psi)
\\&\quad
+
\begin{pmatrix}
\beta_1 
\\
\alpha_1
\end{pmatrix}
\zeta_1(T_\psi^{-1}(1),\phi,\pi-\psi)
+
\begin{pmatrix}
-\alpha_1 
\\ 
-\beta_1 
\end{pmatrix}
\zeta_1(1,\pi-\phi,\pi-\psi) 
=
\begin{pmatrix}
0 \\ 0
\end{pmatrix}, \\
\begin{pmatrix}
G_1 \\ H_1 
\end{pmatrix}
(T_\psi(s),0,\psi)
&=
\begin{pmatrix}
\alpha_1 \\ \beta_1 
\end{pmatrix}
\zeta_1(T_\psi(s),0,\psi)
+
\begin{pmatrix}
-\beta_1 
\\
-\alpha_1
\end{pmatrix}
\zeta_1(T_\psi(T_\psi(s)),\pi,\psi)
\\&\quad
+
\begin{pmatrix}
\beta_1 
\\
\alpha_1
\end{pmatrix}
\zeta_1(T_\psi^{-1}(T_\psi(s)),0,\pi-\psi)
+
\begin{pmatrix}
-\alpha_1 
\\ 
-\beta_1 
\end{pmatrix}
\zeta_1(T_\psi(s),\pi,\pi-\psi) \\
&=
\begin{pmatrix}
\alpha_1 \\ \beta_1 
\end{pmatrix}
\zeta_1(T_\psi(s),0,\psi)
+
\begin{pmatrix}
\beta_1 
\\
\alpha_1
\end{pmatrix}
\zeta_1(s,0,\pi-\psi), \\
\end{aligned}\end{equation*}
\begin{equation*}\begin{aligned}
\begin{pmatrix}
G_1 \\ H_1 
\end{pmatrix}
(s,\pi,\psi)
&=
\begin{pmatrix}
\alpha_1 \\ \beta_1 
\end{pmatrix}
\zeta_1(s,\pi,\psi)
+
\begin{pmatrix}
-\beta_1 
\\
-\alpha_1
\end{pmatrix}
\zeta_1(T_\psi(s),0,\psi)
\\&\quad
+
\begin{pmatrix}
\beta_1 
\\
\alpha_1
\end{pmatrix}
\zeta_1(T_\psi^{-1}(s),\pi,\pi-\psi)
+
\begin{pmatrix}
-\alpha_1 
\\ 
-\beta_1 
\end{pmatrix}
\zeta_1(s,0,\pi-\psi) \\
&=
\begin{pmatrix}
-\beta_1 
\\
-\alpha_1
\end{pmatrix}
\zeta_1(T_\psi(s),0,\psi)
+
\begin{pmatrix}
-\alpha_1 
\\ 
-\beta_1 
\end{pmatrix}
\zeta_1(s,0,\pi-\psi), \\
\begin{pmatrix}
G_1 \\ H_1 
\end{pmatrix}
(s,\pi-\phi,0)
&=
\begin{pmatrix}
\alpha_1 \\ \beta_1 
\end{pmatrix}
\zeta_1(s,\pi-\phi,0)
+
\begin{pmatrix}
-\beta_1 
\\
-\alpha_1
\end{pmatrix}
\zeta_1(T_0(s),\phi,0)
\\&\quad
+
\begin{pmatrix}
\beta_1 
\\
\alpha_1
\end{pmatrix}
\zeta_1(T_0^{-1}(s),\pi-\phi,\pi)
+
\begin{pmatrix}
-\alpha_1 
\\ 
-\beta_1 
\end{pmatrix}
\zeta_1(s,\phi,\pi) \\
&=
\begin{pmatrix}
\alpha_1 \\ \beta_1 
\end{pmatrix}
\zeta_1(s,\pi-\phi,0)
+
\begin{pmatrix}
-\beta_1 
\\
-\alpha_1
\end{pmatrix}
\zeta_1(T_0(s),\phi,0), \\
\begin{pmatrix}
G_1 \\ H_1 
\end{pmatrix}
(s,\phi,\pi)
&=
\begin{pmatrix}
\alpha_1 \\ \beta_1 
\end{pmatrix}
\zeta_1(s,\phi,\pi)
+
\begin{pmatrix}
-\beta_1 
\\
-\alpha_1
\end{pmatrix}
\zeta_1(T_\pi(s),\pi-\phi,\pi)
\\&\quad
+
\begin{pmatrix}
\beta_1 
\\
\alpha_1
\end{pmatrix}
\zeta_1(T_\pi^{-1}(s),\phi,0)
+
\begin{pmatrix}
-\alpha_1 
\\ 
-\beta_1 
\end{pmatrix}
\zeta_1(s,\pi-\phi,0) \\
&=
\begin{pmatrix}
\beta_1 
\\
\alpha_1
\end{pmatrix}
\zeta_1(T_\pi^{-1}(s),\phi,0)
+
\begin{pmatrix}
-\alpha_1 
\\ 
-\beta_1 
\end{pmatrix}
\zeta_1(s,\pi-\phi,0).
\end{aligned}\end{equation*}
Similarly, $(0, G_i, H_i)$ $(i=2,3)$ also satisfy \eqref{eq:12}, \eqref{eq:13}, and \eqref{eq:14}.
\end{proof}

Here, we show the lemma used in the above lemmas.
\begin{lemma}
\label{lem:13}
$T_0^{-1}(s)=T_\pi(s)$, $T_\pi^{-1}(s)=T_0(s)$.
\end{lemma}

\begin{proof}
Applying Lemma \ref{lem:4} to $K=K(\theta,\phi,\psi)$, we see that
\begin{equation}
\label{eq:15}
 \Theta(\pi-\Theta(\theta,\phi,\psi)+\theta,\phi,\psi) + \Theta(\theta,\phi,\psi)=\pi.
\end{equation}
Combining \eqref{eq:44} with the definition of $T_\psi$, we have
\begin{equation*}
\Gamma_\psi\left(
 \Theta(0,\pi,\psi)T_\psi(s)
\right)
=
\pi- \Theta(0,0,\psi)s.
\end{equation*}
By the definition of $\Gamma_\psi$, 
\begin{equation*}
\pi - \Theta(\Theta(0,\pi,\psi)T_\psi(s),0,\psi) + \Theta(0,\pi,\psi)T_\psi(s) = 
\pi- \Theta(0,0,\psi)s.
\end{equation*}
Put $\Theta_0:=\Theta(0,0,0)$ and $\Theta_1:=\Theta(0,\pi,0)$.
Then, by \eqref{eq:45} and \eqref{eq:44},
\begin{equation*}
\Theta_0=\Theta(0,\pi,\pi), \quad \Theta_1=\Theta(0,0,\pi)=\pi-\Theta_0. 
\end{equation*}
Thus, we obtain
\begin{align}
\label{eq:16}
\pi - \Theta(\Theta_1 T_0(s),0,0) + \Theta_1 T_0(s) 
&= 
\pi- \Theta_0 s, \\
\label{eq:17}
\pi - \Theta(\Theta_0 T_\pi(t),0,\pi) + \Theta_0 T_\pi(t)
&=
\pi- \Theta_1 t,
\end{align}
for $s,t \in [0,1]$, and such $T_0(s), T_\pi(t) \in [0,1]$ is uniquely determined,
since $\Gamma_\psi:[0, \Theta(0,\pi,\psi)] \rightarrow [\pi-\Theta(0,0,\psi),\pi]$ is bijective by Lemma \ref{lem:11}.

On substituting $(\theta,\phi,\psi)=(\Theta_0 T_\pi(t),0,\pi)$ into \eqref{eq:15}, we have
\begin{equation*}
\Theta\bigl(\pi-\Theta(\Theta_0 T_\pi(t),0,\pi)+\Theta_0 T_\pi(t),0,\pi\bigr) + \Theta(\Theta_0 T_\pi(t),0,\pi)=\pi.
\end{equation*}
By \eqref{eq:17} and Lemmas \ref{lem:6} and \ref{lem:9},
the first term of the left-hand side is
\begin{equation*}\begin{aligned}
\Theta\bigl(\pi-\Theta(\Theta_0 T_\pi(t),0,\pi)+\Theta_0 T_\pi(t),0,\pi\bigr)
&=
 \Theta(\pi-\Theta_1 t,0,\pi) \\
&=
\Theta(X(\pi) X(-\Theta_1 t) Z(\pi) K) \\
&=
\Theta(X(-\Theta_1 t) Z(\pi) K) \\
&=
\Theta(Z(\pi) X(\Theta_1 t) K) \\
&=
\pi- \Theta(X(\Theta_1 t) K) \\
&= \pi-\Theta(\Theta_1 t, 0,0).
\end{aligned}\end{equation*}
By \eqref{eq:17}, 
\begin{equation*}
\Theta(\Theta_0 T_\pi(t),0,\pi)=
\Theta_0 T_\pi(t) + \Theta_1 t.
\end{equation*}
Consequently, we get
\begin{equation*}
\pi-\Theta(\Theta_1 t, 0,0) + \Theta_1 t 
= \pi - \Theta_0 T_\pi(t).
\end{equation*}
By \eqref{eq:16} with $s=T_\pi(t)$, the uniqueness asserts that $t=T_0(T_\pi(t))$.
\end{proof}

\begin{proof}[Proof of Proposition \ref{prop:4}]
For
\begin{equation*}
\check{\mathscr{F}}=\frac{(\check{F}, \check{G}, \check{H})}{\sqrt{\check{F}^2+\check{G}^2+\check{H}^2}}: \partial D \to S^2, 
\end{equation*}
we consider the following conditions:
\begin{itemize}
 \item[(C1)] $\check{\mathscr{F}}$ is homotopic to $\mathscr{F}$.
 \item[(C2)] $\check{F}^2+\check{G}^2+\check{H}^2 \not= 0$ on $\partial D$.
 \item[(C3)] $(\check{F}, \check{G}, \check{H})$ satisfies the identities \eqref{eq:12}, \eqref{eq:13}, and \eqref{eq:14}.
 \item[(C4)] $\check{G}^2+\check{H}^2 \not= 0$ at the vertices of $\partial D$.
 \item[(C5)] $\check{G}^2+\check{H}^2 \not= 0$ on $\partial M_i$ $(i=1,\dots,6)$.
\end{itemize}
We shall construct $\check{\mathscr{F}}$ such that the conditions (C1)--(C5) hold and 
$(\pm 1, 0,0) \in S^2$ are regular values of $\check{\mathscr{F}}$.

First, we put
\begin{equation*}
\check{\mathscr{F}}_0:=\frac{(F, \check{G}_0, \check{H}_0)}{\sqrt{F^2+\check{G}_0^2+\check{H}_0^2}}
:=
\frac{(F, G+\alpha_0 G_0, H+\alpha_0 H_0)}
{\sqrt{F^2 + (G+\alpha_0 G_0)^2 + (H+\alpha_0 H_0)^2}}: \partial D \rightarrow S^2,
\end{equation*}
where
we take sufficiently small $\alpha_0>0$ if $(G^2+H^2)(0,0,0)=0$, take $\alpha_0=0$ if $(G^2+H^2)(0,0,0) \not= 0$.
Then 
\begin{equation}
\label{eq:26}
(\check{G}_0^2+\check{H}_0^2)(0,0,0)
=
(G(0,0,0)+\alpha_0)^2 + (H(0,0,0))^2 \not=0.
\end{equation}
By Lemma \ref{lem:16}, $(0, G_0, H_0)$ satisfies \eqref{eq:12}, \eqref{eq:13}, and \eqref{eq:14}.
By \eqref{eq:26} and Remark \ref{rem:1}, 
(C4) holds for $(F,\check{G}_0,\check{H}_0)$.

Since $F^2+G^2+H^2 \not=0$ on $\partial D$ by the assumption, 
for sufficiently small $\alpha_0$,
$F^2 + (G+\alpha_0 G_0)^2 + (H+\alpha_0 H_0)^2 \not=0$ on $\partial D$.
Thus, $\check{\mathscr{F}}_0$ satisfies the conditions (C1) and (C2).
By Lemma \ref{lem:16}, $\check{\mathscr{F}}_0$ satisfies the condition (C3).
Consequently, (C1)--(C4) hold for $\check{\mathscr{F}}_0$.

Second, by continuity, there exist $\delta_0>0$ and $\delta_1>0$ such that
$\sqrt{\check{G}_0^2 + \check{H}_0^2} \geq \delta_1$
on $\delta_0$-neighborhood of the vertices of $D$.
For $\alpha_i$ and $\beta_i$ above with $\sqrt{\alpha_i^2+\beta_i^2} \leq \delta_1/4$,
on $\delta_0$-neighborhood of the vertices, 
we see $\abs{(G_i,H_i)}\leq \sqrt{\alpha_i^2+\beta_i^2} \leq \delta_1/4$.
Thus
\begin{equation*}
\abs{(\check{G}_0 + G_1+G_2+G_3, \check{H}_0 +H_1+H_2+H_3)}
\geq 
\abs{(\check{G}_0, \check{H}_0)}
- \sum_{i=1}^{3} \abs{(G_i, H_i)}
\geq 
\delta_1/4 >0
\end{equation*}
on $\delta_0$-neighborhood of the vertices.

Thus, the points on $\partial M_i$
with $(\check{G}_0 + G_1+G_2+G_3, \check{H}_0 +H_1+H_2+H_3)=(0,0)$
are contained in the following 12 lines:
$[\delta_0/2,1-\delta_0/2] \times \{0,\pi\} \times \{0,\pi\}$,
$\{0,1\} \times [\delta_0/2,\pi-\delta_0/2] \times \{0,\pi\}$,
$\{0,1\} \times \{0,\pi\} \times [\delta_0/2,\pi-\delta_0/2]$.
Moreover, on the 12 lines, $(G_i, H_i)$ $(i=1,2,3)$ are constants.
Consider the image $U$ of $(G+\alpha_0 G_0, H+\alpha_0 H_0)$ on the 12 lines.
Since $U$ is the union of 12 curves of class $C^{\infty}$, 
the $2$-dimensional Lebesgue measure of $U$ is $0$.
Therefore, there exists $\alpha_i, \beta_i$ such that
$|\alpha_i|$ and $|\beta_i|$ are small, and
\begin{equation*}
(\pm \alpha_i, \pm \beta_i) \cap U = \emptyset, \
(\pm \beta_i, \pm \alpha_i) \cap U = \emptyset
\end{equation*}
for $i=1,2,3$.
In this setting, on the 12 lines, we obtain
\begin{equation*}
 (\check{G}_0 + G_1+G_2+G_3, \check{H}_0 + H_1+H_2+H_3) \not=(0,0).
\end{equation*}
Thus, we put 
$\check{G}_1:=\check{G}_0 + G_1+G_2+G_3$,
$\check{H}_1:=\check{H}_0 + H_1+H_2+H_3$,
and
\begin{equation*}
\check{\mathscr{F}}_1:=\frac{(F, \check{G}_1, \check{H}_1)}{\sqrt{F^2+\check{G}_1^2+\check{H}_1^2}}
: \partial D \rightarrow S^2
\end{equation*}
for sufficiently small $|\alpha_i|$ and $|\beta_i|$. 
Then $\check{\mathscr{F}}_1$ satisfies (C1)--(C5).

Next, we will construct a deformation of $\check{\mathscr{F}}_1$ which satisfies that $(\pm 1,0,0)$ are its regular values on $M_1 \cup M_2$.

Since $(\check{G}_1, \check{H}_1) \not= (0,0)$ on $\partial M_1$, 
there exists $\delta_2 >0$ such that 
\begin{equation*}
\sqrt{\check{G}_1^2 + \check{H}_1^2} \geq \delta_2 \text{ on } V,
\end{equation*}
where $V \subset M_1$ is some small neighborhood of $\partial M_1$.
By Sard's theorem, there exist positive real numbers  $\alpha_4, \beta_4$ 
with $\sqrt{\alpha_4^2+\beta_4^2} \leq \delta_2/2$
such that
$(\alpha_4, \beta_4)$ is a regular value of $(\check{G}_1, \check{H}_1)$ on $M_1$, and
$(-\beta_4, \alpha_4)$ is a regular value of $(\check{G}_1, \check{H}_1)$ on $M_2$.
We note that, since $|\alpha_4|$ and $|\beta_4|$ are small, $F \not=0$ at the regular points
$(\check{G}_1, \check{H}_1)^{-1}(\alpha_4, \beta_4)$ on $M_1$ and 
$(\check{G}_1, \check{H}_1)^{-1}(-\beta_4, \alpha_4)$ on $M_2$.
Let $\zeta_4(\phi, \psi) \in C_0^\infty(\interior M_1, [0,1])$ be a cut-off function such that
$\zeta_4 =1$ on $M_1 \setminus V$.
We put
\begin{equation*}
 (G_4, H_4)(s, \phi, \psi):=
\begin{cases}
(\alpha_4 \zeta_4(\phi, \psi), \beta_4 \zeta_4(\phi,\psi)) & \text{ if } (s,\phi,\psi) \in M_1, \\
(-\beta_4 \zeta_4(\phi, \psi), \alpha_4 \zeta_4(\phi,\psi)) & \text{ if } (s,\phi,\psi) \in M_2, \\
(0,0) & \text{ if } (s,\phi,\psi) \in M_3 \cup \dots \cup M_6.
\end{cases}
\end{equation*}
Then $(0,G_4, H_4)$ satisfies \eqref{eq:12}, \eqref{eq:13}, and \eqref{eq:14}.
Thus, we define
$\check{G}_2:=\check{G}_1 - G_4$,
$\check{H}_2:=\check{H}_1 - H_4$, and
\begin{equation*}
\check{\mathscr{F}}_2:=\frac{(F, \check{G}_2, \check{H}_2)}{\sqrt{F^2+\check{G}_2^2+\check{H}_2^2}}
: \partial D \rightarrow S^2.
\end{equation*}
Then, $\check{\mathscr{F}}_2$ satisfies (C1)--(C5).
On $V (\subset M_1)$, we obtain
\begin{equation*}
 \abs{(\check{G}_2, \check{H}_2)}
\geq
\abs{(\check{G}_1, \check{H}_1)}
-
\abs{(G_4, H_4)} \geq \frac{\delta_2}{2}>0.
\end{equation*}
So, there exist no zeros of $(\check{G}_2, \check{H}_2)$ on $V$.
On $M_1 \setminus V$, we have
\begin{equation*}
(\check{G}_2, \check{H}_2)=(0,0) \text{ if and only if } (\check{G}_1, \check{H}_1)=(\alpha_4, \beta_4).
\end{equation*}
Thus, by the definition of $(\alpha_4, \beta_4)$, $(0,0)$ is a regular value of $(\check{G}_2, \check{H}_2)$ on $M_1 \setminus V$, and $F\not=0$ at the regular values.
Therefore, $(\pm 1, 0,0)$ are regular values of $\check{\mathscr{F}}_2$ on $M_1$.
Similarly, we obtain that $(\pm 1, 0,0)$ are regular values of $\check{\mathscr{F}}_2$ on $M_2$.

Moreover, we can modify $\check{\mathscr{F}}_2$ similarly with respect to $M_3, \dots, M_6$.
Then the modified $\check{\mathscr{F}}$ satisfies the conditions (i)--(iv) in Proposition \ref{prop:4}.
\end{proof}

\section{The case with a hyperplane symmetry}
\label{sec:4}

Although, for a general $K \in \hat{\K}$, it is very hard to find a suitable linear transformation $\mathcal A$ such that ${\mathcal A}K$ satisfies the condition \eqref{eq:3}, it is not so difficult to do so if $K$ has additional symmetries.
In this appendix, let us consider the case where $K \in \hatK$ is symmetric with respect to a plane.
We start with the following simple observation.

\begin{example}
Consider the case that $K$ is symmetric with respect to the $yz$-plane,
$zx$-plane, and $xy$-plane
(more generally, its image $\mathcal{A}K$ by a linear transformation $\mathcal{A}$
is symmetric with respect to these three planes).
In this case, the condition \eqref{eq:3} automatically holds, and hence we have
$\mathcal{P}(K) \geq 32/3$.
This fact is the three dimensional case of the result by Saint-Raymond \cite{SR}.
\end{example}

From now on, suppose that $K \in \hat{\K}$ is symmetric with respect to a plane.
Since the volume product is invariant under linear transformations of $K$,
we may assume that the hyperplane of symmetry is the $yz$-plane.
In this case, the area of $O*\beta$ and that of $O*(-b)$ coincide.
Taking the signs into consideration, $\overline{\beta}=\overline{b}$ holds.
Similarly, $\overline{\gamma}=\overline{c}$ holds.
Concerning the volume of the eight parts of $K$, we have
\begin{equation*}
|K_{+++}|=
|K_{-++}|=
|K_{+--}|, \quad
|K_{+-+}|=
|K_{--+}|=
|K_{++-}|.
\end{equation*}
Consequently, the condition \eqref{eq:3} holds, provided $K$ satisfies
\begin{equation}
\label{eq:4}
\overline{\alpha}=\overline{a}, \quad
|K_{+++}|=
|K_{+-+}|.
\end{equation}
For a given convex body $K \in \hat{\K}$,
we consider the image of $K$ by the following linear transformation:
\begin{equation*}
A(\varphi) X(\theta) 
\begin{pmatrix}
x \\ y \\ z
\end{pmatrix}
=
\begin{pmatrix}
1 & 0 & 0 \\
0 & 1 & \varphi \\
0 & 0 & 1 
\end{pmatrix}
\begin{pmatrix}
 1 & 0 & 0 \\
0 & \cos \theta & -\sin \theta \\
0 & \sin \theta & \cos \theta
\end{pmatrix}
\begin{pmatrix}
x \\ y \\ z
\end{pmatrix}.
\end{equation*}
For $K \in \hat{\K}$, we introduced vectors
$\overline{\alpha}=(\alpha_1,\alpha_2,\alpha_3)$, etc., in Section \ref{sec:3.3}.
From now on, to clarify the dependence on the convex body, 
we denote the vector $\overline{\alpha}$ for $A(\varphi)X(\theta)K \in \hat{\K}$ by $\overline{\alpha}(A(\varphi)X(\theta)K)$, etc.
Similarly, for the image $A(\varphi)X(\theta)K \in \hat{\K}$ of a convex body $K$, we denote by $(A(\varphi)X(\theta)K)_{+++},\dots$ the decomposition in Section \ref{sec:3.2}.

First, since $A(\varphi)$ and $X(\theta)$ keep the symmetry of $K$ with respect to the $yz$-plane, we have
\begin{equation*}\begin{aligned}
&\overline{\beta}(A(\varphi)X(\theta)K)
= \overline{b}(A(\varphi)X(\theta)K), \\
&\overline{\gamma}(A(\varphi)X(\theta)K)
= \overline{c}(A(\varphi)X(\theta)K), \\
&|(A(\varphi)X(\theta)K)_{+++}|=|(A(\varphi)X(\theta)K)_{-++}|=|(A(\varphi)X(\theta)K)_{+--}|, \\
&|(A(\varphi)X(\theta)K)_{+-+}|=|(A(\varphi)X(\theta)K)_{--+}|=|(A(\varphi)X(\theta)K)_{++-}|
\end{aligned}\end{equation*}
for any $\varphi, \theta$ $\in \mathbb R$.
Next, for a given $\theta \in \mathbb R$
we can choose $\varphi=\varphi(\theta)$ $\in \R$ to satisfy
\begin{equation*}
|(A(\varphi)X(\theta)K)_{+++}|=
|(A(\varphi)X(\theta)K)_{+-+}|,
\end{equation*}
see Proposition\,\ref{prop:3} in the special case that $\Phi(A(\varphi)X(\theta)K)=\Psi(A(\varphi)X(\theta)K)=\pi/2$,
$|(A(\varphi)X(\theta)K)_{+++}|=|(A(\varphi)X(\theta)K)_{-++}|$,
$|(A(\varphi)X(\theta)K)_{--+}|=|(A(\varphi)X(\theta)K)_{+-+}|$ for $X(\theta)K$.
Note that such $\varphi$ is uniquely determined for each $\theta$,
and $\varphi$ depends on $\theta$ continuously.
Moreover, we can easily see that there exists some $\theta_0$ such that $\varphi(\theta_0)=0$.
In this setting, we obtain $\varphi(\theta_0+\pi/2)=0$ and 
\begin{equation}
\label{eq:20}
\begin{aligned}
\overline{\alpha}(A(\varphi(\theta_0+\pi/2))X(\theta_0+\pi/2)K)
&=
\overline{\alpha}(X(\theta_0+\pi/2)K)
=
\overline{a}(A(\varphi(\theta_0))X(\theta_0)K)
, \\
\overline{a}(A(\varphi(\theta_0+\pi/2))X(\theta_0+\pi/2)K)
&=
\overline{a}(X(\theta_0+\pi/2)K)
=
\overline{\alpha}(A(\varphi(\theta_0))X(\theta_0)K).
\end{aligned}
\end{equation}

Next, we put
\begin{equation*}
F(\theta):= 
\overline{\alpha}
(A(\varphi(\theta))X(\theta)K)
-
\overline{a}
(A(\varphi(\theta))X(\theta)K).
\end{equation*}
Since $\varphi$ is continuous with respect to $\theta$, $F$ is continuous on $\R$.
By \eqref{eq:20}, we have $F(\theta_0+\pi/2)=-F(\theta_0)$.
By using the intermediate value theorem, there exists
$\theta_1 \in [\theta_0, \theta_0+\pi/2]$
such that
$F(\theta_1)=0$,
which yields
\begin{equation*}
\overline{\alpha}
(A(\varphi(\theta_1))X(\theta_1)K)
=
\overline{a}
(A(\varphi(\theta_1))X(\theta_1)K).
\end{equation*}
Consequently, $A(\varphi(\theta_1))X(\theta_1)K \in \hatK$ satisfies the condition \eqref{eq:4}, and hence \eqref{eq:3}.
Since $A(\varphi(\theta_1))X(\theta_1)$ is a linear transformation of $\R^3$, we obtain
\begin{equation*}
\mathcal{P}(K) =
\mathcal{P}(A(\varphi(\theta_1))X(\theta_1)K) \geq \frac{32}{3}.
\end{equation*}
Thus, we have proved the following
\begin{proposition}
\label{prop:15}
Let $K \in \hatK$ be a three dimensional centrally symmetric convex body 
which is symmetric with respect to a plane.
Then $\mathcal{P}(K) \geq 32/3$ holds.
\end{proposition}

\section*{Acknowledgements}

The authors would like to thank the anonymous referees for their careful reading and valuable comments, especially, for pointing out references \cite{Me2} and \cite{BMMR}, for pointing out a simplification of the proof of Proposition 3.2, and for their suggestions about the notations of the eight parts decomposition of convex bodies $K, K^\circ$ and the curves on their boundaries.
Including them their valuable feedback has resulted in considerable improvements of exposition of the paper.


\begin{bibdiv}
\begin{biblist}
\bib{AKO}{article}{
   label={AKO},
   author={Artstein-Avidan, Shiri},
   author={Karasev, Roman},
   author={Ostrover, Yaron},
   title={From symplectic measurements to the Mahler conjecture},
   journal={Duke Math. J.},
   volume={163},
   date={2014},
   number={11},
   pages={2003--2022},
%   issn={0012-7094},
%   review={\MR{3263026}},
   doi={10.1215/00127094-2794999},
}
\bib{BF}{article}{
   author={Barthe, F.},
   author={Fradelizi, M.},
   title={The volume product of convex bodies with many hyperplane
   symmetries},
   journal={Amer. J. Math.},
   volume={135},
   date={2013},
   number={2},
   pages={311--347},
%   issn={0002-9327},
%   review={\MR{3038713}},
   doi={10.1353/ajm.2013.0018},
}
\bib{BM}{article}{
   author={Bourgain, J.},
   author={Milman, V. D.},
   title={New volume ratio properties for convex symmetric bodies in ${\bf
   R}^n$},
   journal={Invent. Math.},
   volume={88},
   date={1987},
   number={2},
   pages={319--340},
%   issn={0020-9910},
%   review={\MR{880954}},
   doi={10.1007/BF01388911},
}
\bib{BMMR}{article}{
   author={B\"{o}r\"{o}czky, K. J.},
   author={Makai, E., Jr.},
   author={Meyer, M.},
   author={Reisner, S.},
   title={On the volume product of planar polar convex bodies---lower
   estimates with stability},
   journal={Studia Sci. Math. Hungar.},
   volume={50},
   date={2013},
   number={2},
   pages={159--198},
%   issn={0081-6906},
%   review={\MR{3187810}},
   doi={10.1556/SScMath.50.2013.2.1235},
}
\bib{FMZ}{article}{
   author={Fradelizi, Matthieu},
   author={Meyer, Mathieu},
   author={Zvavitch, Artem},
   title={An application of shadow systems to Mahler's conjecture},
   journal={Discrete Comput. Geom.},
   volume={48},
   date={2012},
   number={3},
   pages={721--734},
%   issn={0179-5376},
%   review={\MR{2957641}},
   doi={10.1007/s00454-012-9435-3},
}\bib{GMR}{article}{
   author={Gordon, Y.},
   author={Meyer, M.},
   author={Reisner, S.},
   title={Zonoids with minimal volume-product---a new proof},
   journal={Proc. Amer. Math. Soc.},
   volume={104},
   date={1988},
   number={1},
   pages={273--276},
%   issn={0002-9939},
%   review={\MR{958082}},
   doi={10.2307/2047501},
}
\bib{HZ}{book}{
   author={Hofer, Helmut},
   author={Zehnder, Eduard},
   title={Symplectic invariants and Hamiltonian dynamics},
   series={Birkh\"auser Advanced Texts: Basler Lehrb\"ucher. [Birkh\"auser
   Advanced Texts: Basel Textbooks]},
   publisher={Birkh\"auser Verlag, Basel},
   date={1994},
   pages={xiv+341},
%   isbn={3-7643-5066-0},
%   review={\MR{1306732}},
   doi={10.1007/978-3-0348-8540-9},
}
\bib{K}{article}{
   author={Kuperberg, Greg},
   title={From the Mahler conjecture to Gauss linking integrals},
   journal={Geom. Funct. Anal.},
   volume={18},
   date={2008},
   number={3},
   pages={870--892},
%   issn={1016-443X},
%   review={\MR{2438998}},
   doi={10.1007/s00039-008-0669-4},
}
\bib{KR}{article}{
   author={Kim, Jaegil},
   author={Reisner, Shlomo},
   title={Local minimality of the volume-product at the simplex},
   journal={Mathematika},
   volume={57},
   date={2011},
   number={1},
   pages={121--134},
%   issn={0025-5793},
%   review={\MR{2764160}},
   doi={10.1112/S0025579310001555},
}\bib{LR}{article}{
   author={Lopez, M. A.},
   author={Reisner, S.},
   title={A special case of Mahler's conjecture},
   journal={Discrete Comput. Geom.},
   volume={20},
   date={1998},
   number={2},
   pages={163--177},
%   issn={0179-5376},
%   review={\MR{1637864}},
   doi={10.1007/PL00000076},
}
\bib{Ma0}{article}{
   label={Ma1},
   author={Mahler, Kurt},
   title={Ein Minimalproblem f\"ur konvexe Polygone},
   language={German},
   journal={Mathematica (Zutphen)},
   volume={B 7},
   date={1939},
   pages={118--127},
}
\bib{Ma}{article}{
   label={Ma2},
   author={Mahler, Kurt},
   title={Ein \"Ubertragungsprinzip f\"ur konvexe K\"orper},
   language={German},
   journal={\v Casopis P\v est. Mat. Fys.},
   volume={68},
   date={1939},
   pages={93--102},
%   issn={0528-2195},
%   review={\MR{0001242}},
}
\bib{Me}{article}{
   label={Me},
   author={Meyer, Mathieu},
   title={Convex bodies with minimal volume product in ${\bf R}^2$},
   journal={Monatsh. Math.},
   volume={112},
   date={1991},
   number={4},
   pages={297--301},
%   issn={0026-9255},
%   review={\MR{1141097}},
   doi={10.1007/BF01351770},
}
\bib{Me2}{article}{
   label={Me2},
   author={Meyer, Mathieu},
   title={Une caract\'{e}risation volumique de certains espaces norm\'{e}s de
   dimension finie},
   language={French, with English summary},
   journal={Israel J. Math.},
   volume={55},
   date={1986},
   number={3},
   pages={317--326},
%   issn={0021-2172},
%   review={\MR{876398}},
   doi={10.1007/BF02765029},
}
\bib{Mi}{article}{
   label={Mi},
   author={Minkowski, Hermann},
   title={Volumen und Oberfl\"{a}che},
   language={German},
   journal={Math. Ann.},
   volume={57},
   date={1903},
   number={4},
   pages={447--495},
%   issn={0025-5831},
%   review={\MR{1511220}},
   doi={10.1007/BF01445180},
}
\bib{MP}{article}{
   author={Meyer, Mathieu},
   author={Pajor, Alain},
   title={On the Blaschke-Santal\'o inequality},
   journal={Arch. Math. (Basel)},
   volume={55},
   date={1990},
   number={1},
   pages={82--93},
%   issn={0003-889X},
%   review={\MR{1059519}},
   doi={10.1007/BF01199119},
}
\bib{MR}{article}{
   author={Meyer, Mathieu},
   author={Reisner, Shlomo},
   title={On the volume product of polygons},
   journal={Abh. Math. Semin. Univ. Hambg.},
   volume={81},
   date={2011},
   number={1},
   pages={93--100},
%   issn={0025-5858},
%   review={\MR{2812036}},
   doi={10.1007/s12188-011-0054-3},
}
\bib{NPRZ}{article}{
   author={Nazarov, Fedor},
   author={Petrov, Fedor},
   author={Ryabogin, Dmitry},
   author={Zvavitch, Artem},
   title={A remark on the Mahler conjecture: local minimality of the unit
   cube},
   journal={Duke Math. J.},
   volume={154},
   date={2010},
   number={3},
   pages={419--430},
%   issn={0012-7094},
%   review={\MR{2730574}},
   doi={10.1215/00127094-2010-042},
}
\bib{OR}{book}{
   author={Outerelo, Enrique},
   author={Ruiz, Jes{\'u}s M.},
   title={Mapping degree theory},
   series={Graduate Studies in Mathematics},
   volume={108},
   publisher={American Mathematical Society, Providence, RI; Real Sociedad
   Matem\'atica Espa\~nola, Madrid},
   date={2009},
   pages={x+244},
%   isbn={978-0-8218-4915-6},
%   review={\MR{2566906}},
   doi={10.1090/gsm/108},
}
\bib{R}{article}{
   author={Reisner, Shlomo},
   title={Random polytopes and the volume-product of symmetric convex
   bodies},
   journal={Math. Scand.},
   volume={57},
   date={1985},
   number={2},
   pages={386--392},
%   issn={0025-5521},
%   review={\MR{832364}},
   doi={10.7146/math.scand.a-12124},
}
\bib{R2}{article}{
   author={Reisner, Shlomo},
   title={Zonoids with minimal volume-product},
   journal={Math. Z.},
   volume={192},
   date={1986},
   number={3},
   pages={339--346},
%   issn={0025-5874},
%   review={\MR{845207}},
   doi={10.1007/BF01164009},
}
\bib{RSW}{article}{
   author={Reisner, Shlomo},
   author={Sch\"utt, Carsten},
   author={Werner, Elisabeth M.},
   title={Mahler's conjecture and curvature},
   journal={Int. Math. Res. Not. IMRN},
   date={2012},
   number={1},
   pages={1--16},
%   issn={1073-7928},
%   review={\MR{2874925}},
   doi={10.1093/imrn/rnr003},
}
\bib{Sa}{article}{
   label={Sa},
   author={Santal{\'o}, L. A.},
   title={An affine invariant for convex bodies of $n$-dimensional space},
   language={Spanish},
   journal={Portugaliae Math.},
   volume={8},
   date={1949},
   pages={155--161},
%   review={\MR{0039293}},
}
\bib{Sc}{book}{
   label={Sc1},
   author={Schneider, Rolf},
   title={Convex bodies: the Brunn-Minkowski theory},
   series={Encyclopedia of Mathematics and its Applications},
   volume={151},
   edition={Second expanded edition},
   publisher={Cambridge University Press, Cambridge},
   date={2014},
   pages={xxii+736},
%   isbn={978-1-107-60101-7},
%   review={\MR{3155183}},
}
\bib{Sc2}{article}{
   label={Sc2},
   author={Schneider, Rolf},
   title={Smooth approximation of convex bodies},
   journal={Rend. Circ. Mat. Palermo (2)},
   volume={33},
   date={1984},
   number={3},
   pages={436--440},
%   issn={0009-725X},
%   review={\MR{779946}},
   doi={10.1007/BF02844505},
}
\bib{SR}{article}{
   author={Saint-Raymond, J.},
   title={Sur le volume des corps convexes sym\'etriques},
   language={French},
   conference={
      title={Initiation Seminar on Analysis: G. Choquet-M. Rogalski-J.
      Saint-Raymond, 20th Year: 1980/1981},
   },
   book={
      series={Publ. Math. Univ. Pierre et Marie Curie},
      volume={46},
      publisher={Univ. Paris VI, Paris},
   },
   date={1981},
   pages={Exp. No. 11, 25},
%   review={\MR{670798}},
}
\bib{V}{article}{
   author={Viterbo, Claude},
   title={Metric and isoperimetric problems in symplectic geometry},
   journal={J. Amer. Math. Soc.},
   volume={13},
   date={2000},
   number={2},
   pages={411--431 (electronic)},
%   issn={0894-0347},
%   review={\MR{1750956}},
   doi={10.1090/S0894-0347-00-00328-3},
}
\end{biblist}
\end{bibdiv}
\end{document}